\documentclass{article}

\usepackage{layout} 
\usepackage[top=2cm, bottom=2cm, left=1.5cm, right=1.5cm]{geometry} 
\usepackage[english]{babel}
\usepackage[utf8]{inputenc}
\usepackage[T1]{fontenc}
\usepackage{amsmath}
\usepackage{amssymb}
\usepackage{mathrsfs}

\usepackage{amsthm}
\newtheorem{theorem}{Theorem}[section]
\newtheorem{proposition}[theorem]{Proposition}
\newtheorem{corollary}[theorem]{Corollary}
\newtheorem{lemma}[theorem]{Lemma}
\newtheorem{remark}[theorem]{Remark}
\newtheorem{defi}[theorem]{Definition}
\newtheorem{claim}[theorem]{Claim}

\newtheorem{op}[theorem]{Open Questions}

\usepackage[colorlinks=true,urlcolor=blue,linkcolor=blue]{hyperref}
\usepackage{cleveref}

\usepackage{enumerate} 
\usepackage{color, colortbl} 
\usepackage{soul} 
\usepackage{tcolorbox}
\usepackage[compact]{titlesec}
\usepackage{graphicx} 
\usepackage[tikz]{bclogo}
\usepackage{enumitem}
\usepackage{comment}
\usepackage{mathtools}
\usepackage{esint}
\usepackage{faktor}

\newcommand{\scal}[2]{\left\langle #1,#2 \right\rangle}
\newcommand{\g}{\nabla}

\newcommand{\di}{\mathrm{div}}
\newcommand{\lap}{\Delta}
\newcommand{\dr}{\partial}
\newcommand{\vol}{\mathrm{vol}}

\newcommand{\diam}{\mathrm{diam}}

\newcommand{\dist}{\mathrm{dist}}
\newcommand{\Span}{\mathrm{Span}}

\newcommand{\tr}{\mathrm{tr}}
\newcommand{\II}{\mathrm{I\!I}}
\newcommand{\IIr}{\mathring{\II}}
\newcommand{\Riem}{\mathrm{Riem}}
\newcommand{\Ric}{\mathrm{Ric}}

\newcommand{\Conf}{\mathrm{Conf}}

\newcommand{\proj}{\mathrm{proj}}

\newcommand{\Id}{\mathrm{Id}}

\newcommand{\Int}{\mathrm{Int}}
\newcommand{\loc}{\mathrm{loc}}

\newcommand{\Imm}{\mathrm{Imm}}

\newcommand{\R}{\mathbb{R}}

\newcommand{\N}{\mathbb{N}}

\newcommand{\s}{\mathbb{S}}

\newcommand{\B}{\mathbb{B}}
\newcommand{\V}{\mathbb{V}}

\newcommand{\Hr}{\mathcal{H}}

\newcommand{\Dr}{\mathcal{D}}

\newcommand{\Er}{\mathcal{E}}
\newcommand{\Cr}{\mathcal{C}}
\newcommand{\Ur}{\mathcal{U}}

\newcommand{\Sr}{\mathcal{S}}
\newcommand{\Or}{\mathcal{O}}
\newcommand{\pr}{\mathcal{P}}
\newcommand{\Gr}{\mathcal{G}}
\newcommand{\Br}{\mathcal{B}}
\newcommand{\I}{\mathcal{I}}

\newcommand{\Sc}{\mathscr{S}}
\newcommand{\Ic}{\mathscr{I}}
\newcommand{\Bc}{\mathscr{B}}
\newcommand{\Cc}{\mathscr{C}}
\newcommand{\Vc}{\mathscr{V}}
\newcommand{\Gc}{\mathscr{G}}
\newcommand{\Pc}{\mathscr{P}}

\newcommand{\Ec}{\mathscr{E}}

\newcommand{\Arond}{\mathring{A}}

\newcommand{\vp}{\varphi}
\newcommand{\eps}{\varepsilon}

\newcommand{\geu}{g_{\mathrm{eucl}}}

\newcommand{\ve}{\vec{e}\, }
\newcommand{\vf}{\vec{f}}

\newcommand{\vu}{\vec{u}}
\newcommand{\vv}{\vec{v}}
\newcommand{\vy}{\vec{y}}
\newcommand{\vq}{\vec{q}}

\newcommand{\vep}{\vec{p}}
\newcommand{\vn}{\vec{n}}

\newcommand{\vtau}{\vec{\tau}}
\newcommand{\vnu}{\vec{\nu}}
\newcommand{\vx}{\vec{x}}
\newcommand{\vPhi}{\vec{\Phi}}
\newcommand{\vPsi}{\vec{\Psi}}
\newcommand{\vphi}{\vec{\phi}}
\newcommand{\vpsi}{\vec{\psi}}
\newcommand{\vII}{\vec{\II}}
\newcommand{\vH}{\vec{H}}
\newcommand{\vz}{\vec{z}}

\newcommand{\vA}{\vec{A}}
\newcommand{\vQ}{\vec{Q}}
\newcommand{\vpi}{\vec{\pi}}

\newcommand{\vsigma}{\vec{\sigma}}

\newcommand{\vtheta}{\vec{\theta}}

\newcommand{\vpit}{\vec{\pi}_{\top}}

\newcommand{\vpipe}{\vec{\pi}_{\perp}}
\newcommand{\vompe}{\vec{\omega}_{\perp}}
\newcommand{\pipa}{\pi_{\parallel}}
\newcommand{\ompa}{\omega_{\parallel}}
\newcommand{\vpip}{\vec{\pi}_{\Pc}}
\newcommand{\vpipp}{\vec{\pi}_{\Pc^{\perp}}}

\usepackage{authblk}
\title{Weak immersions with second fundamental form in a critical Sobolev space}
\date{\today}
\author[1]{Dorian Martino}
\author[2]{Tristan Rivière}
\affil[1,2]{Department of Mathematics, ETH Zürich, 101 Rämistrasse, 8092 Zürich, Switzerland. Email: dorian.martino@math.ethz.ch}

\begin{document}
	
	\maketitle
	
	\begin{abstract}
		We develop the analysis of  Lipschitz immersions of $n$-dimensional manifolds into $\mathbb{R}^d$ having their second fundamental forms bounded in the critical Sobolev space $W^{\frac{n}{2}-1,2}$ in dimension $n\geq 4$ even and any codimension. We prove that, while such a weak immersion is not necessary $C^1$, it generates a $C^1$ differential structure on the domain. More precisely, for any such an immersion, there is an atlas in which the first fundamental form is continuous and the transition maps are $C^1$. We prove that this $C^1$ structure is diffeomorphic to the original one. This result is the starting point of the analysis of the behavior of sequences of weak immersions with second fundamental forms  uniformly bounded in the critical Sobolev space $W^{\frac{n}{2}-1,2}$. In the second part of the paper we establish a weakly sequential closure theorem for such sequences. This analysis is motivated by the study of conformally invariant Lagrangian of immersions in dimension larger than two such as generalized Willmore energies, for instance the Graham--Reichert functional obtained in the computation of renormalized volumes of five-dimensional minimal submanifolds of the hyperbolic space $\mathbb{H}^{d+1}$. 
	\end{abstract}

\section{Introduction}

Compactness questions in geometric analysis are fundamental problems underlying the study of moduli spaces, geometric flows, and variational methods. For immersed submanifolds of the Euclidean space $\mathbb{R}^d$, establishing the compactness of critical points of geometric functionals (such as minimal submanifolds, CMC surfaces, or Willmore surfaces) first requires constructing an atlas satisfying a certain system in which the metric induced by the immersion is controlled, for instance by its volume, diameter, or second fundamental form. Once such an atlas has been constructed, one can use the associated system to obtain control of the immersion itself in a topology of higher regularity (for instance $C^k$ for large $k \geq 0$). In the present work, we are interested in the closure of $C^{\infty}$ immersions into $\mathbb{R}^d$ under a bound on their second fundamental form in certain Sobolev spaces. The motivation for studying such closure properties comes from the analysis of variational problems associated with conformally or scaling invariant Lagrangians of immersions in arbitrary dimensions and codimensions. Throughout the paper, we consider domains $\Sigma$ that are orientable $C^{\infty}$ manifolds of dimension $n$.\\

In the case $n=2$, given a Riemann surface $\Sigma$ and a $C^{\infty}$ immersion $\vPhi \in \Imm(\Sigma; \mathbb{R}^d)$, a controlled atlas can be obtained using the uniformization theorem. Indeed, the induced metric $g_{\vPhi}$ is conformal to a metric $h$ with constant Gaussian curvature on $\Sigma$. Thanks to the Deligne--Mumford compactness theorem, the degeneration of $h$ in the Teichmüller space of $\Sigma$ is governed by the length of its shortest geodesic, see for instance \cite{tromba1992}. Laurain and the second author proved in \cite{laurain2018} that one can construct an atlas of isothermal coordinates for $h$, where the Green kernel of $h$ is controlled solely by the topology of $\Sigma$. To obtain a controlled atlas for $g_{\vPhi}$, it remains to estimate the conformal factor, which satisfies the Liouville equation. Hence, the construction of an atlas of isothermal coordinates for $\vPhi$ reduces to understanding the Deligne--Mumford compactness and the Liouville equation. Using Coulomb frames, Hélein \cite{helein2002} proved that the Liouville equation possesses a div--curl structure, which provides a uniform control of the conformal factor on domains where the $L^2$-norm of the second fundamental form is sufficiently small. One can then study the closure of the set of immersions under a bound on the $L^2$-norm of the second fundamental form. This setting was studied by Mondino and the second author in \cite{mondino2014}, where they proved that the worst possible degeneration that may occur is a bubble tree of \emph{weak} immersions, that is, a finite collection of weak immersions with isolated singularities (branch points) glued together along their singularities. Weak immersions were introduced by the second author in \cite{riviere2014} in order to provide a suitable analytic framework for developing variational problems on submanifolds. Indeed, it is shown in \cite{mondino2014} that, in the same way that Sobolev spaces are necessary to study elliptic PDEs, the setting of $C^{\infty}$ immersions must be enlarged. We now define the space of weak immersions $\mathcal{I}_{k,p}(\Sigma; \mathbb{R}^d)$, already discussed in \cite{MarRiv2025}.

\begin{defi}
	Let $n\geq 1$ be an integer, $(\Sigma,h)$ be a closed orientable $n$-dimensional Riemannian manifold and $d>n$ be an integer. Given $k\in\N$ and $p\in[1,+\infty]$, we define the notion of weak immersion $\I_{k,p}(\Sigma;\R^d)$ as follows:
	\begin{align*}
		\I_{k,p}(\Sigma;\R^d)\coloneqq \left\{
		\vPhi\in W^{k+2,p}(\Sigma;\R^d) : \begin{array}{l}
			\displaystyle \exists c_{\vPhi}>0,\ c^{-1}_{\vPhi} h \leq g_{\vPhi} \leq c_{\vPhi}\, h.
		\end{array}
		\right\}
	\end{align*}     
\end{defi}

In dimension $n \geq 3$, the study of $C^{\infty}$-immersions $\vPhi \in \Imm(\Sigma; \mathbb{R}^d)$ is less straightforward, since there is no canonical way to construct coordinates depending solely on the topology of the manifold $\Sigma$. It has been known since the work of DeTurck--Kazdan \cite{deturck1981} that harmonic coordinates (introduced by Einstein and Lanczos \cite{einstein1916,lanczos1922} in the context of general relativity) provide an alternative to isothermal coordinates: they always exist locally on any manifold and in any dimension, and yield the best possible regularity of the metric coefficients, depending on the regularity of its curvature. However, constructing and estimating the size of a domain carrying harmonic coordinates for a general metric $g$ on a general manifold $\Sigma^n$ is a delicate problem, see for instance \cite{czimek2019,anderson2004,yangI1992,yangII1992,jost1982}. For immersions $\vPhi \in \Imm(\Sigma^n; \mathbb{R}^d)$, such harmonic coordinates can be obtained once it is known that $\vPhi(\Sigma)$ can be represented locally as the graph of a function over a domain of fixed size. This property was established by Langer \cite{langer1985} in dimension $n=2$, and extended to higher dimensions by Breuning \cite{breuning2015} under the assumption that the second fundamental form $\vII_{\vPhi}$ of $\vPhi$ is bounded in some $L^p$ space with $p>n$. We refer to the works of Hutchinson \cite{hutchinson1986,hutchinson2} and Aiex \cite{aiex2024} for the definition of second fundamental forms in the setting of varifolds and for the study of variational problems in analogous non-critical cases. The critical case, where $\vII_{\vPhi}$ is bounded in $L^n$, is much more difficult and remains open, see for instance \cite{li2024} for a recent survey. Immersions in Sobolev spaces have already been studied by Mardare \cite{mardare2005,mardare2007}, Szopos \cite{szopos2008}, and later by Chen--Slemrod--Wang \cite{chen2010,chen20102}, who proved that one can reconstruct an immersion solely from the data of tensors in $L^p$ (for some $p>n$) satisfying the Gauss--Codazzi--Ricci equations in a weak sense. As a consequence of these works, the authors were also able to study the compactness of immersions with bounded second fundamental form in $L^p$, together with a non-degeneracy assumption on the induced metric. Still under this assumption, Li \cite{li2024} generalized these arguments to the case where the second fundamental form is bounded in a critical Morrey space, using compensated-compactness methods combined with Coulomb frames reminiscent of the work of the second author in collaboration with Struwe \cite{riviere2008}. We emphasize that, in the present work, no initial assumption is made on the induced metric, allowing for the presence of singularities. We work under the sole assumption that the second fundamental form is bounded in $W^{\frac{n}{2}-1,2}(\Sigma)$ (which embeds into $L^n(\Sigma)$ by Sobolev injection). This hypothesis precisely matches the requirements for studying scaling-invariant Lagrangians of immersions in arbitrary even dimensions, as we now describe.\\

From a physical perspective, there is a particular interest in considering weak immersions whose second fundamental form lies in some Sobolev space $W^{k,p}$. Indeed, the concept of renormalized volume in the AdS/CFT correspondence was introduced by Henningson--Skenderis \cite{henningson1998} to compute the Weyl anomaly. This procedure was further developed by Graham--Witten \cite{graham1999} for minimal submanifolds, motivated by the computation of various quantities such as the expectation values of Wilson loops and the entanglement entropy, see also \cite{graham2014} for applications. According to the area law introduced by Ryu--Takayanagi \cite{ryu2006}, the entanglement entropy of a quantum field theory in the AdS/CFT correspondence is expected to be proportional to the volume of certain minimal surfaces in non-compact manifolds. However, this volume is infinite. To address this issue, the idea of the \emph{renormalized volume} consists in expanding the volume asymptotically and keeping only the constant term in the expansion. For surfaces, Graham--Witten \cite{graham1999} recovered the Willmore energy, that is the $L^2$-norm of the mean curvature, as explained below. In general, they established the following formal expansion when the ambient manifold is the hyperbolic space $\mathbb{H}^{d+1} \simeq \mathbb{R}^d \times (0,+\infty)$ with asymptotic boundary $\partial_{\infty} \mathbb{H}^{d+1} = \mathbb{R}^d$. Let $Y^{n+1} \subset \mathbb{H}^{d+1}$ be a minimal submanifold with $\Sigma^n \coloneq \partial Y^{n+1} \subset \mathbb{R}^d$. Then the following expansion holds, depending on the parity of $n$:
\begin{align}\label{eq:anomaly}
	\vol_{\mathbb{H}^{d+1}}\!\left( Y^{n+1} \cap \{x^{d+1} > \varepsilon\} \right)
	\underset{\varepsilon \to 0}{=}
	\begin{cases}
		\displaystyle 
		\frac{a_n}{\varepsilon^n} + \frac{a_{n-2}}{\varepsilon^{n-2}} + \text{(odd powers)} + \frac{a_1}{\varepsilon} + a_0 + o(1), 
		& \text{if $n$ is odd}, \\[4mm]
		\displaystyle 
		\frac{b_n}{\varepsilon^n} + \frac{b_{n-2}}{\varepsilon^{n-2}} + \text{(even powers)} + \frac{b_2}{\varepsilon^2} + \mathcal{E} \log\!\left(\frac{1}{\varepsilon}\right) + b_0 + o(1),
		& \text{if $n$ is even}.
	\end{cases}
\end{align}
Since the left-hand side is an integral, all the coefficients $a_k$, $b_j$, and $\mathcal{E}$ are themselves integrals. Graham--Witten proved that $a_1$ and $\mathcal{E}$ define conformally invariant functionals on $\Sigma$. When $n=2$, they obtain the Willmore energy 
\[
\mathcal{E} = -\frac{1}{2}\int_{\Sigma} |\vec{H}_{\Sigma}|^2.
\]
For four-dimensional submanifolds, the computation was independently carried out by Graham--Reichert \cite{graham2020} and Zhang \cite{zhang2021}, who obtained the following functional (originally discovered by Guven \cite{guven205} in 2005 using a different approach):
\begin{align*}
	\forall \vPhi \in \Imm(\Sigma^4; \mathbb{R}^d), \qquad 
	\mathcal{E}_{\mathrm{GR}}(\vPhi)
	= \int_{\Sigma} 
	 |\nabla \vec{H}_{\vPhi}|^2_{g_{\vPhi}} 
	- |\vII_{\vPhi} \cdot \vec{H}_{\vPhi}|^2_{g_{\vPhi}} 
	+ 7 H_{\vPhi}^4 \ d\vol_{g_{\vPhi}}.
\end{align*}
The negative sign in this expression has a significant impact, as it implies that the functional $\mathcal{E}_{\mathrm{GR}}$ is never bounded from below for any manifold $\Sigma$, see \cite{graham2020,martino2024}. This fact has subsequently been interpreted in the context of the AdS/CFT correspondence in \cite{anastasiou2025}. The expansion \eqref{eq:anomaly} was later studied in full generality by Gover--Waldron \cite{gover2017}, who proved in particular that if $n$ is odd, then the operators involved in the computation of $a_0$ are non-local of Dirichlet-to-Neumann type. Consequently, the cases of even and odd dimensions require different analytical tools. In this work, we shall focus on the case where $n$ is even.\\

In a broader context, when $n$ is even, generalized Willmore energies have recently been studied for their relationship with $Q$-curvatures, we refer to \cite{vyatkin} for an introduction. In \cite{gover2020,blitz2024}, Blitz, Gover, and Waldron define a generalized Willmore energy as a conformally invariant functional on immersions $\vPhi \in \Imm(\Sigma^n; \mathbb{R}^d)$ whose Euler--Lagrange equation has leading-order term $\Delta_{g_{\vPhi}}^{\frac{n}{2}} H_{\vPhi}$. In other words, a generalized Willmore energy is a functional of the form
\begin{align*}
	\mathcal{F}(\vPhi) 
	= \int_{\Sigma^n} 
	\big| \nabla^{\frac{n-2}{2}} \vII_{\vPhi} \big|^2_{g_{\vPhi}}\, d\vol_{g_{\vPhi}} 
	+ \text{l.o.t.}
\end{align*}
In \cite{blitz2024}, Blitz--Gover--Waldron proved that such functionals can be constructed using a notion of extrinsic $Q$-curvature. Owing to the wide variety of possible generalized functionals, it is reasonable to first restrict attention to the case of \emph{coercive} energies $\mathcal{F}$. One may consider a generalized Willmore energy $\mathcal{F} \colon \Imm(\Sigma^n; \mathbb{R}^d) \to [0,+\infty)$ such that there exists $C > 0$ satisfying
\begin{align*}
	\forall\, \vPhi \in \Imm(\Sigma^n; \mathbb{R}^d), 
	\qquad 
	\mathcal{F}(\vPhi) 
	\geq 
	C \left\| \vII_{\vPhi} \right\|_{W^{\frac{n}{2}-1,2}(\Sigma, g_{\vPhi})}^2.
\end{align*}
However, the right-hand side is not scale invariant, and such a lower bound is therefore not reasonable. We thus replace it by the following scale-invariant functional. Given an open set $U \subset \Sigma^n$ (with $n \geq 4$ even), we define
\begin{align*}
	\mathcal{E}(\vPhi; U) 
	\coloneqq  
	\sum_{i=0}^{\frac{n}{2}-1}  
	\int_{U} 
	\Big| \nabla^i \vII_{\vPhi} \Big|^{\frac{n}{1+i}}_{g_{\vPhi}}\, 
	d\vol_{g_{\vPhi}}.
\end{align*}
We write $\mathcal{E}(\vPhi) \coloneqq \mathcal{E}(\vPhi; \Sigma)$. The functional $\mathcal{E}$ is scale invariant, that is,
\[
 \forall \lambda>0,\qquad \mathcal{E}(\lambda \vPhi) = \mathcal{E}(\vPhi) .
\]
For instance, in dimension $n = 4$, we obtain the functional
\begin{align*}
	\forall\, \vPhi \in \Imm(\Sigma^4; \mathbb{R}^d), \qquad 
	\mathcal{E}(\vPhi)
	= \int_{\Sigma} 
	|\nabla \vII_{\vPhi}|^2_{g_{\vPhi}} 
	+ |\vII_{\vPhi}|_{g_{\vPhi}}^4 
	\ d\vol_{g_{\vPhi}}.
\end{align*}
In dimension $n = 6$, we have
\begin{align*}
	\forall\, \vPhi \in \Imm(\Sigma^6; \mathbb{R}^d), \qquad 
	\mathcal{E}(\vPhi)
	= \int_{\Sigma} 
	|\nabla^2 \vII_{\vPhi}|_{g_{\vPhi}}^2
	+ |\nabla \vII_{\vPhi}|_{g_{\vPhi}}^3
	+ |\vII_{\vPhi}|_{g_{\vPhi}}^6
	\ d\vol_{g_{\vPhi}}.
\end{align*}
Since our motivation is to study the case where $n$ is even, most of the results in the present work are proved in this setting. Nevertheless, some of them remain valid in all dimensions, and this distinction will be made explicit throughout the paper.\\

This framework suggests that immersions whose second fundamental form is bounded in a Sobolev space $W^{k,p}$, for some $k \geq 1$ and $p > 1$, deserve special attention. If $\frac{1}{p} - \frac{k}{n} < \frac{1}{n}$, then $W^{k,p}(\Sigma^n) \hookrightarrow L^q(\Sigma^n)$ for some $q > n$, and we are back in the non-critical setting discussed in 
\cite{breuning2015,hutchinson1986,mardare2005,mardare2007,szopos2008,chen2010,chen20102}. In the case of generalized Willmore energies, one has $\frac{1}{p} - \frac{k}{n} = \frac{1}{n}$ in which case the standard Sobolev embedding gives $W^{k,p}(\Sigma^n) \hookrightarrow L^n(\Sigma^n)$. At first sight, this seems to offer no improvement compared to the assumption $\vII_{\vPhi} \in L^n$. However, the situation improves dramatically. We show in \Cref{th:atlas} below that, in the Willmore setting, namely when $\vII_{\vPhi} \in W^{\frac{n}{2}-1,2}(\Sigma)$, the immersion $\vPhi$ induces a controlled $C^1$ differential structure on $\Sigma$, even though the assumption  $\vPhi \in W^{\frac{n}{2}+1,2}(\Sigma)$
does not in general imply that its derivatives are continuous, and therefore the induced metric $g_{\vPhi}$ is \emph{a priori not continuous}. Our first result is the following, we refer to \Cref{th:Atlas} for a detailed version. 

\begin{theorem}\label{th:atlas}
	Let $\Sigma$ be a closed oriented manifold of even dimension $n\geq 4$. Every $\vPhi\in \I_{\frac{n}{2}-1,2}(\Sigma;\R^d)$ induces a differential structure of class $C^1$. More precisely, there exists an atlas of harmonic coordinates in which, the metric $g_{\vPhi}$ is continuous and the transition functions are of class $C^1$. This differential structure is $C^1$-diffeomorphic to the initial differential structure of $\Sigma$.
\end{theorem}

For $n=2$, this result holds when harmonic coordinates are replaced by isothermal coordinates (which form a special subclass of harmonic coordinates), see for instance \cite{riviere2016,Lan2025}. In that case, the underlying differential structure is of class $C^{\infty}$. The main tools involved are the use of Coulomb frames and div--curl structures. In the result above, we rely crucially on our previous work \cite{MarRiv2025}, where we constructed harmonic coordinates (also via Coulomb frames) for weak immersions of $\R^n$ and $\B^n$ under suitable boundary conditions. \\

Our next step is to study the compactness properties of immersions with bounded energy $\Er$ in even dimensions $n \geq 4$. More precisely, let $\Sigma$ be a closed, oriented $n$-dimensional manifold with $n \geq 4$ even. We investigate the closure of the set of $C^{\infty}$ immersions $\vPhi \in \Imm(\Sigma;\R^d)$ under the energy bound $\Er(\vPhi) < E$ for some fixed $E > 0$. That is, we consider a sequence $(\vPhi_k)_{k\in\N} \subset \Imm(\Sigma;\R^d)$ satisfying $\Er(\vPhi_k) < E$, and we show that, up to a subsequence, $\vPhi_k$
converges in a weak sense (possibly after rescaling) to a \emph{weak immersion} $\vPhi_{\infty} \in \I_{\frac{n}{2}-1,2}(\Sigma_{\infty};\R^d)$.
Since this constraint does not impose any bound on higher derivatives of $\vPhi_k$, one cannot in general expect strong convergence. We show, however, that only finitely many singularities can appear on the image $\vPhi_k(\Sigma)$. Away from these singularities, we obtain a controlled atlas in which the sequence $(\vPhi_k)$ converges weakly in $W^{\frac{n}{2}+1,2}_{\loc}$ and strongly in $C^0_{\loc}$. We will see below that the topology of $\Sigma_{\infty}$ exhibits a bubble-tree structure analogous to that obtained in \cite{mondino2014}. \\

The construction of the atlas in \Cref{th:Atlas} provides precise control over the size and regularity of the charts. This, in turn, allows us to handle sequences of immersions with uniformly bounded energy. We prove that the closure of $\Imm(\Sigma;\R^d)$ with respect to $\Er$ is contained in $\I_{\frac{n}{2}-1,2}(\Sigma;\R^d)$, up to a finite number of isolated singularities on the image and on the domain.

\begin{theorem}\label{th:compactness}
	Let $E>0$ and $\Sigma$ be a closed connected oriented manifold of even dimension $n\geq 4$. Consider a sequence $(\vPhi_k)_{k\in\N}\subset \Imm(\Sigma;\R^d)$.
	Then there exist 
	\begin{itemize}
		\item a finite number of points $\vq_1,\ldots,\vq_I\in \R^d$,
		
		\item an $n$-dimensional $W^{\frac{n}{2}+1,(2,1)}$-manifold $\Sigma_{\infty}$ possibly non-closed and non-connected,\footnote{We do not describe precisely the topology of $\Sigma_{\infty}$ in this work. For instance, we do not compare its topology and the one of $\Sigma$. We also do not study precisely the topology of the convergence of $\vPhi_k$ to $\vPhi_{\infty}$.}
		
		\item a weak immersion $\vPsi_{\infty}\in \I_{\frac{n}{2}-1,2}(\Sigma_{\infty};\R^d)$,
		
		\item a sequence of maps $\Theta_k\in\Conf(\R^d)$ given by the composition of translations, dilations and rotations,
	\end{itemize}
	such that for any $R>0$,  $(\Sigma\setminus \vPsi_k^{-1}(\B^d(\vq_1,R)\cup \cdots \cup \B^d(\vq_I,R)),g_{\vPhi_k})$ converges in the pointed Gromov--Hausdorff distance to $(\Sigma_{\infty}\setminus \vPsi_{\infty}^{-1}(\B^d(\vq_1,R)\cup \cdots \cup \B^d(\vq_I,R)),g_{\vPsi_{\infty}})$ and the maps $\vPsi_k\coloneqq \Theta_k\circ \vPhi_k$ satisfy
	\begin{align*}
		&\bullet\qquad  \vPsi_k  \xrightarrow[k\to+\infty]{} \vPsi_{\infty} \qquad \text{in }C^0_{\loc}(\Sigma_{\infty}) , \\[2mm]
		&\bullet\qquad  \vol_{g_{\vPsi_k}}(\Sigma) \xrightarrow[k\to +\infty]{} \vol_{g_{\vPsi_{\infty}}} (\Sigma_{\infty}), \\[2mm]
		&\bullet\qquad  \overline{\vPsi_{\infty}(\Sigma_{\infty}) } = \vPsi_{\infty}(\Sigma_{\infty})\cup \left\{ \vq_1,\ldots,\vq_I\right\}.
	\end{align*}
	We have the following description of the singularities $\vq_1,\ldots,\vq_I$
	\begin{itemize}
		\item for $R>0$ small enough, the restriction of $\vPsi_{\infty}$ to each connected component of the sets $\vPsi_{\infty}^{-1}(\B^d(\vq_i,R)\setminus\{\vq_i\})\subset\Sigma_{\infty}$ realizes an embedding from $\B^n(0,1)\setminus \{0\}$ into $\R^d$ that extends across the origin to a weak immersion $\vPsi_{\infty}\in \I_{0,(n,2)}(\B^n(0,1))$. \footnote{This is a major difference with the case $n=2$, where one can obtain branch points and where the restrictions of $\vPhi$ to good slices $\vPhi^{-1}(\s^1(\vq_i,R))$ can parametrize multiple covers of $\s^1$. This is due to the fact that $\s^{n-1}$ is simply connected in dimension $n-1\geq 2$ but not in dimension $n-1=1$.}
		
		\item The exist harmonic coordinates in which $\big( g_{\vPsi_{\infty}} \big)_{\alpha\beta} \in C^0(\B^n(0,1))$. In particular, the manifold $\left(\Sigma_{\infty},g_{\vPsi_{\infty}} \right)$ can be completed by adding a finite number of points into a closed $C^1$ manifold.
	\end{itemize}
\end{theorem}

The key ingredient is a refinement of the Sobolev embeddings using Lorentz spaces. Indeed, if $\frac{1}{p} - \frac{k}{n} = \frac{1}{n}$, then we have the embedding $W^{k,p}(\Sigma) \hookrightarrow L^{(n,p)}(\Sigma)$ for $p < n$, which provides a slight improvement over the standard embedding into $L^n$. Given a weak immersion $\vPhi \in \I_{\frac{n}{2}-1,2}(\B^n;\R^d)$ (for $n$ even), the authors proved in \cite{MarRiv2025} that harmonic coordinates can be constructed under a smallness assumption on $\vII_{\vPhi}$ in the critical Sobolev space $W^{\frac{n}{2}-1,2}$, together with appropriate boundary conditions. The fundamental reason lies in the fact that the Riemann tensor of $g_{\vPhi}$ is quadratic in the second fundamental form, thanks to the Gauss--Codazzi equations:
\begin{align*}
	\Riem^{g_{\vPhi}}_{ijkl} = \big(\vII_{\vPhi}\big)_{ik} \cdot \big(\vII_{\vPhi}\big)_{jl} - \big(\vII_{\vPhi}\big)_{il} \cdot \big(\vII_{\vPhi}\big)_{jk}.
\end{align*}
Consequently, if $\vII_{\vPhi}$ belongs to the Lorentz space $L^{(n,2)}$, then $\Riem^{g_{\vPhi}}$ lies in $L^{(\frac{n}{2},1)}$. This Lorentz exponent $1$ provides the crucial improvement in the elliptic estimates, allowing us to recover $L^{\infty}$ control on the coefficients of the metric in harmonic coordinates. Hence, the regularity of the coefficients $(g_{\vPhi})_{ij}$ improves from merely $VMO(\B^n)$ to $C^0(\B^n)$. \\

To construct an atlas associated with a $C^{\infty}$ immersion $\vPhi \in \Imm(\Sigma;\R^d)$, we rely on the recent work \cite{MarRiv2025} where the authors construct harmonic coordinates for weak immersions on $\R^n$, together with a combination of ideas from geometric measure theory and Morse theory required in order to deduce geometric and topological properties on the domain $\Sigma$ from estimates on the second fundamental form of $\vPhi$. We fix a point $\vep \in \vPhi(\Sigma)$ and a radius $r > 0$ such that 
\[
\Er\big(\vPhi; \vPhi^{-1}(\B^d(\vep,r))\big) < \eps^n.
\]
We then find a suitable slice of $\vPhi(\Sigma) \cap \B^d(\vep,r)$, namely a sphere $\s^{d-1}(\vq,s)$ with $\vq$ close to $\vep$ and $s$ close to $r$, such that the intersection $\vPhi(\Sigma)\cap\s^{d-1}(\vq,s)$ is an $(n-1)$-dimensional submanifold of $\s^{d-1}(\vq,s)$ with second fundamental form bounded in $L^n$ and bounded volume. On this slice, we prove that the oscillations of $\vn_{\vPhi}$ are bounded above by $\eps$. Hence, for $\eps$ small enough, the restriction of $\vPhi$ to each connected component $\Sr$ of $\vPhi^{-1}(\s^{d-1}(\vq,s))$ is a local graph over a given $(n-1)$-dimensional round sphere. This property defines a smooth covering map $\Sr \to \s^{n-1}$. Since $n-1 \ge 3$, the sphere $\s^{n-1}$ is simply connected, hence the covering map is a diffeomorphism. Therefore, every connected component of $\vPhi^{-1}(\s^{d-1}(\vq,s))$ is diffeomorphic to $\s^{n-1}$, and its image under $\vPhi$ is a graph close to a fixed $(n-1)$-dimensional sphere. This sharply contrasts with the two-dimensional case, where a covering map $\Sr \to \s^1$ need not be bijective (e.g.\ connected double covers of $\s^1$), leading to the occurrence of multivalued graphs. Such phenomena do not appear in dimensions $n \ge 4$, since $\s^{n-1}$ becomes simply connected. For example, Brieskorn manifolds, which are topological spheres but not diffeomorphic to round spheres, are \emph{branched} coverings, see \cite{milnor1956}.  \\

Once the good slices are understood, we extend them by flat planes as in \cite{MarRiv2025}, obtaining an immersion $\vPsi \colon \tilde{\Sigma} \to \R^d$ with small second fundamental form in $L^{(n,2)}$, where $\tilde{\Sigma}$ is the manifold obtained by gluing $\vPhi^{-1}(\B^d(\vq,s))$ with half-spheres along each connected component of its boundary $\vPhi^{-1}(\s^{d-1}(\vq,s))$. We study the topology of $\tilde{\Sigma}$ using Morse theory. After a well-chosen inversion, we can construct a Morse function on $\tilde{\Sigma}$ with very few critical points, following the strategy of Chern--Lashof \cite{chern1957,chern1958}. As a consequence of the proof of the Reeb theorem \cite[Theorem 4.1]{milnor1963}, we deduce that $\tilde{\Sigma}$ is homeomorphic (though not necessarily diffeomorphic) to the $n$-sphere, and that $\vPhi^{-1}(\B^d(\vq,s))$ is diffeomorphic (not merely homeomorphic) to a ball $\B^n$. Harmonic coordinates are then obtained by direct application of \cite{MarRiv2025}.  
More precisely, we shall prove that there exists $\eps_0 > 0$, depending only on $n$, $d$, and the diameter of $\vPhi(\Sigma)$ and its energy, such that if a ball $\B^d(\vep,r)$ satisfies 
\[
\Er\big(\vPhi; \vPhi^{-1}(\B^d(\vep,r))\big) < \eps_0^n,
\]
then each connected component of $\vPhi^{-1}(\B^d(\vep,r))$ admits controlled harmonic coordinates. This provides the atlas of \Cref{th:atlas} for smooth immersions. The result for weak immersions follows by the approximation argument developed in \cite{Lan2025}. \\

We now sketch the proof of \Cref{th:compactness}. Let $(\vPhi_k)_{k\in\N} \subset \Imm(\Sigma;\R^d)$ be a sequence with uniformly bounded energy and diameter. For every $x \in \Sigma$, there exists $r_{x,k} \in (0,1)$ such that 
\begin{align*}
	\Er\Big(\vPhi_k; \vPhi_k^{-1}\big(\B^d(\vPhi_k(x),r_{x,k})\big)\Big) = \frac{\eps_0}{2},
\end{align*}
where $\eps_0$ is as defined above. For $r>0$ and $k\in\N$, we define the \emph{good} and \emph{bad} sets of scale $r$ for $\vPhi_k$ by
\begin{align}
	\Gr_k^r &\coloneqq \bigcup \left\{ \vPhi_k^{-1}\big(\B^d(\vPhi_k(x),r_{x,k})\big) : x\in\Sigma,\ r_{x,k}\geq r \right\}, \label{eq:good1}\\[3mm]
	\Br_k^r &\coloneqq \bigcup \left\{ \vPhi_k^{-1}\big(\B^d(\vPhi_k(x),r_{x,k})\big) : x\in\Sigma,\ r_{x,k}< r \right\}. \label{eq:bad1}
\end{align}
By construction, the manifolds $(\Gr_k^r,g_{\vPhi_k})$ carry a controlled atlas of harmonic coordinates. By standard compactness results for Riemannian manifolds (see for instance \cite{petersen2016}), we deduce that for every $r>0$, the sequence $(\Gr_k^r,g_{\vPhi_k})_{k\in\N}$ converges in the pointed $C^1$ topology to a manifold $(\Gr_\infty^r,g_\infty)$, whose charts have the same regularity as in \Cref{th:Atlas}. Moreover, for $s<r$, we have $\Gr_\infty^r \subset \Gr_\infty^s$. We then define the limit manifold 
\[
\Sigma_\infty \coloneqq \bigcup_{r>0} \Gr_\infty^r,
\]
which is a priori neither connected nor closed but carries a $C^1$ differential structure. For the bad set, a standard covering argument shows that $\vPhi_k(\Br_k^r)$ is contained in finitely many balls $\B^d(\vq_{i,k},r)$. Passing first to the limit $k\to\infty$ and then letting $r\to 0$, we conclude that the singular set of the limit image $\vPhi_\infty(\Sigma_\infty)$ is finite. \\

This approach does not yield the same degree of understanding as that obtained in \cite{mondino2014} for surfaces. In dimension $n=2$, the uniformization theorem combined with the Deligne--Mumford decomposition provides a powerful mechanism for detecting bubbles purely from the domain geometry. In that setting, blowing up around a concentration point on the domain produces bubbles with the topology of $\s^2$. In contrast, in the present work, the construction of charts (and of the entire atlas) proceeds via the \emph{image} of the immersion. This is why inversions are not used in the construction of the conformal maps $\Theta_k$ in \Cref{th:compactness}. Consequently, if the main part of an immersion collapses to a point, one must perform an inversion with respect to a suitably chosen ball in $\R^d$ and study the behaviour of $\Er$ under this conformal transformation (the functional $\Er$ is not conformally invariant). Moreover, we do not develop a generalization of the distinction between ``neck regions’’ and ``collars’’ as in the two-dimensional case. \\

To analyze the singularities of the domain $\Sigma_\infty$, we consider two good slices around a given singularity $\vq_i$ on the target. Given $r>0$ and $\delta \in (0,1)$ (eventually sent to $0$), we take two good slices $\vPsi_\infty(\Sigma_\infty) \cap \s^{d-1}(\vep_{i,1},\rho_1)$ and $\vPsi_\infty(\Sigma_\infty) \cap \s^{d-1}(\vep_{i,2},\rho_2)$ with $\vep_{i,1}\simeq \vep_{i,2}\simeq \vq_i$ and radii $\rho_1 \simeq r$ and $\rho_2 \simeq \delta r$. Using the same argument as in the chart construction, we show that for each connected component $\Cr \subset \vPsi_\infty^{-1}(\B^d(\vq_i,\rho_1))$, the intersection $\Cr \cap \vPsi_\infty^{-1}\big(\B^d(\vep_{i,1},\rho_1)\setminus\B^d(\vep_{i,2},\rho_2)\big)$ is homeomorphic to an annulus. Letting $\delta \to 0$, we find that $\Cr$ is homeomorphic to $\B^n \setminus \{0\}$. Hence, $\Sigma_\infty$ can be viewed as the regular part of a metric space whose singularities occur only at finitely many points. \\

Several aspects remain incompletely understood. First, the topology of the limiting manifold near the singularities requires further clarification. Second, the behaviour of $\vPsi_\infty$ close to the singularities is not yet fully characterized. In the current description of singularity formation, the maps $\Theta_k$ are employed only to rescale the images $\vPhi_k(\Sigma)$ to unit size. However, parts of the immersion may still degenerate near the singularities. 
\begin{op}
		\begin{enumerate}
		\item Describe the difference between the topology of $\Sigma_{\infty}$ in \Cref{th:compactness} and the one of $\Sigma$.
		
		\item It would be useful to study the asymptotic behaviour of the immersions $\vPsi_k$ and $(\Sigma,g_{\vPsi_k})$ in the intermediate regions between the singularities and the macroscopic limits. It would be interesting to study some blow up procedures as well, to explore the structure of the immersion near the singularities.
	\end{enumerate}
\end{op}

\subsubsection*{Organization of the paper.}

In \Cref{sec:Preliminaries}, we introduce some notations used throughout the paper, the setting of weak immersions and Lorentz spaces. In \Cref{sec:Extrinsic}, we gather the required geometric and analytic properties such as Sobolev inequalities, Hardy inequalities and the volume growth of extrinsic balls,  which are consequences of the assumptions $\vII\in W^{\frac{n}{2}-1,2}(\Sigma^n)$. In order to prove \Cref{th:atlas}, we proceed by approximation by $C^{\infty}$-immersions. In \Cref{sec:One_chart}, we construct a chart of harmonic coordinates near a given point for a $C^{\infty}$-immersion and record the necessary estimates in order to pass to the limit. In \Cref{sec:Atlas}, we gather these charts to obtain an atlas for $C^{\infty}$-immersions and obtain the existence of an atlas for weak immersions by an approximation argument. This will prove \Cref{th:atlas}. In \Cref{sec:Compactness}, we prove \Cref{th:compactness} and study the singularities of the limiting domain.\\

    \subsubsection*{Acknowledgments.} 
    This project is financed by Swiss National Science Foundation, project SNF 200020\textunderscore219429.

	\section{Geometric and analytic setting}\label{sec:Preliminaries}

	\subsection{Notations}
	
	Given a Riemannian manifold $(\Sigma,g)$, we will denote $B_{\Sigma,g}(a,r)$ (or $B_g(a,r)$ when $\Sigma$ is implicit) the open geodesic ball of center $a\in \Sigma$ and radius $r>0$. When $\Sigma = \R^d$, we will denote $\B^d(a,r) \coloneqq B_{(\R^d,\geu)}(a,r)$. Since we consider immersions $\vPhi\colon \Sigma \to \R^d$, we will use variables in the domain $\Sigma$ and the target $\R^d$. To distinguish them, we will use arrows for the variables in the target $\R^d$. For example, we will denote $x\in \Sigma$ and $\vx \in \R^d$. For immersions $\vPhi\colon \Sigma\to \R^d$, we will denote $g_{\vPhi}\coloneq \vPhi^*\geu$ the first fundamental form and $\vII_{\vPhi}$ its second fundamental form defined as follows. If $\vn_{\vPhi} = \vn_1\wedge \cdots \wedge \vn_{d-n}\colon \Sigma\to \Lambda^{d-n}\R^d$ is the generalized Gauss map of $\vPhi$, then we have in coordinates
	\begin{align*}
		\big( \vII_{\vPhi} \big)_{ij} = \proj_{T\vPhi(\Sigma)^{\perp}} \left( \dr^2_{ij} \vPhi \right) = -\sum_{\alpha=1}^{d-n} \left( \dr_i \vPhi \cdot \dr_j \vn_{\alpha} \right)\, \vn_{\alpha}.
	\end{align*}
	We define the mean curvature of $\vPhi$ as 
	\begin{align*}
		\vH_{\vPhi} \coloneq \frac{1}{n}\, \tr_{g_{\vPhi}}\big( \vII_{\vPhi} \big) = \frac{1}{n}\, \lap_{g_{\vPhi}} \vPhi.
	\end{align*}
	For an immersion $\vPsi\colon \Sigma\to \s^{d-1}$, we will denote $\vA_{\vPsi}$ its second fundamental form.\\
	
	\subsection{Lorentz spaces}
	
	For an introduction to Lorentz spaces on more general spaces, we refer to \cite{bennett1988,grafakos2014}. Given a Riemannian manifold $(M^n,g)$, we define the Lorentz space $L^{(p,q)}(M)$ for $p,q\in[1,+\infty]$ as the set of measurable functions $f\colon M\to \R$ such that 
	\begin{align*}
		\|f\|_{L^{(p,q)}(M,g)}^q = \|f\|_{L^{(p,q)}(M)}^q = p \left\| \lambda\mapsto \lambda\, \vol_g\left(\{x\in M: |f(x)|>\lambda \}\right)^{\frac{1}{p}} \right\|_{L^q\left((0,+\infty),\frac{d\lambda}{\lambda}\right)}^q < +\infty.
	\end{align*}
	The Lorentz spaces verify the following properties:
	\begin{enumerate}
		\item If $p=q$, then $L^{(p,p)}(M)=L^p(M)$ with 
		\begin{align*}
			\forall f\in L^p(M),\qquad \|f\|_{L^p(M)} = \|f\|_{L^{(p,p)}(M)}.
		\end{align*}
		\item If $q<r$, then $L^{(p,q)}(M) \subsetneq L^{(p,r)}(M)$ with 
		\begin{align*}
			\forall f\in L^{(p,q)}(M),\qquad \|f\|_{L^{(p,r)}(M)} \leq C(p,q,r)\, \|f\|_{L^{(p,q)}(M)}.
		\end{align*}
		\item If $p<q$ and $s,t\in[1,\infty]$ and $\vol_g(M)<+\infty$, then $L^{(q,s)}(M)\subsetneq L^{(p,t)}(M)$ with 
		\begin{align*}
			\forall f\in L^{(q,s)}(M),\qquad \|f\|_{L^{(p,t)}(M)} \leq C(p,q,s,t)\, \vol_g(M)^{\frac{1}{p}-\frac{1}{q}} \, \|f\|_{L^{(q,s)}(M)}.
		\end{align*}
	\end{enumerate}
	We denote $W^{k,(p,q)}(M)$ the space of functions $f\colon M\to\R$ whose first $k^{th}$ derivatives lie in $L^{(p,q)}(M)$. We then have the following Sobolev inequalities on $(\B^n,\geu)$, valid for all $f\in C^{\infty}(\B)$ and $p\in[1,n)$ and $q\in[1,+\infty]$,
	\begin{align*}
		\begin{cases}
			\displaystyle \|f\|_{L^{\left(\frac{np}{n-p},q\right)}(\B)} \leq C(n,p,q)\, \|f\|_{W^{1,(p,q)}(\B)} , \\[3mm]
			\displaystyle \|f\|_{L^{\infty}(\B)} \leq C(n)\, \|f\|_{W^{1,(n,1)}(\B)}.
		\end{cases}
	\end{align*}
	If we denote $\bar{f}\coloneq \fint_{\B} f(x)\, dx$, we also have the estimates
	\begin{align}\label{eq:Sobolev_avg}
		\begin{cases}
			\displaystyle \|f-\bar{f}\|_{L^{\left(\frac{np}{n-p},q\right)}(\B)} \leq C(n,p,q)\, \|\g f\|_{L^{(p,q)}(\B)} , \\[3mm]
			\displaystyle \|f-\bar{f}\|_{L^{\infty}(\B)} \leq C(n)\, \|\g f\|_{L^{(n,1)}(\B)}.
		\end{cases}
	\end{align}
	
	\subsection{Weak immersions and the functional $\Er$}

	Since we will work with Lorentz spaces, it will be important to define a version of weak immersions when the second fundamental form lies in some Lorentz--Sobolev spaces. We set the following notations.
	\begin{defi}
		Let $(\Sigma^n,h)$ be a closed Riemannian manifold and $d>n$ be an integer. Given $k\in\N$ and $p\in[1,+\infty]$, we define the notion of weak immersion $\I_{k,p}(\Sigma;\R^d)$ as follows:
		\begin{align*}
			\I_{k,p}(\Sigma;\R^d)\coloneqq \left\{
			\vPhi\in W^{k+2,p}(\Sigma;\R^d) :  \exists c_{\vPhi}>0,\ c^{-1}_{\vPhi} h \leq g_{\vPhi} \leq c_{\vPhi}\, h
			\right\}.
		\end{align*}    
		We also have the Lorentz--Sobolev version, for $q\in[1,+\infty]$
		\begin{align*}
			\I_{k,(p,q)}(\Sigma;\R^d)\coloneqq \left\{
			\vPhi\in W^{k+2,(p,q)}(\Sigma;\R^d) :  \exists c_{\vPhi}>0,\ c^{-1}_{\vPhi} h \leq g_{\vPhi} \leq c_{\vPhi}\, h
			\right\}.
		\end{align*}    
	\end{defi}
	
	As discussed in Sections 6.1 and 6.2 of \cite{MarRiv2025}, we have that $\I_{\frac{n}{2}-1,2}(\Sigma^n;\R^d)\subset \I_{0,(n,2)}(\Sigma^n;\R^d)$. If $n$ is even, then the number $\Er(\vPhi)$ is well-defined for any $\vPhi\in \I_{\frac{n}{2}-1,2}(\Sigma^n;\R^d)$. In order to distinguish with $C^{\infty}$-immersion, we define
	\begin{align*}
		\Imm(\Sigma;\R^d) \coloneq \left\{ \vPhi\in C^{\infty}(\Sigma;\R^d) \ \text{immersion}
		\right\}.
	\end{align*}
	
	\subsection{Convergence of manifolds}
	
	Diverse notions of convergence of manifolds and maps between converging sequences of manifolds are available in the literature. We will mostly use the notion of pointed Gromov--Hausdorff convergence. We refer to \cite[Chapter 11]{petersen2016} for an introduction. The Gromov--Hausdorff distance between two metric spaces $(X,\dist_X)$ and $(Y,\dist_Y)$ is given by
	\begin{align*}
		\dist_{GH}(X,Y)\coloneq \inf\left\{  \dist_H^{Z}(\iota(X),\iota(Y)) : \begin{array}{l}
			\exists (Z,\dist_Z)\ \text{ a metric space},\\[2mm]
			\exists \iota\colon X\sqcup Y \to Z \text{ an embedding such that} \\[2mm]
			\quad \iota|_X \text{ and } \iota|_Y \text{ are isometric immersions}.
		\end{array}
		\right\}.
	\end{align*}
	In the above definition, the distance $\dist_H^{Z}$ between sets is the Hausdorff distance in the metric space $Z$. The map $\dist_{GH}$ defined on compact metric spaces is symmetric and verifies the triangle inequality. Moreover, we have $\dist_{GH}(X,Y)=0$ if and only if $X$ and $Y$ are isometric, see for instance \cite[Proposition 11.1.3]{petersen2016}. This makes the set of compact metric spaces up to isometries endowed with $\dist_{GH}$ a metric space. \\
	
	In order to deal with convergence of non-compact sets, the notion of pointed Gromov--Hausdorff convergence is useful. We define the pointed Gromov--Hausdorff distance between $(X,x)$ and $(Y,y)$, where $x\in X$ and $y\in Y$ as
	\begin{align*}
		\dist_{pGH}\big( (X,x),(Y,y) \big)\coloneq \inf\left\{  \dist_H^{Z}(\iota(X),\iota(Y)) + \dist_Z(\iota(x),\iota(y)): \begin{array}{l}
			\exists (Z,\dist_Z)\ \text{ a metric space},\\[2mm]
			\exists \iota\colon X\sqcup Y \to Z \text{ an embedding such that} \\[2mm]
			\quad \iota|_X \text{ and } \iota|_Y \text{ are isometric immersions}.
		\end{array}
		\right\}.
	\end{align*}
	We say that a sequence of pointed metric spaces $(X_i,x_i,\dist_i)_{i\in\N}$ converges to some limiting pointed metric space $(X_{\infty},x_{\infty},\dist_{\infty})$ in the pointed Gromov--Hausdorff topology, if for all $R>0$ there exists a sequence $(R_i)_{i\in\N}\subset (0,+\infty)$ converging to $R$ such that
	\begin{align*}
		\dist_{pGH}\Big[ \big( B_{X_i}(x_i,R_i), x_i, \dist_i \big), \big( B_{X_{\infty}}(x_{\infty},R), x_{\infty}, \dist_{\infty} \big) \Big]  \xrightarrow[i\to +\infty]{} 0.
	\end{align*}
	We characterize precompact sets for the Gromov--Hausdorff topology using the following result.
	\begin{proposition}[Proposition 11.1.10 in \cite{petersen2016}]\label{pr:charac_GH}
		Let $D>0$ and $\Cc$ be a set of compact metric spaces with diameters bounded from above by $D$. The following statements are equivalent
		\begin{enumerate}
			\item $\Cc$ is precompact for the Gromov--Hausdorff topology. 
			
			\item There exists a function $N_1\colon (0,\alpha)\to (0,+\infty)$ for some $\alpha>0$ such that for all $X\in\Cc$ and $\eps\in(0,\alpha)$, the maximum number of disjoint balls in $X$ of radius $\frac{\eps}{2}$ is bounded from above by $N_1(\eps)$.
			
			\item There exists a function $N_2\colon (0,\alpha)\to (0,+\infty)$ for some $\alpha>0$ such that for all $X\in\Cc$ and $\eps\in(0,\alpha)$, the minimum number of balls in $X$ of radius $\eps$ required to cover $X$ is bounded from above by $N_2(\eps)$.
		\end{enumerate}
	\end{proposition}
	
	We now define the convergence of maps between converging metric spaces.	We consider a sequence of metric spaces $X_k$ and $Y_k$ converging in the Gromov--Hausdorff topology to $X$ and $Y$. For each $k\in\N$, we consider a map $f_k\colon X_k\to Y_k$. We say that $(f_k)_{k\in\N}$ converges uniformly to a map $f\colon X\to Y$ if for any $x\in X$ and any sequence of points $x_k\in X_k$ converging to $x$, the sequence $(f_k(x_k))_{k\in\N}$ converges to $f(x)$. The sequence $(f_k)_{k\in\N}$ is called equicontinuous if for every $\eps>0$, there exists $\delta>0$ such that for any every sequence $x_k\in X_k$ and all $k\in\N$, we have $f_k(B_{X_k}(x_k,\delta)) \subset B_{Y_k}(f_k(x_k),\eps)$. We have the following version of Arzela--Ascoli lemma. 
	
	\begin{lemma}[Lemma 11.1.9 in \cite{petersen2016}]\label{lm:ArzelaAscoli}
		Let $(X_k)_{k\in\N}$ and $(Y_k)_{k\in\N}$ be two sequences of metric spaces converging in the pointed Gromov--Hausdorff distance to metric spaces $X$ and $Y$ respectively. Any equicontinuous family of maps $f_k\colon X_k\to Y_k$ has a uniformly converging subsequence. 
	\end{lemma}

	\section{Extrinsic estimates}\label{sec:Extrinsic}

		For the rest of the section, we fix a $C^{\infty}$ immersion $\vPhi\colon \Sigma\to \R^d$ be from a closed oriented manifold $\Sigma$ of even dimension $n\geq 4$. In this section, we will work mostly on the image $M\coloneq \vPhi(\Sigma)$, as it is done with varifolds. This will provide crucial uniform estimates. We proved in \cite{MarRiv2025} that weak immersions with small $L^{(n,\infty)}$ norm of $H$ have a uniformly controlled Sobolev constant, see \Cref{th:Sobolev} below. In \Cref{sec:Hardy}, we study Hardy inequalities. In \Cref{sec:Balls}, we prove that the volume of extrinsic balls has Euclidean growth. \Cref{sec:Hardy} and \Cref{sec:Balls} are valid only for $n$ even since we will need the energy $\Er$.\\
		
		We recall the following result, that will be crucial to obtain estimate on the growth of extrinsic balls, see Proposition \ref{pr:Lower_Extrinsic} below.
		\begin{theorem}[Theorem 6.2 in \cite{MarRiv2025}]\label{th:Sobolev}
			Let $d>n\geq 3$ be arbitrary integers. There exist constants $\eps=\eps(n,d)>0$ such that the following holds. Let $\vPhi\in \Imm(\Sigma; \R^d)$ be a $C^{\infty}$ immersion from a compact $n$-dimensional oriented manifold $\Sigma$ with $\dr\Sigma\neq \emptyset$. Assume that $\|\vH_{\vPhi}\|_{L^{(n,\infty)}(\Sigma,g_{\vPhi})}\leq \eps$.\\
			For any $p\in[1,n)$ and $q\in[1,+\infty]$, there exists a constant $c>0$ depending only on $n,d,p,q$ such that 
			\begin{align*}
				\forall \vp\in C^{\infty}_c(\Sigma),\qquad \|\vp\|_{L^{\left(\frac{np}{n-p},q\right)}\left(\Sigma,g_{\vPhi} \right)} \leq c\, \|d\vp\|_{L^{(p,q)}\left( \Sigma,g_{\vPhi} \right)}.
			\end{align*}
		\end{theorem}
		
		In this section, if $\vPhi\in \Imm(\Sigma^n;\R^d)$, we will consider the following measure on $\R^d$: for any Borel set $A\subset \R^d$, we denote
		\begin{align*}
			\mu_{\vPhi}(A) \coloneqq \vol_{g_{\vPhi}}\left( \vPhi^{-1}(A) \right).
		\end{align*}
		If $\vx\in \vPhi(\Sigma)$, we denote $\theta_{\vPhi}(\vx) \coloneq \sharp\vPhi^{-1}(\{\vx\})$ the density of $\vPhi(\Sigma)$ at $\vx$. We obtain 
		\begin{align}\label{eq:measure}
			\mu_{\vPhi} = \theta_{\vPhi}\, \Hr^n\llcorner \vPhi(\Sigma).
		\end{align}
		
		\subsection{Hardy inequality and consequences}\label{sec:Hardy}

		Following the proof of \cite[Theorem 3.1]{cabre2022}, we obtain a Hardy inequality for immersed submanifold of any codimension in Propositon \ref{pr:Hardy}. By iterating this inequality, we obtain uniform estimates on diverse quantities in Corollaries \ref{cor:II_xn2} and \ref{cor:int_xperp}.
		
		\begin{proposition}\label{pr:Hardy}
			Let $\Sigma$ be an n-dimensional complete manifold possibly with boundary of arbitrary dimension $n\geq 3$. Let $\vPhi\colon \Sigma^n\to \R^d$ be a smooth immersion and define $M\coloneqq \vPhi(\Sigma)$. Let $\B^d(\vx_0,r)\subset \R^d$ be a ball such that $\B^d(\vx_0,r)\cap \vPhi(\dr\Sigma)=\emptyset$. Let $\vp\in C^{\infty}_c(\B^d(\vx_0,r))$. For any $a\in[0,n)$, it holds
			\begin{align*}
				\int_{M\cap \B^d( \vx_0,r)} (n-a)\, \frac{|\vp(\vx)|^2}{|\vx|^a} +a\, \frac{|\vx^{\perp}|^2}{|\vx|^{a+2}}\, |\vp(\vx)|^2\ d\mu_{\vPhi}(\vx) \leq \frac{1}{n-a}\int_{M\cap \B^d(\vx_0,r)} \frac{|2\g^T\vp - \vH \vp|^2}{|\vx|^{a-2}}\, d\mu_{\vPhi}(\vx).
			\end{align*}
		\end{proposition}
		
		\begin{proof}
			Let $x\in \Sigma$ and a small neighbourhood $\Ur\subset \Sigma$ of $x$ such that $\vPhi\colon \Ur\to \vPhi(\Ur)\subset \R^d$ is bijective. Let $(\ve_1,\ldots,\ve_n)$ be an orthonormal frame for $T\vPhi(\Ur)$. We have the following computation of $\di_M(\vx)$:
			\begin{align*}
				\di_M(\vx)  = \sum_{i=1}^{n} \ve_i\cdot \g_{\ve_i} \vx = n.
			\end{align*}
			Integrating by parts (see for instance \cite[Equation 7.6]{simon1983}), we obtain
			\begin{align*}
				n\int_{M} \frac{|\vp(\vx)|^2}{|\vx|^a}\, d\mu_{\vPhi}(\vx) & = \int_{M} \frac{|\vp(\vx)|^2}{|\vx|^a}\, \di_M(\vx)\, d\mu_{\vPhi}(\vx) \\[2mm]
				& = -\int_{M} \left( 2\frac{\vp(\vx)}{|\vx|^a}\, \g^{\top} \vp(\vx)\cdot \vx + |\vp(\vx)|^2\, \vx\cdot \g^{\top} |\vx|^{-a} - \frac{|\vp(\vx)|^2}{|\vx|^a}\, \vH\cdot \vx\right)\ d\mu_{\vPhi}(\vx).
			\end{align*}
			We compute the middle term separately:
			\begin{align*}
				\vx\cdot \g^{\top} |\vx|^{-a} = -a\, \frac{|\vx^{\top}|^2}{|\vx|^{a+2}} = -a\, \frac{|\vx|^2 - |\vx^{\perp}|^2 }{ |\vx|^{a+2} }.
			\end{align*}
			Hence we obtain
			\begin{align}
				\int_{M} \left( (n-a)\, \frac{|\vp(\vx)|^2}{|\vx|^a}+ a\, \frac{|\vx^{\perp}|^2}{|\vx|^{a+2}}\, |\vp(\vx)|^2\right)\, d\mu_{\vPhi}(\vx) & = -\int_{M} \left( 2\, \frac{\vp(\vx)}{|\vx|^a}\, \g^{\top} \vp(\vx)\cdot \vx  - \frac{|\vp(\vx)|^2}{|\vx|^a}\, \vH\cdot \vx\right)\ d\mu_{\vPhi}(\vx) \label{eq:Hardy1}\\[2mm]
				& \leq \int_{M} \frac{|\vp(\vx)|}{|\vx|^{a-1}}\, \left| 2 \g^{\top} \vp(\vx) - \vp(\vx) \vH\right|  \ d\mu_{\vPhi}(\vx). \nonumber
			\end{align}
			By Hölder inequality, we obtain
			\begin{align*}
				\int_{M} \left( (n-a)\, \frac{|\vp(\vx)|^2}{|\vx|^a}+ a\, \frac{|\vx^{\perp}|^2}{|\vx|^{a+2}} \, |\vp(\vx)|^2\right)\ d\mu_{\vPhi}(\vx) 
				\leq  \left( \int_M \frac{|\vp(\vx)|^2}{|\vx|^a}\, d\mu_{\vPhi}(\vx) \right)^{\frac{1}{2}} \left( \int_{M} \frac{\left| 2 \g^{\top} \vp(\vx) - \vp(\vx) \vH\right|^2}{ |\vx|^{a-2} }\, d\mu_{\vPhi}(\vx) \right)^{\frac{1}{2}}.
			\end{align*}
			Forgetting the second term of the left-hand side, we obtain 
			\begin{align*}
				\int_{M}  (n-a)^2\, \frac{|\vp(\vx)|^2}{|\vx|^a}\, d\mu_{\vPhi}(\vx) \leq \int_{M} \frac{\left| 2 \g^{\top} \vp(\vx) - \vp(\vx) \vH\right|^2}{ |\vx|^{a-2} }\, d\mu_{\vPhi}(\vx) .
			\end{align*}
		\end{proof}
		
		We now consider specific cases of this Hardy inequality.
		
		\begin{corollary}\label{cor:II_xn2}
			We consider integers $d>n\geq 4$ with $n$ even and a closed $n$-dimensional manifold $\Sigma$. There exists a constant $C>0$ depending only on $n$ and $d$ such that the following holds. Let $\vPhi\colon\Sigma\to \R^d$ be a smooth immersion such that $0\in M\coloneqq \vPhi(\Sigma)\subset \B^d(0,1)$. It holds
			\begin{align}\label{eq:integrabilityII}
				\int_{M} \frac{|\vII|^2}{|\vx|^{n-2}}\, d\mu_{\vPhi}(\vx) \leq C\, \Er(\vPhi).
			\end{align}
		\end{corollary}
		\begin{proof}
			We iterate Proposition \ref{pr:Hardy}:
			\begin{align*}
				\int_{M} \frac{|\vII|^2}{|\vx|^{n-2}}\, d\mu_{\vPhi}(\vx) & \leq \frac{1}{4}\, \int_M \frac{|2\g |\vII| +|\vII|\, \vH|^2}{|\vx|^{n-4}}\, d\mu_{\vPhi}(\vx) \\[2mm]
				& \leq \frac{1}{64} \int_M \frac{\left|\g\left[ \left|2\g |\vII| +\big|\vII|\, \vH \right|\right] + |2\g |\vII| +|\vII|\, \vH\big|\, \vH\right|^2}{|\vx|^{n-6}}\, d\mu_{\vPhi}(\vx) \\[2mm]
				& \leq C\, \int_M \frac{\left|\g^2 \vII\right|^2 + |\g\vII|^3 + |\vII|^6 }{|\vx|^{n-6}}\, d\mu_{\vPhi}(\vx) \\[2mm]
				& \leq \cdots \\
				& \leq C\, \Er(\vPhi).
			\end{align*}
		\end{proof}
		
		\begin{corollary}\label{cor:int_xperp}
			Let $\vPhi\colon \Sigma^n\to \R^d$ be a smooth immersion such that $0\in M\coloneqq \vPhi(\Sigma)\subset \B^d(0,1)$. It holds
			\begin{align}\label{eq:int_xperp}
				\int_M \frac{|\vx^{\perp}|^2}{|\vx|^{n+2}}\, d\mu_{\vPhi}(\vx) \leq C\, \Er(\vPhi)^{\frac{1}{2}}.
			\end{align}
		\end{corollary}
		\begin{proof}
			We consider the choices $\vp=1$ and $a=n$ in \eqref{eq:Hardy1} in the proof of \Cref{pr:Hardy}:
			\begin{align*}
				\int_M n \frac{|\vx^{\perp}|^2}{|\vx|^{n+2}}\, d\mu_{\vPhi}(\vx) = \int_M \frac{\vH\cdot \vx}{|\vx|^n}\, d\mu_{\vPhi}(\vx) \leq \int_M \frac{|\vH|}{|\vx|^{\frac{n-2}{2}}} \frac{|\vx^{\perp}|}{|\vx|^{\frac{n+2}{2}}}\, d\mu_{\vPhi}(\vx).
			\end{align*}
			By Hölder inequality, we obtain
			\begin{align*}
				\int_M n \frac{|\vx^{\perp}|^2}{|\vx|^{n+2}}\, d\mu_{\vPhi}(\vx) \leq \left( \int_M \frac{|\vH|^2}{|\vx|^{n-2}}\, d\mu_{\vPhi}(\vx) \right)^{\frac{1}{2}} \left(\int_M \frac{|\vx^{\perp}|^2}{|\vx|^{n+2}}\, d\mu_{\vPhi}(\vx) \right)^{\frac{1}{2}}.
			\end{align*}
			We conclude thanks to \eqref{eq:integrabilityII}.
		\end{proof}
		
		\subsection{Volume growth of extrinsic balls}\label{sec:Balls}
		
		By following \cite[Proposition 2.1]{akutagawa1994}, see also \cite[Proposition 2.4]{carron1996}, we deduce a lower bounds for the volume of extrinsic balls from \Cref{th:Sobolev}. 
		
		\begin{proposition}\label{pr:Lower_Extrinsic}
			Let $d>n\geq 3$ be integers and $\Sigma$ be an $n$-dimensional orientable manifold with boundary.
			There exist $\eps_0>0$ and $c>0$ depending only on $n,d$ such that the following holds. Consider an immersion $\vPhi\in \Imm(\Sigma;\R^d)$ such that $\|\vH\|_{L^{(n,\infty)}\left( \Sigma,g_{\vPhi} \right)}\leq\eps_0$. Then for any $\vx\in \R^d$ and $r>0$, if $\B^d(\vx,r)\cap \dr \vPhi( \Sigma)=\emptyset$, it holds
			\begin{align*}
				\mu_{\vPhi}\Big( \B^d(\vx,r) \Big) \geq c\, r^n.
			\end{align*}
		\end{proposition}
		
		\begin{remark}
			Proposition \ref{pr:Lower_Extrinsic} does not depend on the parity of $n$. It only depends on the Sobolev inequalities in \Cref{th:Sobolev}, valid in all dimensions $n\geq 3$.
		\end{remark}
		
		\begin{proof}
			Let $M\coloneqq \vPhi(\Sigma)$. Consider $\eta\in C^{\infty}_c([0,1];[0,1])$ such that $\eta=1$ in $[0,\frac{1}{2}]$. Fix a ball $\B^d(\vep,r)$ such that $\B^d(\vep,r)\cap \dr M=\emptyset$. We define the function
			\begin{align*}
				\forall x\in \Sigma,\quad \chi(x) = \eta\left(\frac{|\vep-\vPhi(x)|}{2r}\right).
			\end{align*}
			Then we have $0\leq \chi \leq 1$, $\chi = 1$ in $M\cap \B^d(\vep,r)$ and $\chi=0$ on $\R^d\setminus (M\cap \B^d(\vep,2r))$. In particular, we have $\chi=0$ on $\dr M$. By Sobolev inequality in \Cref{th:Sobolev}, we have 
			\begin{align*}
				\mu_{\vPhi}\Big(\B^d(\vep,r)\Big)^{\frac{n-2}{2n}} & \leq \|\chi\|_{L^{\frac{2n}{n-2}}\left(\vPhi^{-1}( \B^d(\vep,2r)),g_{\vPhi}\right)}  \leq C\, \|d\chi\|_{L^2\left( \vPhi^{-1} (\B^d(\vep,2r)),g_{\vPhi} \right)}  \leq \frac{C}{r}\, \left\| \g^{\top}|\vep-\vx| \right\|_{L^2\left( \B^d(\vep,2r),\mu_{\vPhi} \right)}.
			\end{align*}
			We obtain
			\begin{align*}
				\mu_{\vPhi} \Big( \B^d(\vep,r) \Big)^{\frac{n-2}{2n}}\leq \frac{C}{r}\, \mu_{\vPhi}\Big( \B^d(\vep,2r) \Big)^{\frac{1}{2}}.
			\end{align*}
			We write the above estimate as follows:
			\begin{align*}
				\mu_{\vPhi}\Big( \B^d(\vep,r) \Big) \geq C\, \left( \frac{ r}{2} \right)^2 \mu_{\vPhi}\Big(  \B^d(\vep,r/2) \Big)^{\frac{n-2}{n}}.
			\end{align*}
			For simplicity, we denote $f(r):= \mu_{\vPhi}\Big( \B^d(\vep,r) \Big)$. We iterate the above inequality: for any integer $k\geq 1$, it holds
			\begin{align}\label{eq:lower_volume}
				f(r) \geq (C\, r^2)^{\sum_{\ell=0}^{k-1} (\frac{n-2}{n})^{\ell}} 4^{-\sum_{\ell=1}^k \ell(\frac{n-2}{n})^{\ell-1}} f\left(2^{-k} r \right)^{(\frac{n-2}{n})^k }.
			\end{align}
			Since $\vPhi$ is $C^{\infty}$, we have a constant $c>1$ such that for $s>0$ small enough, it holds $c^{-1} s^n < f(s) < c\, s^n$. Hence we have
			\begin{align*}
				f\left(2^{-k} r \right)^{(\frac{n-2}{n})^k } \xrightarrow[k\to +\infty]{} 1.
			\end{align*}
			We also have
			\begin{align*}
				\sum_{\ell=0}^{+\infty} \left(\frac{n-2}{n} \right)^{\ell} = \frac{1}{1-\frac{n-2}{n}} = \frac{n}{2}.
			\end{align*}
			We obtain the result by letting $k\to +\infty$ in \eqref{eq:lower_volume}.
		\end{proof}

		Thanks to the monotonicity formula and Hardy inequality, we now obtain an upper bound for the volume growth of extrinsic balls without any smallness assumption.
		
		\begin{proposition}\label{pr:Extrinsic_Hdiff}
			Let $d>n\geq 4$ be integers with $n$ even and $\Sigma$ be an $n$-dimensional manifold possibly with boundary.
			Let $\vPhi\in \Imm(\Sigma;\R^d)$ and denote $M\coloneqq \vPhi(\Sigma)$. There exists $C>1$ depending only on $n$ such that for any ball $\B^d(\vx,2)$ that does not intersect $\dr M$, it holds
			\begin{align}\label{eq:volume_growth_balls}
				\forall r\in\left(0,1\right),\qquad \mu_{\vPhi}\Big(  \B^d(\vx,r)\Big) \leq r^n \Big( 2\, \mu_{\vPhi}\big(\B^d(\vx,2) \big) + C\, \Er\big(\vPhi;\vPhi^{-1} (\B^d(\vx,2)) \big) \Big).
			\end{align}
			Moreover, if $M\subset \B^d(0,1)$ has no boundary, then it holds
			\begin{align}\label{eq:diameter_energy_volume}
				\vol_{g_{\vPhi}}(\Sigma) \leq C\, \Er(\vPhi;\Sigma).
			\end{align}
		\end{proposition}
		
		\begin{remark}
			\begin{enumerate}
				\item This implies that $\frac{1}{|\cdot-\vep|}\in L^{(n,\infty)}(M)$ for any $\vep\in M$. Indeed, there exists a constant $c>1$ depending only on $n$, $d$ and $\Er(\vPhi)$ such that it holds
				\begin{align}\label{eq:Lninfty}
					\vol_{g_{\vPhi}}\left( \left\{x\in \Sigma : \frac{1}{ \left|\vPhi(x)-\vep \right|}\geq \lambda\right\}\right) = \vol_{g_{\vPhi}}\Big(\vPhi^{-1} (\B^d( \vep,1/\lambda))\Big) \in\left[ \frac{1}{c\, \lambda^n}, \frac{c}{\lambda^n} \right].
				\end{align}
				\item If $M$ has no boundary but dropping the assumption $M\subset \B^d(0,1)$, we have 
				\begin{align*}
					\frac{ \vol_{g_{\vPhi}}(\Sigma) }{ \diam(M)^n } \leq C\, \Er(\vPhi;\Sigma).
				\end{align*}
			\end{enumerate}
		\end{remark}
		
		\begin{proof}
			To shorten the notations, we will drop the index $\vPhi$ and write $\mu = \mu_{\vPhi}$. We start by the estimate \eqref{eq:volume_growth_balls}. Thanks to the monotonicity formula \cite[Equation 17.4]{simon1983}, the following inequality holds for $0<\sigma<\rho<1$:
			\begin{align}
				\frac{\mu(\B^d_{\sigma})}{\sigma^n} & \leq \frac{\mu(\B^d_{\rho})}{\rho^n} - \int_{ \B^d_{\rho}\setminus \B^d_{\sigma}} \frac{|\vx^{\perp}|^2}{|\vx|^{n+2}} d\mu(\vx) + \int_{\B^d_{\rho}\setminus \B^d_{\sigma}} \frac{|\vx^{\perp}|}{|\vx|^{\frac{n+2}{2}}} \frac{|\vH|}{|\vx|^{\frac{n-2}{2}}}\, d\mu(\vx) + \int_{\B^d_{\sigma}} \frac{|\vx^{\perp}|}{\sigma^n} |\vH|\, d\mu(\vx)  \nonumber \\[3mm]
				& \leq \frac{\mu( \B^d_{\rho})}{\rho^n} + \frac{1}{2}\int_{ \B^d_{\rho}} \frac{|\vH|^2}{|\vx|^{n-2}}\, d\mu(\vx) +\int_{\B^d_{\sigma}} \frac{|\vH|}{\sigma^{n-1}}\, d\mu \nonumber \\[3mm]
				& \leq \frac{\mu( \B^d_{\rho})}{\rho^n} + \frac{1}{2}\int_{\B^d_{\rho}} \frac{|\vH|^2}{|\vx|^{n-2}}\, d\mu(\vx) + \left( \frac{\mu(\B^d_{\sigma})}{\sigma^n} \right)^{\frac{n-1}{n}} \|\vH\|_{L^n(\B^d_{\sigma},\mu)}. \label{eq:montonicity_balls}
			\end{align}
			We take $\rho=1$ and $\chi\in C^{\infty}_c(\B^d_2)$ be a cut-off function such that $\chi=1$ in $\B^d_1$. We obtain for any $\sigma<1$:
			\begin{align*}
				\frac{\mu( \B^d_{\sigma})}{\sigma^n} & \leq \mu( \B^d_1) + \int_{\B^d_2} \frac{\chi(x)^2|\vH(x)|^2}{|\vx|^{n-2}}\, d\mu(\vx) + C\|\vH\|_{L^n(\B^d_2,\mu)}^n.
			\end{align*}
			By Proposition \ref{pr:Hardy}, we obtain
			\begin{align*}
				\frac{\mu(\B^d_{\sigma})}{\sigma^n} & \leq \mu(\B^d_1) + \int_{\B^d_2} \frac{|\g^{\top}(\chi \vH)|^2 + \chi^2 |\vH|^4}{|\vx|^{n-4}}\, d\mu(\vx) + C\|\vH\|_{L^n(\B^d_2,\mu)}^n.
			\end{align*}
			By iteration, we obtain for every integer $k$ such that $n-2k\geq 0$,
			\begin{align}
				\frac{\mu( \B^d_{\sigma})}{\sigma^n} & \leq  \mu( \B^d_1)  + C\, \|\vH\|_{L^n( \B^d_2,\mu)}^n \nonumber \\[2mm]
				& + C(n,k)\int_{\B^d_2} \frac{|(\g^{\top})^k(\chi \vH)|^2 + |\vH|^2 |(\g^{\top})^{k-1}(\chi \vH)|^2 + \cdots + \chi^2 |\vH|^{2(k+1)} }{|\vx|^{n-2(k+1)}}\, d\mu(\vx). \label{eq:iteration_density}
			\end{align}
			Here the higher derivative $\g^j\chi$ are defined as follows: for $\vx\in \vPhi(\Sigma)$, 
			\begin{align*}
				(\g^{\top})^j\chi(\vx) = \proj_{T_{\vx}\vPhi(\Sigma)}\left[ \g (\g^{\top})^{j-1}\chi(\vx) \right].
			\end{align*}
			Hence, we have 
			\begin{align*}
				|\g^{\top}\chi| \leq C(n).
			\end{align*} 
			We can estimate the higher derivatives as follows
			\begin{align*}
				|(\g^{\top})^2\chi| & = \left|\proj_{T\vPhi(\Sigma)} \g\left( \proj_{T\vPhi(\Sigma)} \g \right)\chi \right|   \leq \left| \g\left(\proj_{T\vPhi}\right) \g\chi \right| + |\g^2\chi|   \leq C(n)\left( 1 + |\vII_{\vPhi}|_{g_{\vPhi}} \right).
			\end{align*}
			By induction, we have the following pointwise estimate for $j\geq 2$
			\begin{align*}
				\left|(\g^{\top})^j\chi\right| \leq C(n,j) \left( 1 + \sum_{i=0}^{j-2} \big| \g^i \vII_{\vPhi} \big|_{g_{\vPhi}}^{j-1-i} \right) .
			\end{align*}
			Plugging this into \eqref{eq:volume_growth_balls} together with the choice $k=\frac{n}{2}-1$ (so that the term $|x|^{n-2(k+1)}$ becomes $1$), we obtain 
			\begin{align*}
				\frac{\mu( \B^d_{\sigma})}{\sigma^n} & \leq  \mu( \B^d_1) + C\, \|\vH\|_{L^n( \B^d_2,\mu)}^n  + C(n)\int_{\B^d_2} \left|(\g^{\top})^{\frac{n}{2}-1}(\chi \vH)\right|^2 + |\vH|^2 \left|(\g^{\top})^{\frac{n}{2}-2}(\chi \vH) \right|^2 + \cdots + \chi^2 |\vH|^n \ d\mu(\vx) .
            \end{align*}
            Rearranging the terms, we obtain
            \begin{align*} 
				\frac{\mu( \B^d_{\sigma})}{\sigma^n} & \leq \mu( \B^d_1)  + C(n)\int_{\B^d_2} \left|(\g^{\top})^{\frac{n}{2}-1} \vII\right|^2 + |\vII|^2 \left|(\g^{\top})^{\frac{n}{2}-2} \vII \right|^2 + \cdots +  |\vII|^n \ d\mu(\vx)  \\[2mm]
                & \qquad + C(n) \sum_{ \substack{0\leq k \leq \frac{n}{2}-2\\[1mm] 2\leq p\leq n \\[1mm] \frac{1}{p} - \frac{k}{n}>0} } \int_{\B_2^d}  |\g^k \vII|^p\ d\mu(\vx).
			\end{align*} 
			Using Young inequality, we obtain for any $\eta\in(0,1)$ that 
			\begin{align*}
				\frac{\mu( \B^d_{\sigma})}{\sigma^n} & \leq  (1+\eta)\, \mu( \B^d_2) + C(\eta)\, \Er\left(\vPhi;\vPhi^{-1}(\B^d_2) \right).
			\end{align*}
			We obtain \eqref{eq:volume_growth_balls} by taking $\eta=1$.\\
			
			Concerning the estimate \eqref{eq:diameter_energy_volume}, we take the limit $\rho\to +\infty$ in \eqref{eq:montonicity_balls}, with the choice $\B_{\sigma}=\B^d(0,1)$:
			\begin{align*}
				\mu(\R^d) \leq \frac{1}{2}\int_{\R^d} \frac{|\vH|^2}{|\vx|^{n-2}}\, d\mu(\vx) + \mu(\R^d)^{\frac{n-1}{n}} \|\vH\|_{L^n(\R^d,\mu)}.
			\end{align*}
			Again using Hardy inequality and Young's inequality for the second term, we obtain
			\begin{align*}
				\frac{\mu(\R^d)}{2} \leq C\, \Er(\vPhi).
			\end{align*}
		\end{proof}

	\section{Construction of one chart}\label{sec:One_chart}

	In this section, we prove that \underline{for $C^{\infty}$ immersions} $\vPhi\colon \Sigma\to \R^d$ with small energy in a ball $\B^d(\vep,r)\subset \R^d$ and bounded volume, we can construct harmonic coordinates on each connected component of $\vPhi^{-1}(\B^d(\vep,r))$. More precisely, we prove the following result, where all the quantities involved are actually $C^{\infty}$ since we work with $C^{\infty}$-immersions but we only record the estimates needed to pass to the limit in an approximation of a weak immersion by $C^{\infty}$-immersions.
	
	\begin{theorem}\label{th:Construction_chart}
		Let $d>n\geq 4$ be integers with $n$ even and $\Sigma$ a closed orientable $n$-dimensional manifold.
		Let $V>0$. There exists $\eps_0>0$ and $C_0>0$ depending only on $n$, $d$ and $V$ such that the following holds. Let $\eps\in(0,\eps_0)$ and $\vPhi\in \Imm(\Sigma;\R^d)$. Consider $\vep\in \vPhi(\Sigma)$ and $r\in(0,\frac{1}{2})$ such that 
		\begin{align}\label{hyp:Chart}
			\Er\Big( \vPhi ; \vPhi^{-1} \big( \B^d(\vep,r)\big) \Big) \leq \eps^n, \qquad \text{and }\qquad \vol_{g_{\vPhi}}\left( \vPhi^{-1} \big( \B^d(\vep,1)\big) \right) \leq V.
		\end{align}
		Then there exists $\vq\in \B^d(\vep,\frac{r}{100})$ and $\rho\in \left(\frac{r}{4},r\right)$ such that 
		\begin{enumerate}
			\item Every connected component $\Cr$ of $\vPhi^{-1}\big( \B^d(\vq,\rho)\big)\subset \Sigma$ is diffeomorphic to a ball $\B^n$. 
			In particular, the boundary $\Sr  \coloneqq \dr \Cr$ is diffeomorphic to $\s^{n-1}$. 
			\item Each map $\vPhi\colon \Cr\to \B^d(\vq,\rho)$ is injective.
			
			\item There exists $\delta=\delta(n,d)\in(0,1)$ and an affine $n$-plane $\Pc_{\Sr}\subset \R^d$ such that the set $\vPhi(\Sr)$ verifies, for some map $\vphi_{\Sr}\colon \Sc_{\Pc}\coloneq \Pc_{\Sr}\cap \s^{d-1}(\vq,\rho) \to \R^d$, the following relations
			\begin{align*}
				\vPhi(\Sr) = \left\{ \vtheta + \vphi_{\Sr}(\vtheta) : \vtheta \in \Sc_{\Pc} \right\}.
			\end{align*}
			Moreover, it holds 
			\begin{align*}
				\begin{cases} 
					\displaystyle \dist(\vq,\Pc_{\Sr}) \leq \delta\, \rho, \\[3mm]
					\displaystyle r^{-1}\|\vphi_{\Sr}\|_{L^{\infty}(\Sc_{\Pc})} + \|\g \vphi_{\Sr}\|_{L^{\infty}(\Sc_{\Pc})} \leq C_0\, \eps ,\\[3mm]
					\displaystyle r^{\frac{1}{n}}\, \|\g^2 \vphi_{\Sr}\|_{L^n(\Sc_{\Pc})}  \leq C_0.
				\end{cases} 
			\end{align*}
			
			\item\label{it:Coordinates} There exist an open set $\Or\subset \R^n$ diffeomorphic to $\B^n$ and a chart $\vp\colon \Or\to \Cr$ providing harmonic coordinates such that 
			\begin{enumerate}
				\item There exists a diffeomorphism $f\colon \dr \Or\to \s^{n-1}(0,r)$ of class $W^{2,n}(\dr \Or)$ satisfying $\int_{\dr\Or}f =0$ together with the estimates
				\begin{align}\label{eq:bound_domain}
					\|\g f \|_{L^{\infty}(\dr \Or)} + \|\g f^{-1} \|_{L^{\infty}(\s^{n-1}(0,r))} + r^{\frac{1}{n} } \|\g^2 f\|_{L^n(\dr \Or)} + r^{\frac{1}{n}} \|\g^2 f^{-1} \|_{L^n(\s^{n-1}(0,r))} \leq C_0.
				\end{align}
				
				\item The map $f\colon \dr \Or \to \s^{n-1}(0,r)$ extends to a bi-Lipschitz homeomorphism $f\colon \Or\to \B^n(0,r)$ with the estimates
				
				\begin{align}
					& \|\g f\|_{L^{\infty}(\Or)} + \| \g f^{-1} \|_{L^{\infty}(\B^n(0,r))} \leq C_0\, \left( \|g_{\vPhi}\|_{L^{\infty}(\Cr)} + 1 \right), \nonumber\\[2mm]
					& \| f^{-1} \|_{W^{1,n}(\B^n(0,r))} \leq C(n)\, \|g_{\vPhi}\|_{L^n(\Cr)}. \label{eq:interior_domain}
				\end{align}
				
				\item In the chart $\vp$, the coefficients $\left( g_{\vPhi\circ\vp} \right)_{ij}$ verify for all $1\leq i,j,k,l\leq n$
				\begin{align}\label{eq:bound_coordinates}
					\left\| \left( g_{\vPhi\circ \vp} \right)_{ij} - \delta_{ij} \right\|_{L^{\infty}(\Or,\geu)} + \left\|\dr_k \left( g_{\vPhi\circ \vp} \right)_{ij} \right\|_{L^{(n,1)}(\Or,\geu)} + \left\| \dr^2_{kl}\left( g_{\vPhi\circ \vp} \right)_{ij} \right\|_{L^{\left(\frac{n}{2},1\right)}(\Or,\geu)} \leq C_0\, \eps.
				\end{align}
			\end{enumerate}
		\end{enumerate}
	\end{theorem}

	\subsubsection*{Ideas of the proof.}
	
	Let $M\coloneqq \vPhi(\Sigma)$. We find a good slice $M\cap \s^{d-1}(\vq,\rho)$ such that $\vII_{\vPhi}\in L^n(M\cap \s^{d-1}(\vq,\rho))$ and such that the second fundamental form $\vA$ of the immersion $\vPhi_{\rho} \coloneq \vPhi\big|_{\vPhi^{-1}(\s^{d-1}(\vq,\rho))}$ lies in $L^n$. Under this setting, we prove that the image of every connected component $\Sr\subset \Sigma$ of $\vPhi^{-1}(\s^{d-1}(\vq,\rho))$ by $\vPhi_{\rho}$ is locally a graph over the projection $\Sc_{\Sr}$ of $\vPhi_{\rho}(\Sr)$ on a given affine $n$-dimensional plane $\Pc_{\Sr}$. We obtain a covering map $\iota\colon \Sr\to \s^{n-1}$ by associating each $x\in \Sr$ the closest point to $\vPhi_{\rho}(x)$ in $\Sc_{\Pc}\coloneqq \Pc_{\Sr}\cap \s^{d-1}(\vq,\rho)$, which has to be a bijection since $\s^{n-1}$ is simply connected in dimension $n-1\geq 2$. Therefore, we obtain that $\Sr$ is diffeomorphic to a sphere and $\vPhi_{\rho}\big|_{\Sr}$ is injective. Then, we extend $\vPhi(\Sr)$ to a whole graph over $\R^n\setminus \s^{n-1}$ with control on the second fundamental form. We deduce that every connected component of $M\cap \B^d(\vq,\rho)$ can be extended in a simply connected immersed manifold with small $L^n$-norm of the second fundamental form and flat outside of a compact set. Adapting the ideas of Chern--Lashof \cite{chern1958} (see also \cite{ferus1968}) and using Morse theory, we prove that each connected component $\Cr$ of $\vPhi^{-1}\big(\B^d(\vq,\rho)\big)$ is a diffeomorphic to a ball $\B^n$. \Cref{it:Coordinates} will be a direct consequence of \cite{MarRiv2025}.\\

	\subsubsection*{Setting and notations.}
	
	In the rest of the section, we fix integers $d>n\geq 4$ with $n$ even. We let $\Sigma$ to be a closed orientable $n$-dimensional manifold and we let $\vPhi\colon \Sigma^n\to \R^d$ be a $C^{\infty}$ immersion. We denote $M\coloneqq \vPhi(\Sigma)$ and $\vn_{\vPhi}$ the generalized Gauss map of $\vPhi$ given, in local coordinates, by the following formula:
	\begin{align*}
		\vn_{\vPhi} \coloneqq *_{\R^d} \frac{\dr_1\vPhi\wedge \cdots \wedge \dr_n \vPhi}{|\dr_1\vPhi\wedge \cdots \wedge \dr_n\vPhi|} .
	\end{align*}
	We will often consider a basis of the vector space spanned by $\vn_{\vPhi}$. We will denote $(\vn_1,\ldots,\vn_{d-n})$ a direct orthonormal basis of $\Span(\vn_{\vPhi})$:
	\begin{align*}
		\vn_{\vPhi} = \vn_1\wedge \cdots \wedge \vn_{d-n}.
	\end{align*}
	The good slice $\vPhi^{-1}\big( \s^{d-1}(\vq,\rho)\big)$ will be denoted $S_{\rho}\subset \Sigma$. A connected component of $S_{\rho}$ will be denoted $\Sr\subset S_{\rho}\subset \Sigma$. The projection of $\vPhi(\Sr)$ on the plane $\Pc_{\Sr}$ will be denoted $\Sc_{\Sr}$. We will denote $\Pc_{\Sr}$ the affine plane and $\Pc$ the associated vectorial plane.\\

	\subsubsection*{Structure.}
	
	In \Cref{sec:Formulas}, we compute the second fundamental form of a slice and the second fundamental form of an immersion in a round sphere written as a graph. In \Cref{sec:slice}, we construct a good slice $S_{\rho}\subset \Sigma$ for which we have a bound on the second fundamental form $\vA$ of the slice seen as a submanifold of a round sphere and smallness of the restriction of $\vII_{\vPhi}$ to this slice. In \Cref{sec:graph_good_slice}, we show that for every connected component $\Sr\subset S_{\rho}$, the immersion $\vPhi\big|_{\Sr}$ is given by the graph of a function bounded $W^{2,n}$ and small in $W^{1,\infty}$. In \Cref{sec:extension}, we extend each of these graphs and glue them to flat affine $n$-planes. In \Cref{sec:Topology}, we prove that each connected component of the resulting manifold is homeomorphic to a sphere, and that the interior of the good slice is diffeomorphic to a ball and conclude the proof of \Cref{th:Construction_chart}.

	\subsection{Second fundamental form of a slice}\label{sec:Formulas}
	
	In this section, we compute the second fundamental form of a slice $M\cap \s^{d-1}(\vq,\rho)\subset \s^{d-1}(\vq,\rho)$, under the assumption that all the quantities involved are well-defined in Lemma \ref{lm:II_slice}. We also compute the second fundamental form of a submanifold of $\s^{d-1}$ written as a graph in Lemma \ref{lm:A_graph}.
	
	\begin{lemma}\label{lm:II_slice}
		With the notations of \Cref{th:Construction_chart}, we denote $M\coloneqq \vPhi(\Sigma)$. The second fundamental form $\vA$ of $M\cap \s^{d-1}(\vq,\rho)$ seen as a submanifold of $\s^{d-1}(\vq,\rho)$ is given, in coordinates as follows. Let $\vnu = \vnu_1\wedge \cdots \wedge \vnu_{d-n}$ be the generalized Gauss map of $\vPhi\big|_{\vPhi^{-1}(\s^{d-1}(\vq,\rho))}$. In local coordinates on $\vPhi^{-1}(\s^{d-1}(\vq,\rho))$, we denote
		\begin{align*}
			\vA_{ij} = \sum_{\alpha=1}^{d-n} A^{\alpha}_{ij}\, \vnu_{\alpha}.
		\end{align*}
		Using the decomposition $A_{ij}^{\alpha} = \Arond_{ij}^{\alpha} + h^{\alpha}\, ( g_{\vPhi} )_{ij}$ in traceless part and mean curvature, it holds
		\begin{align}\label{eq:Arond_h}
			\begin{cases} 
				\displaystyle \Arond_{ij}^{\alpha} = \frac{\rho^2}{|\rho^2\, \vn_{\alpha} - (\vn_{\alpha}\cdot (\vPhi-\vq))(\vPhi-\vq)|} \IIr_{ij}^{\alpha}, \\[5mm]
				\displaystyle h^{\alpha} = \frac{\rho^2\, H_{\Phi}^{\alpha} - (\vn_{\alpha}\cdot (\vPhi-\vq))}{ | \rho^2 \vn_{\alpha} - (\vn_{\alpha}\cdot (\vPhi-\vq))(\vPhi-\vq)|  }.
			\end{cases} 
		\end{align}
		Moreover, the full norm of $\vA$ satisfies
		\begin{align}\label{eq:II_slice}
			\left|\vA\right|_{g_{\vPhi}} \leq (2+\sqrt{n})\, \frac{  \rho^2\ |\vII|_{g_{\vPhi}} }{ |\big(\vn_{\vPhi}\wedge (\vPhi -\vq)\big) \llcorner (\vPhi -\vq)| } + \sqrt{n}\, \frac{ |\vn_{\vPhi}\llcorner (\vPhi -\vq)|}{ |\big(\vn_{\vPhi}\wedge (\vPhi -\vq) \big) \llcorner (\vPhi -\vq)| }.
		\end{align}
	\end{lemma}
	
	\begin{proof}
		For simplicity, we assume that $\vq=0$. We denote $\s^{d-1}_{\rho}= \s^{d-1}(0,\rho)$ and $S_{\rho} \coloneqq \vPhi^{-1}\left( \s^{d-1}_{\rho}\right)$. We assume that $\vPhi\colon S_{\rho} \to \s^{d-1}_{\rho}$ is an immersion. 
		Its generalized Gauss map $\vnu\colon S_{\rho} \to \Lambda^{d-n} (T_{\vPhi}\, \s^{d-1}_{\rho})$ is given by\footnote{We check the orientation. If $\theta_1,\ldots,\theta_{n-1}$ are coordinates for the slice $S\coloneqq \vPhi^{-1}(\s^{d-1}(\vq,\rho))$ and $(\vPhi^{\top},\dr_{\theta_1}\vPhi,\ldots,\dr_{\theta_{n-1}}\vPhi)$ is a direct local basis of the tangent space of $\vPhi(\Sigma)$, then $( \dr_{\theta_1}\vPhi,\ldots,\dr_{\theta_{n-1}} \vPhi,\vnu )$ is a direct local basis of the tangent space of $\s^{d-1}(\vq,\rho)$. Hence, both $\left( \vPhi^{\top},\dr_{\theta_1}\vPhi,\ldots,\dr_{\theta_{n-1}} \vPhi,\vn_{\vPhi} \right)$ and $\left(\vPhi,\dr_{\theta_1}\vPhi,\ldots,\dr_{\theta_{n-1}}\vPhi,\vnu \right)$ are direct basis of $\R^d$. Thus we have 
			\begin{align*}
				\frac{ \vPhi^{\top} \wedge \dr_{\theta_1} \vPhi \wedge \cdots \wedge \dr_{\theta_{n-1}} \vPhi \wedge \vn_{\vPhi} }{\left| \vPhi^{\top}\wedge \dr_{\theta_1} \vPhi \wedge \cdots \wedge \dr_{\theta_{n-1}} \vPhi \wedge  \vn_{\vPhi} \right| } & = \frac{\vPhi\wedge \dr_{\theta_1} \vPhi \wedge \cdots \wedge \dr_{\theta_{n-1}} \vPhi \wedge \vnu  }{\left|\vPhi\wedge  \dr_{\theta_1} \vPhi \wedge \cdots \wedge \dr_{\theta_{n-1}} \vPhi \wedge \vnu  \right| }  .
			\end{align*}
			Therefore, we have 
			\begin{align*}
				\frac{ \vPhi^{\top}\wedge \vn_{\vPhi} }{ \left| \vPhi^{\top}\wedge \vn_{\vPhi} \right| }= \frac{ \vPhi \wedge \vnu }{ \left|  \vPhi \wedge \vnu \right|} .
			\end{align*}
			Since $\vnu\llcorner \vPhi=0$ (we have $\vPhi$ orthogonal to $\s^{d-1}(0,\rho)$), we obtain
			\begin{align*}
				\frac{ |\vPhi|^2 \vnu }{ \left|  \vPhi \wedge \vnu \right| } = \frac{ \left( \vPhi^{\top}\wedge \vn_{\vPhi} \right) \llcorner \vPhi}{\left| \vPhi^{\top}\wedge \vn_{\vPhi} \right| }.
			\end{align*}
		}
		\begin{align}\label{eq:def_vnu}
			\vnu =  \frac{ \big(\vPhi \wedge \vn_{\vPhi} \big) \llcorner \vPhi}{ |\big(\vPhi \wedge \vn_{\vPhi} \big)\llcorner \vPhi| }.
		\end{align}
		If $(\theta_1,\ldots,\theta_{n-1})$ are coordinates on $S_{\rho}$. The second fundamental form of $\vPhi\colon S_\rho\to \R^d$ is given by the following expression:
		\begin{align*}
			\dr_{\theta_i} \vnu =\frac{\big(\vPhi\wedge \dr_{\theta_i} \vn_{\vPhi}\big) \llcorner \vPhi + \big(\dr_{\theta_i} \vPhi \wedge \vn_{\vPhi} \big) \llcorner \vPhi + \big(\vPhi \wedge \vn_{\vPhi} \big) \llcorner \dr_{\theta_i} \vPhi}{ |\big(\vPhi \wedge \vn_{\vPhi} \big) \llcorner \vPhi|} - \frac{\big(\vPhi \wedge \vn_{\vPhi}\big) \llcorner \vPhi}{|\big( \vPhi \wedge \vn_{\vPhi} \big) \llcorner \vPhi|^2}\, \dr_{\theta_i} \Big( \left| \big(\vPhi \wedge \vn_{\vPhi}\big) \llcorner \vPhi \right| \Big).
		\end{align*}
		Since all the derivatives above are taken along $S_{\rho}$, we obtain 
		\begin{align*}
			\dr_{\theta_i} \vPhi\cdot \vPhi =  \dr_{\theta_i} \left(\frac{|\vPhi|^2}{2}\right) =0.
		\end{align*}
		Hence, we have 
		\begin{align*}
			\dr_{\theta_i} \vnu =\frac{\big(\vPhi\wedge \dr_{\theta_i} \vn_{\vPhi} \big) \llcorner \vPhi - \dr_{\theta_i} \vPhi\wedge (\vn_{\vPhi}\llcorner \vPhi) }{ |\big(\vPhi \wedge \vn_{\vPhi}\big) \llcorner \vPhi|} - \frac{\big(\vPhi \wedge \vn_{\vPhi} \big) \llcorner \vPhi}{|\big(\vPhi \wedge \vn_{\vPhi} \big) \llcorner \vPhi|^3}\, \scal{\big(\vPhi \wedge \vn_{\vPhi} \big)\llcorner \vPhi}{ \big(\vPhi\wedge \dr_{\theta_i} \vn_{\vPhi}\big) \llcorner \vPhi -(\vPhi\wedge \vn_{\vPhi})\wedge \dr_{\theta_i} \vPhi }.
		\end{align*}
		The second fundamental form of $\vPhi\colon S_\rho\to \s^{d-1}_{\rho}$ is given by 
		\begin{align*}
			\vA_{ij} = \proj_{\vnu}\left( \frac{\dr^2 \vPhi}{\dr \theta_i\, \dr \theta_j}\right).
		\end{align*}
		We consider $\vnu_1,\ldots,\vnu_{d-n}$ a basis of the normal space represented by $\vnu$ in $T_{\vPhi}\, \s^{d-1}_{\rho}$ and $\vn_1,\ldots,\vn_{d-n}$  basis of the normal space represented by $\vn_{\Phi}$:
		\begin{align}\label{eq:def_vnu_alpha}
			\begin{cases} 
				\displaystyle \vnu = \vnu_1\wedge \cdots \wedge \vnu_{d-n} \in \Lambda^{d-n} (T_{\vPhi}\, \s^{d-1}_{\rho}),\\[3mm]
				\displaystyle \vnu_{\alpha} = \frac{\rho^2\, \vn_{\alpha} -(\vn_{\alpha}\cdot \vPhi)\, \vPhi}{|\rho^2\, \vn_{\alpha} -(\vn_{\alpha}\cdot \vPhi)\, \vPhi|}.
			\end{cases}
		\end{align}
		We have the following expression of $\vA$:
		\begin{align*}
			\vA_{ij} = \sum_{1\leq \alpha\leq d-n} \left( \vnu_{\alpha} \cdot \frac{\dr^2 \vPhi}{\dr\theta_i\, \dr\theta_j} \right)\ \vnu_{\alpha} = - \sum_{1\leq \alpha\leq d-n} \left( \frac{ \dr \vnu_{\alpha}}{\dr \theta_i} \cdot \frac{\dr \vPhi}{ \dr\theta_j} \right)\ \vnu_{\alpha}.
		\end{align*}
		Each coefficient satisfies
		\begin{align*}
			\frac{ \dr \vnu_{\alpha}}{\dr \theta_i} \cdot \frac{\dr \vPhi}{ \dr\theta_j}  & = \Bigg( \frac{ \rho^2\, \dr_i \vn_{\alpha} -(\dr_i \vn_{\alpha}\cdot \vPhi)\, \vPhi -(\vn_{\alpha}\cdot \vPhi)\, \dr_i \vPhi }{| \rho^2\, \vn_{\alpha} - (\vn_{\alpha}\cdot \vPhi)\, \vPhi| } \Bigg) \cdot \dr_j \vPhi \\[3mm]
			& = \frac{\rho^2\, \vII^{\, \alpha}_{ij} - (\vn_{\alpha}\cdot \vPhi)\, g_{ij}}{ | \rho^2\, \vn_{\alpha} - (\vn_{\alpha}\cdot \vPhi)\, \vPhi|  } \\[3mm]
			& = \frac{\rho^2}{ | \rho^2\, \vn_{\alpha} - (\vn_{\alpha}\cdot  \vPhi)\, \vPhi|  }\, \IIr_{ij}^{\alpha}  + \frac{\rho^2\, H_{\Phi}^{\alpha} - (\vn_{\alpha}\cdot \vPhi)}{ | \rho^2\, \vn_{\alpha} - (\vn_{\alpha}\cdot \vPhi)\, \vPhi|  }\, g_{ij}.
		\end{align*}
		Hence, the full norm of $\vA$ satisfies
		\begin{align*}
			\left|\vA\right|_{g_{\vPhi}} \leq (2+\sqrt{n})\, \rho^2 \frac{ |\vII|_{g_{\vPhi}} }{ |\big(\vn_{\vPhi}\wedge \vPhi \big) \llcorner \vPhi| } + \sqrt{n}\, \frac{ |\vn_{\vPhi}\, \llcorner\, \vPhi|}{ |\big(\vn_{\vPhi}\wedge \vPhi \big) \llcorner \vPhi| }.
		\end{align*}
	\end{proof}
	
	We will need to understand the relation between the second fundamental form of the graph of a function $u$ in $\s^{d-1}$ and the Hessian of $u$.
	
	\begin{lemma}\label{lm:A_graph}
		If $\vf\colon \B^{n-1}_s\to \s^{d-1}_{\rho}$ is an immersion of the form $\vf(x)=(x,\vu(x))$, then we have the pointwise estimate
		\begin{align*}
			\left|\g^2\, \vu\right| \leq \left( 1+ |\g \vu|^2 \right)^{\frac{3}{2}} \left|\vA_{\vf} \right|_{g_{\vf}} + \rho^{-1}\, |\g \vu|^2.
		\end{align*}
		Hence, for any $p\geq 1$, it holds
		\begin{align*}
			\left\|\g^2 \vu \right\|_{L^p \left(\B^{n-1}_s,\geu \right)} & \leq C(p,n)\Bigg[ \left( 1+\|\g \vu\|_{L^{\infty}\left(\B^{n-1}_s,\geu\right)}^3 \right) \left\|\vA_{\vf}\, \right\|_{L^p\left(\B^{n-1}_s,g_{\vf}\right)}  + s^{\frac{n-1}{p}}\, \rho^{-1}\, \|\g \vu\|_{L^{\infty}\left(\B^{n-1}_s,\geu\right)}^2 \Bigg].
		\end{align*}
	\end{lemma}
	
	\begin{proof}
		We denote $\ve_1,\ldots,\ve_{n-1}$ the canonical basis of $\R^{n-1}$. Since $\vf$ is value in a sphere, the tangent space of $\vf$ is the same in $\R^d$ or in $\s^{d-1}$ and is given by
		\begin{align*}
			T\, \vf(\B^{n-1}) = \Span \left[ \begin{pmatrix}
				\ve_1\\[1mm] \dr_1 \vu
			\end{pmatrix},\ldots , \begin{pmatrix}
				\ve_{n-1}\\[1mm] \dr_{n-1} \vu
			\end{pmatrix} \right].
		\end{align*}
		Hence, the induced metric is given by
		\begin{align*}
			\begin{cases} 
				\displaystyle g_{ij} = \delta_{ij} + \dr_i \vu \cdot \dr_j \vu, \\[3mm]
				\displaystyle \sqrt{\det g} = \sqrt{\det \Big(\Id_{n-1} + (d\vu\, )^{\top} (d\vu\, ) \Big)} \geq 1.
			\end{cases}
		\end{align*}
		We consider the orthogonal projection on the normal space of $\vf(\B_s^{n-1})$:
		\begin{align*}
			\Pi \coloneqq \Id_{\R^d} - g^{\alpha\beta}\,(\dr_{\alpha} \vf\, )\otimes  (\dr_{\beta} \vf\, ) \colon \B^{n-1} \to \R^{d\times d}.
		\end{align*}
		The second fundamental form $\vA_{\vf}$ of $\vf(\B^{n-1})\subset \s^{d-1}_{\rho}$ is given by
		\begin{align*}
			\vA_{ij} & = \dr^2_{ij} \vf - g^{\alpha\beta}\, \left[ (\dr^2_{ij} \vf) \cdot (\dr_{\alpha} \vf)\right]\, (\dr_{\beta} \vf) - \rho^{-2}\left[ (\dr^2_{ij} \vf)\cdot \vf\right]\, \vf   = \Pi\, \begin{pmatrix}
				0\\ \dr^2_{ij}\, \vu
			\end{pmatrix} -\rho^{-2} (\dr_i \vu\cdot \dr_j \vu)\, \vf.
		\end{align*}
		For any vector $\vv\in \R^{d+1-n}$, we have
		\begin{align*}
			\left|\Pi \begin{pmatrix}
				0 \\ \vv
			\end{pmatrix} \right|^2 & = \left|\begin{pmatrix}
				0\\ \vv
			\end{pmatrix} - g^{\alpha\beta}\, (\dr_{\beta} \vu \cdot \vv) \begin{pmatrix}
				\ve_{\alpha}\\ \dr_{\alpha} \vu
			\end{pmatrix} \right|^2  = \left| (\dr_{\beta} \vu\cdot \vv)\, g^{\alpha\beta} \ve_{\alpha} \right|^2 + \left|\vv-g^{\alpha\beta}\, (\dr_{\beta} \vu \cdot \vv)\, \dr_{\alpha} \vu \right|^2  \geq \frac{|\vv|^2}{1+|\g \vu|^2}.
		\end{align*}
		We obtain the last estimate by considering the two cases $\vv\in \Span(d\vu)$ and $\vv\in \Span(d\vu)^{\perp}$. Hence, we obtain the following pointwise estimate
		\begin{align*}
			\left|\g^2\, \vu\right| \leq \left( 1+ |\g \vu|^2 \right)^{\frac{3}{2}} \left|\vA_{\vf} \right|_{g_{\vf}} + \rho^{-1}\, |\g \vu|^2.
		\end{align*}
		If now $\vA_{\vf}\in L^p(\B^{n-1}_s,g_{\vf})$, we obtain 
		\begin{align*}
			\int_{\B^{n-1}_s} |\g^2 \vu |^p\, dx & \leq 2^p \int_{\B^{n-1}_s} \left( 1+ |\g \vu|^2\right)^{\frac{3p}{2}} \left|\vA_{\vf}\right|^p_{g_{\vf}}\, dx + 2^p \rho^{-p} \int_{\B^{n-1}_s} |\g \vu|^{2p}\, dx \\[3mm]
			& \leq 2^p \left( 1+ \|\g \vu\|_{L^{\infty}(\B^{n-1})}^2 \right)^{\frac{3p}{2}} \int_{\B^{n-1}_s} \left|\vA_{\vf} \right|_{g_{\vf}}^p\, d\vol_{g_{\vf}}+ 2^p\, \Hr^{n-1}(\B^{n-1})\, s^{n-1}\, \rho^{-p} \|\g \vu\|_{L^{\infty}(\B^{n-1})}^{2p}.
		\end{align*}
	\end{proof}
	
	\subsection{Construction of the good slice}\label{sec:slice}
	
	The goal of this section is to prove Proposition \ref{pr:est_good_slice} below, namely that if the energy on an extrinsic ball is small enough, then we can find a good slice $\Sr\subset \Sigma$ such that $\vPhi\big|_{\Sr}$ is valued into a round sphere $\s^{d-1}$ with a control of the induced second fundamental form $\vA$, of the restriction of $\vII_{\vPhi}$ and of the volume of $\Sr$. Since we work with $C^{\infty}$-immersions, the slice is actually a $C^{\infty}$-manifold, hence all the usual Sobolev inequalities are valid, but we record only the estimates required to pass to the limit in an approximation of a weak immersion by $C^{\infty}$-immersions.

	\begin{proposition}\label{pr:est_good_slice}
		With the notations of \Cref{th:Construction_chart}, let $M\coloneq \vPhi(\Sigma)$. Fix a point $\vep_0\in M$ and $r>0$. We denote $Q_r\coloneqq \vPhi^{-1}(\B^d(\vep_0,2r))$. Assume that 
		\begin{align*}
			\Er(\vPhi;Q_r)^{\frac{1}{n}} \leq \eps, \qquad \text{and } \qquad \vol_{g_{\vPhi}}(Q_1) \leq V.
		\end{align*}
		Then there exist $\vq\in \B^d\left( \vep_0, \frac{r}{100}\right)$ and $\rho\in\left( \frac{6r}{10},\frac{9r}{10} \right)$ such that the following holds. The set $S_{\rho} \coloneqq \vPhi^{-1}(\s^{d-1}(\vq,\rho))$ is an $(n-1)$-dimensional manifold and $\vPhi\colon S_{\rho} \to \s^{d-1}(\vq,\rho)$ is a weak immersion whose second fundamental for $\vA$ satisfies
		\begin{align*}
			r^{\frac{1}{n}} \left\|\vA\right\|_{L^n(S_{\rho},g_{\vPhi})} + \frac{r^{\frac{1}{n}}}{\eps} \left\|\vII_{\vPhi}\right\|_{L^n\left(S_{\rho},g_{\vPhi}\right)} + \frac{\vol_{g_{\vPhi}}(S_{\rho})}{r^{n-1}} \leq C(n,d,V).
		\end{align*}
	\end{proposition}
	
	\begin{remark}
		As a consequence of these estimates, we obtain that the second fundamental form of $M\cap \s^{d-1}(\vq,\rho)\subset \s^{d-1}(\vq,\rho)$ is also in $L^{n-1}$ thanks to Hölder inequality and the control of the volume.
	\end{remark}
	
	For simplicity, we will assume that $\vep_0=0$ in the whole section. In order to choose a good slice $\s^{d-1}(\vq,\rho)$, we consider the following estimate.
	\begin{claim}\label{cl:est_above_slice}
		Given $1>r>0$, it holds
		\begin{align*}
			& \fint_{\B^d\left(0,\frac{r}{100}\right)} \Bigg( \int_{\vPhi^{-1}\left( \B^d(\vq,\frac{9r}{10})\setminus \B^d(\vq,\frac{6r}{10})\right)} \left|d|\vPhi-\vq|\right|_{g_{\vPhi}}\left[ \frac{|\vPhi-\vq|^{2n}\ |\vII|^n_{g_{\vPhi}}}{\left|\big(\vn_{\vPhi}\wedge (\vPhi-\vq)\big)\llcorner (\vPhi-\vq) \right|^n\ \Er(\vPhi;Q_r)}\right. \\[3mm]
			& \qquad \left. + \frac{|\vPhi-\vq|^n}{\left|(\vn_{\vPhi}\wedge(\vPhi-\vq))\llcorner (\vPhi-\vq) \right|^n} + \frac{1}{r^n} \right]\, d\vol_{g_{\vPhi}}\Bigg)\, d\vq
			\leq C(n,d,V).
		\end{align*}
	\end{claim}
	
	\begin{proof} 
		We first enlarge the domain of integrate of the second integral in order to have a domain of integration independent of $\vq$:
		\begin{align*}
			&  \fint_{\B^d\left(0,\frac{r}{100}\right)} \Bigg( \int_{\vPhi^{-1}\left( \B^d(\vq,\frac{9r}{10})\setminus \B^d(\vq,\frac{6r}{10})\right)} \left|d|\vPhi-\vq|\right|_{g_{\vPhi}}\Bigg[ \frac{|\vPhi-\vq|^{2n}\ |\vII|^n_{g_{\vPhi}}}{ |\big(\vn_{\vPhi}\wedge (\vPhi-\vq)\big)\llcorner (\vPhi-\vq) |^n\ \Er(\vPhi;Q_r)} \\[3mm]
			& \qquad + \frac{|\vPhi-\vq|^n}{|(\vn_{\vPhi}\wedge(\vPhi-\vq))\llcorner (\vPhi-\vq)|^n} +\frac{1}{r^n} \Bigg]\, d\vol_{g_{\vPhi}}\Bigg)\, d\vq \\[3mm]
			\leq & \fint_{\B^d\left(0,\frac{r}{100}\right)} \Bigg( \int_{\vPhi^{-1}\left( \B^d(0,r)\setminus \B^d(0,\frac{r}{2})\right)}  \left|d|\vPhi-\vq|\right|_{g_{\vPhi}}\Bigg[ \frac{|\vPhi-\vq|^{2n}\ |\vII|^n_{g_{\vPhi}} }{|\big(\vn_{\vPhi}\wedge (\vPhi-\vq)\big)\llcorner (\vPhi-\vq) |^n\ \Er(\vPhi;Q_r)} \\[3mm]
			& \qquad + \frac{|\vPhi-\vq|^n}{|(\vn_{\vPhi}\wedge(\vPhi-\vq))\llcorner (\vPhi-\vq)|^n} +\frac{1}{r^n} \Bigg]\, d\vol_{g_{\vPhi}}\Bigg)\, d\vq .
		\end{align*}
		We now exchange the integrals on the right-hand side:
		\begin{align}
			&\fint_{\B^d\left(0,\frac{r}{100}\right)} \Bigg( \int_{\vPhi^{-1}\left( \B^d(\vq,\frac{9r}{10})\setminus \B^d(\vq,\frac{6r}{10})\right)} \left|d|\vPhi-\vq|\right|_{g_{\vPhi}}\Bigg[ \frac{|\vPhi-\vq|^{2n}\ |\vII|^n_{g_{\vPhi}} }{|\big(\vn_{\vPhi}\wedge (\vPhi-\vq)\big)\llcorner (\vPhi-\vq) |^n\ \Er(\vPhi;Q_r)} \nonumber \\[3mm]
			& \qquad + \frac{|\vPhi-\vq|^n}{|(\vn_{\vPhi}\wedge(\vPhi-\vq))\llcorner (\vPhi-\vq)|^n}+  \frac{1}{r^n} \Bigg]\, d\vol_{g_{\vPhi}}\Bigg)\, d\vq  \nonumber \\[3mm]
			\leq & \int_{\vPhi^{-1}\left( \B^d(0,r)\setminus\B^d(0,\frac{r}{2})\right)}  \Bigg[ \frac{ |\vII|^n_{g_{\vPhi}} }{ \Er(\vPhi;Q_r)}\ \left( \fint_{\B^d\left(0,\frac{r}{100}\right)} \frac{|d|\vPhi-\vq||_{g_{\vPhi}}\ |\vPhi-\vq|^{2n}}{ |\big(\vn_{\vPhi}\wedge (\vPhi-\vq)\big)\llcorner (\vPhi-\vq) |^n } \ d\vq\right) \label{eq:Int11}\\[3mm]
			& \qquad + \left( \fint_{\B^d\left(0,\frac{r}{100}\right)} \frac{|d|\vPhi-\vq||_{g_{\vPhi}}\ |\vPhi-\vq|^n}{|(\vn_{\vPhi}\wedge(\vPhi-\vq))\llcorner (\vPhi-\vq)|^n}\ d\vq\right)  +\left(\fint_{\B^d\left(0,\frac{r}{100}\right)} \frac{d\vq}{r^n}\right)\Bigg]\, d\vol_{g_{\vPhi}} \label{eq:Int12} .
		\end{align}
		We write the term $|d|\vPhi-\vq||_{g_{\vPhi}}$ as follows:
		\begin{align*}
			\left|d|\vPhi-\vq|\right|_{g_{\vPhi}} = \frac{ \left| \scal{\vPhi-\vq}{d\vPhi} \right|_{g_{\vPhi}} }{|\vPhi-\vq|} = \frac{ \left| \vn_{\vPhi}\wedge (\vPhi-\vq) \right|}{|\vPhi-\vq|}.
		\end{align*}
		We estimate the first integral of \eqref{eq:Int12}. By change of variables, we obtain for any point $\vPhi\in \B^d(0,r)\setminus \B^d(0,\frac{r}{2})$ and $|\vn|=1$ that
		\begin{align}\label{eq:Int121}
			\fint_{\B^d\left(0,\frac{r}{100}\right)}\frac{|\vn\wedge(\vPhi-\vq)|\ |\vPhi-\vq|^{n-1}}{|(\vn\wedge(\vPhi-\vq))\, \llcorner\, (\vPhi-\vq)|^n} \ d\vq & \leq C(n,d)\, r^{n-1-d} \int_{\B^d\left(\vPhi,\frac{r}{100}\right)} \frac{|\vn\wedge \vep|}{|(\vn\wedge \vep) \, \llcorner\,  \vep|^n} \, d\vep.
		\end{align}
		We now estimate the denominator of the left-hand side\footnote{In codimension 1, we have assuming $\vn=\ve_1$,
			\begin{align*}
				|(\vn\wedge \vep)\llcorner \vep|^2  = \left| |\vep|^2 \vn - \vep_1\, \vep\right|^2 = \left(|\vep|^2 - \vep_1^{\, 2}\right)^2+ \vep_1^{\, 2}\left(\vep_2^{\, 2} +\cdots + \vep_d^{\, 2}\right) = \left(\vep_2^{\, 2} +\cdots + \vep_d^{\, 2}\right)^2+ \vep_1^{\, 2}\left(\vep_2^{\, 2} +\cdots + \vep_d^{\, 2}\right) = |\vep|^2\, \left(\vep_2^{\, 2} +\cdots + \vep_d^{\, 2}\right) .
			\end{align*}
		}. Let $(\ve_1,\ldots,\ve_d)$ the canonical basis of $\R^d$. Up to rotation, we can assume that $\vn=\ve_1\wedge \cdots \wedge \ve_{d-n}$. We obtain
		\begin{align*}
			|\vn\wedge \vep| = \left(\vep_{d-n+1}^{\, 2} + \cdots + \vep_d^{\, 2}\right)^{\frac{1}{2}}.
		\end{align*}
		Moreover, we have
		\begin{align*}
			\left(\ve_1\wedge \cdots \wedge \ve_{d-n}\wedge \vep \right)\llcorner \vep & = \left(\ve_1\wedge \cdots \wedge \ve_{d-n}\wedge \left[ \sum_{d-n<k\leq d} \vep^{\, k}\, \ve_k \right] \right)\llcorner \left[\sum_{l=1}^d \vep^{\, l}\, \ve_l \right] \\[3mm]
			& = \sum_{\substack{d-n<k\leq d \\[1mm] 1\leq l\leq d}} \vep^{\, k}\, \vep^{\, l}\, \left(\ve_1\wedge \cdots \wedge \ve_{d-n}\wedge \ve_k \right)\llcorner \ve_l \\[3mm]
			& = \sum_{k=d-n+1}^d \vep^{\, k}\, \sum_{i=1}^{d-n} (-1)^{i-1}\, \vep^{\, i}\, \ve_1\wedge\cdots \wedge \ve_{i-1}\wedge \ve_{i+1} \wedge \cdots \wedge \ve_{d-n} \wedge \ve_k \\[3mm]
			&\qquad + \sum_{k=d-n+1}^d(-1)^{d-n}\, (\vep^{\, k})^2 \, \ve_1\wedge \cdots\wedge \ve_{d-n}.
		\end{align*}
		Since all the multivectors in the above expression are unitary and orthogonal, we obtain\footnote{From this expression, we deduce that 
			\begin{align*} 
				\left| \left(\ve_1\wedge \cdots \wedge \ve_{d-n}\wedge \vep \right)\llcorner \vep\right|^2 \leq \min_{1\leq i\leq d-n} \left| (\ve_i\wedge \vep)\llcorner \vep \right|^2.
		\end{align*}}
		\begin{align}\label{eq:computation}
			\left| \left(\ve_1\wedge \cdots \wedge \ve_{d-n}\wedge \vep \right)\llcorner \vep\right|^2 & = \sum_{ \substack{d-n<k\leq d\\[1mm] 1\leq i\leq d-n} } (\vep^{\, k})^2\, (\vep^{\, i})^2 + \left| \sum_{k=d-n+1}^d (\vep^{\, k})^2\right|^2 = \left( \sum_{d-n<k\leq d} (\vep^{\, k})^2\right)  |\vep|^2.
		\end{align}
		Plugging this into \eqref{eq:Int121}, we obtain
		\begin{align*}
			\fint_{\B^d\left(0,\frac{r}{100}\right)}\frac{|\vn\wedge(\vPhi-\vq)|\ |\vPhi-\vq|^{n-1}}{|(\vn\wedge(\vPhi-\vq))\llcorner (\vPhi-\vq)|^n} \ d\vq & \leq C(n,d)\, r^{n-1-d} \int_{[-2r,2r]^d\setminus \left[-\frac{r}{100}, \frac{r}{100}\right]^d} \frac{ \left(\vep_{d-n+1}^{\, 2} + \cdots \vep_d^{\, 2}\right)^{\frac{1}{2}}}{|\vep|^n\,  \left(\vep_{d-n+1}^{\, 2} + \cdots \vep_d^{\, 2}\right)^{\frac{n}{2}}} \, d\vep \\[3mm]
			& \leq C(n,d)\, r^{-1-d} \int_{[-2r,2r]^d} \frac{d\vep}{\left(\vep_{d-n+1}^{\, 2} + \cdots \vep_d^{\, 2}\right)^{\frac{n-1}{2}}} \\[3mm]
			& \leq C(n,d)\, r^{-1-n} \int_{[-2r,2r]^n} \frac{dy}{|y|^{n-1}}\\[3mm]
			& \leq C(n,d)\, r^{-1-n} \int_0^{4r} \, ds.
		\end{align*}
		We obtain 
		\begin{align}\label{eq:Int2}
			\fint_{\B^d\left(0,\frac{r}{100}\right)}\frac{|d|\vPhi-\vq||_{g_{\vPhi}}\ |\vPhi-\vq|^n}{|(\vn_{\vPhi}\wedge(\vPhi-\vq))\llcorner (\vPhi-\vq)|^n} \ d\vq \leq \frac{C(n,d)}{r^n}.
		\end{align}
		We now estimate the integral \eqref{eq:Int11} by using that $|\vPhi-q|\leq 1$ on the domain of integration:
		\begin{align*}
			\fint_{\B^d\left(0,\frac{r}{100}\right)}\frac{|d|\vPhi-\vq||_{g_{\vPhi}}\ |\vPhi-\vq|^{2n}}{|(\vn_{\vPhi}\wedge(\vPhi-\vq))\llcorner (\vPhi-\vq)|^n} \ d\vq \leq C(n,d)\, r^n \fint_{\B^d\left(0,\frac{r}{100}\right)}\frac{|d|\vPhi-\vq||_{g_{\vPhi}}\ |\vPhi-\vq|^n}{|(\vn_{\vPhi}\wedge(\vPhi-\vq))\llcorner (\vPhi-\vq)|^n} \ d\vq .
		\end{align*}
		By \eqref{eq:Int2}, we obtain 
		\begin{align}
			\fint_{\B^d\left(0,\frac{r}{100}\right)}\frac{|d|\vPhi-\vq||_{g_{\vPhi}}\ |\vPhi-\vq|^{2n}}{|(\vn_{\vPhi}\wedge(\vPhi-\vq))\llcorner (\vPhi-\vq)|^n} \ d\vq \leq C(n,d). \label{eq:Int3}
		\end{align}	
		Plugging \eqref{eq:Int2}-\eqref{eq:Int3} into \eqref{eq:Int11}-\eqref{eq:Int12}, we obtain
		\begin{align*}
			& \fint_{\B^d\left(0,\frac{r}{100}\right)} \Bigg( \int_{\vPhi^{-1}\left( \B^d(\vq,\frac{9r}{10})\setminus \B^d(\vq,\frac{6r}{10})\right)} \left|d|\vPhi-\vq| \right|_{g_{\vPhi}}\Bigg[ \frac{|\vPhi-\vq|^{2n}\ |\vII|^n}{|\big(\vn_{\vPhi}\wedge (\vPhi-\vq)\big)\llcorner (\vPhi-\vq) |^n\ \Er(\vPhi;Q_r)} \\[3mm]
			& \qquad + \frac{|\vPhi-\vq|^n}{|(\vn_{\vPhi}\wedge(\vPhi-\vq))\llcorner (\vPhi-\vq)|^n} + \frac{1}{r^n} \Bigg]\, d\vol_{g_{\vPhi}}\Bigg)\, dq \\[3mm]
			\leq &\ C(n,d) \int_{\vPhi^{-1}\left( \B^d(0,r)\setminus \B^d(0,\frac{r}{2})\right)} \left[ \frac{ |\vII|^n_{g_{\vPhi}} }{\Er(\vPhi; Q_r)} + \frac{1}{r^n}\right]\, d\vol_{g_{\vPhi}} +C(n,d,V)\\[2mm]
			\leq &\ C(n,d,V).
		\end{align*}
		In the last estimate, we used the assumption \eqref{hyp:Chart} and the estimate \eqref{eq:volume_growth_balls}.
	\end{proof}

	By \eqref{eq:II_slice}, we deduce the existence of a slice having a second fundamental form in $L^n$ and volume bounded from above by $r^{n-1}$.
	
	\begin{claim}\label{cl:slicing}
		There exists $\vq\in \B^d(0,\frac{r}{100})$ and $\rho\in\left[ \frac{6r}{10},\frac{9r}{10}\right]$ such that 
		\begin{align*}
			\int_{ \vPhi^{-1}(\s^{d-1}(\vq,\rho)) }\Bigg[ \frac{|\vPhi-\vq|^{2n}\ |\vII|^n}{|\big(\vn_{\vPhi}\wedge (\vPhi-\vq)\big)\llcorner (\vPhi-\vq) |^n\ \Er(\vPhi;Q_r)}  + \frac{|\vPhi-\vq|^n}{|(\vn_{\vPhi}\wedge(\vPhi-\vq))\llcorner (\vPhi-\vq)|^n} + \frac{1}{r^n} \Bigg]\, d\vol_{g_{\vPhi}} \leq \frac{C(n,d,V)}{r}.
		\end{align*}
		Moreover, the set $\vPhi^{-1}(\s^{d-1}(\vq,\rho))$ is a $C^{\infty}$ manifold.
	\end{claim}
	
	\begin{proof}
		We use the coarea formula:
		\begin{align*}
			& \fint_{\B^d\left(0,\frac{r}{100}\right)} \Bigg( \int_{\vPhi^{-1}\left( \B^d(\vq,\frac{9r}{10})\setminus\B^d(\vq,\frac{6r}{10})\right)} \left|d|\vPhi-\vq|\right|_{g_{\vPhi}}\Bigg[ \frac{|\vPhi-\vq|^{2n}\ |\vII|^n}{|\big(\vn_{\vPhi}\wedge (\vPhi-\vq)\big)\llcorner (\vPhi-\vq) |^n\ \Er(\vPhi;Q_r)} \\[3mm]
			& \qquad + \frac{|\vPhi-\vq|^n}{|(\vn_{\vPhi}\wedge(\vPhi-\vq))\llcorner (\vPhi-\vq)|^n}+ \frac{1}{r^n} \Bigg]\, d\vol_{g_{\vPhi}}\Bigg)\, d\vq \\[3mm]
			= &\ \fint_{\B^d\left(0,\frac{r}{100}\right)} \int_{\frac{6r}{10}}^{\frac{9r}{10}} \Bigg( \int_{ \vPhi^{-1}(\s^{d-1}(\vq,\rho)) }\Bigg[ \frac{|\vPhi-\vq|^{2n}\ |\vII|^n}{|\big(\vn_{\vPhi}\wedge (\vPhi-\vq)\big)\llcorner (\vPhi-\vq) |^n\ \Er(\vPhi;Q_r)} \\[3mm]
			& \qquad + \frac{|\vPhi-\vq|^n}{|(\vn_{\vPhi}\wedge(\vPhi-\vq))\llcorner (\vPhi-\vq)|^n}  + \frac{1}{r^n} \Bigg]
			\, d\vol_{g_{\vPhi}}\Bigg)\, d\rho\, d\vq \\[3mm]
			\geq & \frac{3r}{10}\, \inf_{\vq,r} \Bigg( \int_{ \vPhi^{-1}(\s^{d-1}(\vq,\rho)) }\Bigg[ \frac{|\vPhi-\vq|^{2n}\ |\vII|^n}{|\big(\vn_{\vPhi}\wedge (\vPhi-\vq)\big)\llcorner (\vPhi-\vq) |^n\ \Er(\vPhi;Q_r)} \\[3mm]
			&\qquad + \frac{|\vPhi-\vq|^n}{|(\vn_{\vPhi}\wedge(\vPhi-\vq))\llcorner (\vPhi-\vq)|^n} + \frac{1}{r^n} \Bigg]\, d\vol_{g_{\vPhi}}\Bigg).
		\end{align*}
		We can choose $\vq$ and $r$ such that $\vPhi^{-1}(\s^{d-1}(\vq,\rho))$ is a $C^{\infty}$ manifold, since we assumed $\vPhi$ to be $C^{\infty}$. Indeed, for any $\vq\in \R^d$, we consider the map $f(x)\coloneq |\vPhi(x)-\vq|^2$. By Sard Theorem, the image by $f$ of its critical point have zero measure. Moreover, $\vPhi(\Sigma)$ is tangential to a sphere $\s^{d-1}(\vq,r)$ at a point $x\in\Sigma$ if and only if $x$ is a critical point of $f$. Hence, for a.e. $r>0$, the intersection $\vPhi(\Sigma)$ with $\s^{d-1}(\vq,r)$ is transversal. For these radii, this intersection defines a $C^{\infty}$ immersed submanifold.
	\end{proof}
	
	We now conclude the proof of Proposition \ref{pr:est_good_slice}.
	
	\begin{proof}[Proof of Proposition \ref{pr:est_good_slice}.]
		It suffices to show that 
		\begin{align*}
			\int_{\vPhi^{-1}\left( \s^{d-1}(\vq,\rho) \right)} \frac{|\vn_{\vPhi}\llcorner (\vPhi-\vq)|^n}{|(\vn_{\vPhi}\wedge(\vPhi-\vq))\llcorner (\vPhi-\vq)|^n}\, d\vol_{g_{\vPhi}} \leq C(n,d,V).
		\end{align*}
		This follows from Claim \ref{cl:slicing}:
		\begin{align*}
			\int_{\vPhi^{-1}\left( \s^{d-1}(\vq,\rho) \right)} \frac{|\vn_{\vPhi}\llcorner (\vPhi-\vq)|^n}{|(\vn_{\vPhi}\wedge(\vPhi-\vq))\llcorner (\vPhi-\vq)|^n}\, d\vol_{g_{\vPhi}} & \leq \int_{\vPhi^{-1}\left( \s^{d-1}(\vq,\rho) \right)} \frac{|\vPhi-\vq|^n}{|(\vn_{\vPhi}\wedge(\vPhi-\vq))\llcorner (\vPhi-\vq)|^n}\, d\vol_{g_{\vPhi}}  \leq \frac{C(n,d,V)}{r}.
		\end{align*}
	\end{proof}

	\subsection{A given good slice is a union of topological spheres, each of them being close to a round sphere}\label{sec:graph_good_slice}
	
	In this section, we prove that the slices provided by Proposition \ref{pr:est_good_slice} are diffeomorphic to round spheres and can be described as graphs. Since \underline{we work with $C^{\infty}$-immersions in this section}, the graphs and the diffeomorphisms are actually of class $C^{\infty}$, but we record only the estimates required to pass to the limit in an approximation of a weak immersion by $C^{\infty}$-immersions.
	
	\begin{proposition}\label{pr:global_graph_slice}
		With the notations of Proposition \ref{pr:est_good_slice}, there exist $\eps_0\in(0,1)$ and $C>0$ depending only on $n$, $d$ and $V$ such that the following holds. If $\eps\leq \eps_0$, then each connected component $\Sr$ of $S_{\rho}$ is diffeomorphic\footnote{Since $\vPhi\in \Imm(\Sigma;\R^d)$, the diffeomorphism is also of class $C^{\infty}$.} to $\s^{n-1}$. Moreover, there exists an $n$-dimensional affine plane $\Pc_{\Sr} = \vq_{\Sr} + \Pc\subset \R^d$, where $\Pc$ is vectorial $n$-plane, satisfying the following properties:
		\begin{enumerate}
			\item\label{it:Centre} There exists a constant $\delta>0$ depending only on $n$, $d$ and $V$ such that $|\vq - \vq_{\Sr}|\leq \delta\, \rho$,
			
			\item\label{it:Graph} There exists a map $\vphi_{\Sr} \in W^{2,n}(\Sc_{\Pc};\R^d)$\footnote{Since $\vPhi\in\Imm(\Sigma;\R^d)$, the map $\vphi_{\Sr}$ is actually $C^{\infty}$.}, where $\Sc_{\Pc}\coloneq \Pc_{\Sr}\cap \s^{d-1}(\vq,\rho)$, such that the set $\vPhi(\Sr)$ is given by 
			\begin{align*}
				\vPhi(\Sr) = \left\{ \vtheta + \vphi_{\Sr}(\vtheta) : \vtheta\in \Sc_{\Pc} \right\}.
			\end{align*}
			Moreover, the following estimates hold
			\begin{align*}
				\begin{cases} 
					\displaystyle r^{-1}\, \|\vphi_{\Sr}\|_{L^{\infty}(\Sc_{\Sr})} + \|\g \vphi_{\Sr}\|_{L^{\infty}(\Sc_{\Sr})} \leq C\, \eps,\\[3mm]
					\displaystyle r^{\frac{1}{n}}\, \|\g^2 \vphi_{\Sr}\|_{L^n(\Sc_{\Sr})} \leq C.
				\end{cases} 
			\end{align*}
			
			\item\label{it:Osc_n} We have 
			\begin{align*}
				\forall x,y\in \Sr,\qquad \left| \vn_{\vPhi}(x) - \vn_{\vPhi}(y)\right| \leq C\, \eps.
			\end{align*}
		\end{enumerate}
	\end{proposition}

	By the previous section, we have a slice $S_{\rho}^{n-1}\coloneqq \vPhi^{-1}(\s^{d-1}(\vq,\rho))$ and a constant $C=C(n,d,V)>0$ such that the generalized Gauss map $\vnu$ of the immersion $\vPhi_{\rho}\coloneqq \vPhi\big|_{S_{\rho}}$ satisfies
	\begin{align}
		& \vol_{g_{\vPhi}}(S_{\rho})\leq C\, \rho^{n-1}, \label{eq:vol_good_slice}\\[3mm]
		& \rho^{\frac{1}{n}}\, \|\vA\|_{L^n\left(S_{\rho},g_{\vPhi}\right)} \leq C, \label{eq:Ln_good_slice} \\[3mm]
		& \rho^{\frac{1}{n}}\|\vII_{\vPhi}\|_{L^n\left(S_{\rho},g_{\vPhi}\right)}  \leq C\, \eps. \label{eq:Lnr_good_slice}
	\end{align}
	\textbf{In the rest of the section, we fix $\Sr\subset S_{\rho}$ a connected component.} By \eqref{eq:Ln_good_slice}, we know that $\vPhi(\Sr)$ is locally given by graphs of fixed size over its tangent space in $\s^{d-1}(\vq,\rho)$. We adapt the proof of \cite[Theorem 2.6]{breuning2015} to our context in order to obtain quantitative estimates.
	\begin{lemma}\label{lm:Lipschitz_graph}
		There exist $\kappa \in(0,1)$ and $\lambda>0$ depending only on $n$, $d$ and $V$ satisfying the following property. Given any point $p\in \Sr$, we consider the connected component $\Sr(p)\subset \Sr$ of $\vPhi_{\rho}^{-1}\left(\B^d(\vPhi_{\rho}(p),\lambda\, \rho)\right)$ containing $p$.
		Let $\vnu_1(p),\ldots,\vnu_{d-n}(p)$ be a basis of the normal space of $\vPhi_{\rho}$ at $p$. Given a point $\vx\in \R^d$, we decompose it on 
		\begin{align*} 
			\vPhi(p) + \Span(\vPhi(p)-\vq )\oplus T_{\vPhi(p)}\, \vPhi(\Cr(p)) \oplus \Span(\vnu_1(p),\ldots,\vnu_{d-n}(p)).
		\end{align*}    
		We use the following notations
		\begin{align}\label{eq:decompo_ambient}
			\begin{cases} 
				\displaystyle \vx = \vPhi(p) + \pipa(\vx) \frac{\vPhi(p)-\vq}{\rho} + \vpit(\vx) + \vpipe(\vx) , \\[5mm]
				\displaystyle \pipa(\vx) = (\vx-\vPhi(p)) \cdot \frac{\vPhi(p)-\vq}{\rho}, \\[5mm]
				\displaystyle \vpit(\vx) = \proj_{T_{\vPhi(p)} \vPhi(\Cr(p))} \big( \vx-\vPhi(p) \big), \\[5mm]
				\displaystyle \vpipe(\vx) = \sum_{\alpha=1}^{d-n} \big[ (\vx-\vPhi(p))\cdot \vnu_{\alpha}(p) \big]\, \vnu_{\alpha}(p).
			\end{cases}
		\end{align}
		Then, the following holds. Given $r>0$, we define the ball $\Bc(p,r)\subset T_{\vPhi(p)} \vPhi(\Sr(p))$ centred at the origin with radius $r$ and $\mathscr{I}(p,r) \subset \vPhi(\Sr(p))$ being the reciprocal image of $\Bc(p,r)$ by $\vpit$ restricted to $\vPhi(\Sr(p))$:
		\begin{align*}
			\begin{cases} 
				\displaystyle \Bc(p,r) \coloneqq B_{T_{\vPhi(p)} \vPhi(\Sr(p))} (0,r),\\[3mm]
				\displaystyle \Ic(p,r) \coloneqq \vPhi(\Sr(p))\cap \vpit^{-1}(\Bc(p,r)).
			\end{cases}
		\end{align*}
		The map $\vpit\colon \Ic(p,\lambda\rho)\to \Bc(p,\lambda\rho)$ is invertible and the maps $\vompe \coloneqq \vpipe\circ \left((\vpit)_{|\Ic(p,\lambda\rho)}\right)^{-1}$ and $\ompa \coloneqq \pipa\circ  \left( (\vpit)_{|\Ic(p,\lambda\rho)}\right)^{-1}$ satisfy
		\begin{align}\label{eq:est_graph1}
			\rho^{-1}\Big(\|\vompe\|_{L^{\infty}(\Bc(p,\lambda\rho))} + \|\ompa\|_{L^{\infty}(\Bc(p,\lambda\rho))} \Big) + \|\g \vompe\|_{L^{\infty}(\Bc(p,\lambda\rho))} + \|\g \ompa\|_{L^{\infty}(\Bc(p,\lambda\rho))} \leq \kappa.
		\end{align}
	\end{lemma}
	
	\begin{proof}
		Let $p\in S_{\rho}$. Given $s>0$, we define $\Sr(p,s)$ to be the connected component of $\Sr \cap \vPhi_{\rho}^{-1}\left(\B^d(\vPhi_{\rho}(p),s)\right)$ containing $p$. We also denote by $\Vc_s\coloneqq \vpit \left( \vPhi(\Sr(p,s)) \right) \subset T_{\vPhi(p)} \vPhi(\Sr(p))$.\\
		
		\textit{Step 1: for $s>0$ small enough, the map $\vpit\circ \left( \vPhi\big|_{\Sr(p,s)}\right)$ is bi-Lipschitz and $\Vc_s$ contains $\Bc(p,s/2)$.}\\
		Since $\vnu$ lies in $W^{1,n}\big( \Sr \big)$\footnote{Here we use that $\Sr$ is regular enough to define the Sobolev spaces.} and $\dim \Sr=n-1$, we have that each $\vnu_{\alpha}$ (always defined locally) lies in $C^{0,\frac{1}{n}}$. Hence for any $\delta>0$, there exists $s_{\delta}>0$ such that 
		\begin{align*}
			\forall x\in \Sr(p,s_{\delta}),\qquad |\vnu_{\alpha}(x)-\vnu_{\alpha}(p)| \leq \delta.
		\end{align*}
		Thus, for $\delta>0$ small enough, the following projection map is bi-Lipschitz on each connected component
		\begin{align}\label{eq:PhiRho_Loc_graph}
			\vPhi(\Sr(p,s_{\delta})) \to \Span\big(\vnu_1(p),\ldots,\vnu_{d-n}(p) \big)^{\perp}\cap T_{\vPhi(p)}\s^{d-1}(\vq,\rho).
		\end{align}
		Indeed, given $x,y\in \Sr(p,s_{\delta})$  and a $C^1$-path $\gamma\colon [0,1]\to \Sr(p,s_{\delta})$ such that $\gamma(0)=x$ and $\gamma(1)=y$. The path $\vsigma \coloneqq \vPhi\circ\gamma$ is therefore a $C^1$-path from $\vPhi(x)$ to $\vPhi(y)$. We have
		\begin{align*}
			|(\vsigma(0)-\vsigma(1))\cdot \vnu_{\alpha}(p)| & = \left|\int_0^1 \frac{d\vsigma}{dt}(t)\cdot \vnu_{\alpha}(p)\, dt \right|   = \left| \int_0^1 \frac{d\vsigma}{dt}(t)\cdot \big(\vnu_{\alpha}(p) - \vnu_{\alpha}(\gamma(t)) \big)\, dt\right|   \leq \delta\, \mathrm{Length}(\vsigma) = \delta\, \mathrm{Length}_{g_{\vPhi}}(\gamma).
		\end{align*}
		Since $\vPhi$ is $C^{\infty}$ and $d\vPhi(p) \colon (T_p\Sigma,g_{\vPhi}(p)) \to (T_{\vPhi(p)}\vPhi(\Sigma),\geu)$ is an isometry, we obtain that, up to reducing $s_{\delta}$, it holds
		\begin{align*}
			\left|(\vPhi(x)-\vPhi(y))\cdot \vnu_{\alpha}(p)\right| \leq 2\, \delta\, \dist_{g_{\vPhi}}(x,y) \leq 4\, \delta\, \left|\vPhi(x)-\vPhi(y)\right|.
		\end{align*}
		Therefore, each map $\vx\mapsto \vx\cdot \vnu_{\alpha}(p)$ is $(4\, \delta)$-Lipschitz on $\vPhi(\Sr(p,s_{\delta}))$. Reproducing the same computation as above with $\vnu_{\alpha}(p)$ replaced by $\rho^{-1}(\vPhi(p)-\vq)$ and using that $\frac{d\vsigma}{dt}$ is orthogonal to $(\vsigma-\vq)$, we deduce that up to reducing $s_{\delta}$, it holds
		\begin{align*}
			\left| \big(\vPhi(x)-\vPhi(y) \big)\cdot \frac{\vPhi(p)-\vq}{\rho}\right| \leq \frac{ C\, \delta}{\rho}\, \left|\vPhi(x)-\vPhi(y)\right|.
		\end{align*}
		Thus, for $\delta$ small enough, we deduce that the projection \eqref{eq:PhiRho_Loc_graph} is bi-Lipschitz on $\vPhi(\Sr(p,s_{\delta}))$, that is to say $(\vpit)\big|_{\vPhi(\Sr(p,s_{\delta}))}$ is bi-Lipschitz. Up to reducing $s_{\delta}$, we have
		\begin{align*}
			\vpit \left( \vPhi(\Sr(p,s_{\delta}))\right) = \proj_{T_{\vPhi(p)}\, \vPhi(\Sr(p,\rho))} \Big( \vPhi(\Sr(p,s_{\delta})) - \vPhi(p) \Big) \supset \Bc \left(p, \frac{s_{\delta}}{2} \right).
		\end{align*}
		
		\textit{Step 2: Maximal size of the graph.}\\
		Let $M_{\rho}\coloneqq \vPhi(\Sr(p,\rho))$. Consider $s_{\max}>0$ the supremum of all the radius $s>0$ such that the set $M_{\rho}\cap \B^d(\vPhi_{\rho}(p),s)$ can be represented by the map $(\vpit)\big|_{\Ic(p,s)}$ as in the statement of Claim \ref{lm:Lipschitz_graph}. 
		Since $\vPhi_{\rho}$ is valued into $\s^{d-1}(\vq,\rho)$, we must have $s_{\max}\leq \rho$. If $s<s_{\max}$, we denote 
		\begin{align*}
			\begin{cases}
				\displaystyle \vompe^s \coloneqq \vpipe\circ ( (\vpit)_{|\Ic(p,s)})^{-1}, \\[3mm]
				\displaystyle \ompa^s \coloneqq \pipa\circ( (\vpit)_{|\Ic(p,s)})^{-1}.
			\end{cases}
		\end{align*}
		By maximality, we must have
		\begin{align}\label{eq:maximality}
			s^{-1} \left( \|\vompe^s \|_{L^{\infty}(\Bc(p,s))} + \|\ompa^s\|_{L^{\infty}(\Bc(p,s))}\right) + \|\g \vompe^s \|_{L^{\infty}(\Bc(p,s))} + \|\g \ompa^s\|_{L^{\infty}(\Bc(p,s))}\xrightarrow[s\to s_{\max}]{} +\infty.
		\end{align}
		Otherwise, $\ompa^{s_{\max}}$ and $\vompe^{s_{\max}}$ would be $C^1$ on the whole $\overline{\Bc(p,s_{\max})}$ by Lemma \ref{lm:A_graph}. This would imply that $((\vpit)_{|\vPhi(\Sr(p,s_{\max}))})^{-1}$ as well and that we could extend these maps on an open neighbourhood of $\Bc(p,s_{\max})$, contradicting the definition of $s_{\max}$.\\
		
		We now consider the following map which is continuous and vanishes for $s=0$: 
		\begin{align*} 
			f\colon s\mapsto  s^{-1} \left( \|\vompe^s\|_{L^{\infty}(\Bc(p,s))} + \|\ompa^s\|_{L^{\infty}(\Bc(p,s))}\right) + \|\g \vompe^s \|_{L^{\infty}(\Bc(p,s))} + \|\g \ompa^s\|_{L^{\infty}(\Bc(p,s))}.
		\end{align*} 
		By \eqref{eq:maximality}, we there exists a radius $s_0\in(0,s_{\max})$ such that the following equality holds
		\begin{align}\label{eq:Est_v}
			f(s_0)= \kappa.
		\end{align}
		By construction, it holds $\g \vompe^{s_0}(0)=0$, $\g \ompa^{s_0}(0)=0$, $\ompa^{s_0}(0)=0$ and $\vompe^{s_0}(0)=0$. Since $\Bc(p,s_0)$ is a Euclidean $(n-1)$-dimensional ball of radius $s_0$, we have by Sobolev embedding, that
		\begin{align*}
			f(s_0)&\leq C\, s_0^{\frac{1}{n}}\, \left( \left\|\g^2 \ompa^{s_0}\right\|_{L^n\left( \Bc(p,s_0)\right)} + \left\|\g^2 \vompe^{s_0} \right\|_{L^n\left( \Bc(p,s_0)\right)} \right).
		\end{align*}
		Moreover, the image of $\vPhi$ is given by the graph of the function $(\vompe^{s_0},\ompa^{s_0})$. Thanks to Lemma \ref{lm:A_graph}, we obtain
		\begin{align*}
			f(s_0) \leq C\, s_0^{\frac{1}{n}} \|\vA\|_{L^n\left( \Bc (p,s_0) \right)} + C\, s_0\, \rho^{-1}\,  f(s_0)^2.
		\end{align*}
		By \eqref{eq:Est_v} and \eqref{eq:Ln_good_slice}, we obtain a constant $c_0>0$ depending only on $n$, $d$ and $V$ such that 
		\begin{align*}
			\kappa \leq c_0\, \left( \frac{s_0}{\rho}\right)^{\frac{1}{n}}+ c_0\, \frac{s_0}{\rho}\, \kappa^2.
		\end{align*}
		Since $s_0\leq \rho$, if $\kappa\in \left(0,\frac{1}{2\, c_0}\right)$, we obtain 
		\begin{align*}
			\frac{\kappa}{2} \leq c_0 \left( \frac{s_0}{\rho}\right)^{\frac{1}{n}} .
		\end{align*}
		Therefore we have the following property. For any radius $s\in\left(0, \frac{\kappa^n \rho}{2^n\, c_0^n}\right)$, the set $M_{\rho}\cap \B^d(\vPhi_{\rho}(p),s)$ is the graph of a Lipschitz function $\vu_s\coloneqq (\vompe^s,\ompa^s)$ such that $\vu_s(0)=0$, $\g \vu_s(0)=0$ and
		\begin{align*}
			s^{-1} \|\vu_s\|_{L^{\infty}(\Bc(p,s))} + \|\g \vu_s\|_{W^{1,\infty}\left( \Bc(p,s) \right)} \leq \kappa.
		\end{align*}
	\end{proof}
	
	As a direct corollary of \Cref{lm:Lipschitz_graph}, we can consider a Besicovitch covering of $\vPhi(\Sr)$ by the graphs introduced in \Cref{lm:Lipschitz_graph}. We now prove that $\vn_{\vPhi}$ is uniformly close to a constant vector along $\Sr$.
	
	\begin{claim}\label{cl:osc_n}
		There exists $C>0$ depending only on $n$, $d$ and $V$ such that the following holds
		\begin{align*}
			\forall x,y\in \Sr,\qquad |\vn_{\vPhi}(x)-\vn_{\vPhi}(y)|\leq C\, \eps.
		\end{align*}
	\end{claim}
	\begin{proof}
		By Lemma \ref{lm:Lipschitz_graph}, there exists $\kappa,\lambda \in(0,1)$ depending only on $n$, $d$ and $V$ satisfying the following property. Given any point $p\in \Sr$, we consider the connected component $\Sr(p)$ of $\Sr\cap \vPhi_{\rho}^{-1}\big(\B^d(\vPhi_{\rho}(p),\lambda\, \rho) \big)$ containing $p$. Then, the submanifold $\vPhi(\Sr(p))$ is (up to translation and rotation) the graph of a function $\vu_p\colon \Bc_p\coloneqq B_{T_{\vPhi(p)} \vPhi(\Sr(p))}(0,\lambda \rho)\to \R^{d-n}$ such that 
		\begin{align}\label{eq:est_graph2}
			\rho^{-1} \|\vu_p\|_{L^{\infty}(\Bc_p)} + \|\vu_p\|_{W^{1,\infty}\left( \Bc_p \right)} \leq \kappa.
		\end{align}
		Using the volume estimate \eqref{eq:vol_good_slice}, we can cover $\vPhi(\Sr)$ by a number $I$ of such balls, with $I\leq C(n,d,V)$. We denote these balls $(\Bc_{p_i})_{1\leq i\leq I}$. Moreover, we obtain for any $\vx,\vy\in \vPhi(\Sr(p_i))$ represented as $\vx=(a_x,\vu_{p_i}(a_x))$ and $\vy=(a_y,\vu_{p_i}(a_y))$, we have that (see for instance \cite[Theorem 7.17]{gilbarg2001})
		\begin{align}\label{eq:Sobolev_graph}
			\left|\vn_{\vPhi}(\vx)-\vn_{\vPhi}(\vy)\right| \leq C\rho^{\frac{1}{n}} \left\| \g\left[ \vn_{\vPhi}(\cdot, \vu_{p_i}(\cdot))\right] \right\|_{L^n(\Bc_{p_i},\, g_{eucl})}.
		\end{align}
		We estimate the $L^n$ norm using \eqref{eq:est_graph2}:
		\begin{align*}
			\int_{\Bc_{p_i}} \left|\g\left[\vn_{\vPhi}(x,\vu_{p_i}(x)\right] \right|^n\, dx  
			& \leq C\left(1+\|\g \vu_{p_i}\|_{L^{\infty}(\Bc_{p_i})}^n \right) \int_{\Bc_{p_i}} |\g\vn_{\vPhi}(x,\vu_{p_i}(x))|^n\, \sqrt{\det \left( I_{n-1} + (d\vu_{p_i})^{\top} (d\vu_{p_i}) \right)}\, dx \\[2mm]
			& \leq C\int_{\vPhi(\Sr(p_i))} |d\vn_{\vPhi}|^n_{g_{\vPhi}}\, d\vol_{g_{\vPhi}}.
		\end{align*}
		Combining \eqref{eq:Sobolev_graph} and \eqref{eq:Lnr_good_slice}, we obtain
		\begin{align*}
			|\vn_{\Phi}(\vx)-\vn_{\Phi}(\vy)| \leq C\, \eps.
		\end{align*}
		Since there exist at most $I$ balls $\Bc_i$ with $I\leq C(n,d,V)$, we obtain 
		\begin{align}\label{eq:Osc_n_slice}
			\forall x,y\in \Sr,\qquad \left|\vn_{\vPhi}(x)-\vn_{\vPhi}(y)\right| \leq C(n,d,V)\, \eps.
		\end{align}
	\end{proof} 
	
	We consider vector fields $\vn_1,\ldots,\vn_{d-n}\in W^{1,n}(\Sr; \s^{d-1})$ orthonormal such that
	\begin{align*}
		\vn_{\vPhi}\big|_{\Sr} = \vn_1\wedge \cdots \wedge \vn_{d-n}.
	\end{align*}
	For each $\alpha\in \{1,\ldots,d-n\}$, we define
	\begin{align*}
		\vn_{\alpha,\Sr}\coloneqq \fint_{\Sr} \vn_{\alpha}\, d\vol_{g_{\vPhi}}.
	\end{align*}
	If $\eps_0$ is small enough, we deduce from \eqref{eq:Osc_n_slice} that $|\vn_{\alpha,\Sr}|\geq \frac{1}{2}$ and the projection $\vpip \colon \R^d \to \Span(\vn_{1,\Sr},\ldots, \vn_{d-n,\Sr})^{\perp}$ is well-defined. We define the point $\vq_{\Sr}$ to be the average of the normal component of $\vPhi-\vq$ to $\vPhi_\ast T_x\Sigma^n$ along $\Sr$: 
	\begin{align}\label{eq:choice_qC}
		\vq_{\Sr} \coloneqq \vq + \sum_{\alpha=1}^{d-n} \left(\fint_{\Sr} (\vPhi-\vq)\cdot \vn_{\alpha,\Sr}\, d\vol_{g_{\vPhi}} \right)\, \vn_{\alpha,\Sr}.
	\end{align}
	Hence the following linear and affine $n$-planes are well-defined:
	\begin{align}\label{eq:def_planes}
		\begin{cases} 
			\displaystyle \Pc \coloneqq \Span(\vn_{1,\Sr},\ldots,\vn_{d-n,\Sr})^{\perp}, \\[3mm]
			\displaystyle \Pc_{\Sr} \coloneqq \vq_{\Sr} + \Pc.
		\end{cases} 
	\end{align}
	Given $\vx\in \R^d$, we consider its decomposition along $\Pc_{\Sr}\oplus \Pc^{\perp}$. Namely, the following decomposition holds:
	\begin{align*}
		\begin{cases} 
			\displaystyle \vx= \vq_{\Sr} + \vpipp(\vx)  + \vpip(\vx), \\[5mm]
			\displaystyle \vpipp(\vx) \coloneqq  \sum_{\alpha=1}^{d-n} \left[ (\vx-\vq_{\Cr})\cdot \vn_{\alpha,\Cr} \right]\, \vn_{\alpha,\Cr}.
		\end{cases}
	\end{align*}
	We now prove that the map $\vpip\circ\vPhi$ is locally injective on $\Sr$.
	
	\begin{claim}\label{cl:graph_good_slice}
		There exists $\lambda>0$ and $\eps_0>0$ depending only on $n$, $d$ and $V$ such that the following holds. Let $\eps\in(0,\eps_0)$. Given any $x\in \Sr$, we denote $\Sr(x)\subset \Cr$ the connected component containing $x$ of the following set, where $\lambda$ is given by \Cref{lm:Lipschitz_graph}: 
		\begin{align*}
			\left\{ y\in \Sr : |\vPhi(y) - \vPhi(x)|< \lambda\, \rho \right\}.
		\end{align*}
		For any $x\in \Sr$, the map $(\vpip)\big|_{\vPhi(\Sr(x))}$ is injective and there exists a constant $C>0$ depending only on $n$, $d$ and $V$ such that the map $\vpipp$ satisfies
		\begin{align}\label{eq:small_perp_proj}
			\forall \vy,\vz \in \vPhi_{\rho}(\Sr(x)),\qquad |\vpipp(\vy)-\vpipp(\vz)| \leq C\, \eps\, |\vy-\vz|.
		\end{align}
	\end{claim}
	\begin{proof} 
		Consider the numbers $\lambda>0$ and $\kappa\in(0,1)$ provided by Lemma \ref{lm:Lipschitz_graph}. Let $x\in \Cr$. Given $y,z\in \Sr(x)$, there exists a path $\gamma:[0,1]\to \Sr(x)$ such that $\gamma(0)=z$ and $\gamma(1)=y$. Hence, it holds
		\begin{align*}
			\left( \vPhi(z)-\vPhi(y) \right)\cdot \vn_{\alpha,\Sr}  & = \int_0^1 \frac{d\vPhi\circ\gamma}{dt}(t)\cdot \vn_{\alpha,\Sr}\, dt   = \int_0^1 \frac{d\vPhi\circ\gamma}{dt}(t)\cdot \Big(\vn_{\alpha,\Sr} - \vn_{\alpha}(\gamma(t)) \Big)\, dt .
		\end{align*}
		By \eqref{eq:Osc_n_slice}, we obtain
		\begin{align*}
			\left|\vpipp(\vPhi(z))-\vpipp(\vPhi(y)) \right| \leq C(n,d,V)\, \eps \, \mathrm{Length}_{g_{\vPhi}}(\gamma).
		\end{align*}
		Taking the infimum over all path $\gamma$, we obtain 
		\begin{align*}
			\left|\vpipp(\vPhi(z))-\vpipp(\vPhi(y)) \right| \leq C(n,d,V)\, \eps\, \dist_{\Sr(x),g_{\vPhi}} (z,y).
		\end{align*}
		By definition of $\Sr(x)$ and Lemma \ref{lm:Lipschitz_graph}, we obtain
		\begin{align*}
			\left|\vpipp(\vPhi(z))-\vpipp(\vPhi(y)) \right| \leq C(n,d,V)\, \eps\, \left|\vPhi(z) - \vPhi(y) \right|.
		\end{align*}
		If $\eps>0$ is small enough, the projection $(\vpip)\big|_{\vPhi(\Sr(x))}$ is bi-Lipschitz.
	\end{proof}
	
	As a corollary of Claim \ref{cl:graph_good_slice}, we obtain that the following maps are bi-Lipschitz for every $x\in \Cr$: 
	\begin{align*} 
		\vpi_{x}\coloneq (\vpip)\big|_{\vPhi(\Sr(x))} \colon \vPhi_{\rho}(\Sr(x))\to \vpip\left(\vPhi_{\rho}(\Sr(x)) \right).
	\end{align*} 
	We denote the image of $\vpi_x$ by
	\begin{align}\label{eq:domain_local_graph}
		\Sc_{x,\Sr}\coloneqq \vpi_x\left(\vPhi_{\rho}(\Sr(x)) \right).
	\end{align}
	Then, the set $\vPhi(\Sr(x))$ is the graph of the Lipschitz function 
	\begin{align*} 
		\vu_x\coloneqq \vpipp\circ (\vpi_x)^{-1}\colon \Sc_{x,\Sr}\to \Pc^{\perp}.
	\end{align*} 
	By Claim \ref{cl:graph_good_slice}, we have the following estimates:
	\begin{align}\label{eq:smallness_gu}
		\|\g \vu_x\|_{L^{\infty}(\Sc_{x,\Sr})} \leq \left\| \g \left( \vpipp\circ (\vpi_x)^{-1} \right) \right\|_{L^{\infty}(\Sc_{x,\Sr})} = \left\|\g \left( \vpipp \big|_{\vPhi(\Cr(x))} \right) \right\|_{L^{\infty}\left( \vPhi(\Sr(x)) \right)} \leq C\, \eps.
	\end{align}
	We now control the $L^{\infty}$ norm of $\vu_x$. This is a consequence of the choice of $\vq_{\Sr}$ in \eqref{eq:choice_qC} together with \eqref{eq:small_perp_proj}.
	
	\begin{claim}\label{cl:smallness_n_dot_Phi}
		There exists a constant $C>1$ depending only on $n$, $d$ and $V$ such that
		\begin{align}\label{eq:smallness_u}
			\forall x\in \Sr,\qquad \|\vu_x\|_{L^{\infty}(\Sc_{x,\Sr})} \leq C\, \eps\, \rho.
		\end{align}
	\end{claim}
	\begin{proof}
		By construction, we have (the set $\Sr(x)$ has been defined in Claim \ref{cl:graph_good_slice}) that
		\begin{align*}
			\forall x\in \Sr,\qquad \vu_x(\Sc_{x,\Sr})= \left[\vPhi-\vq_{\Sr}-\vpip \circ (\vPhi-\vq_{\Sr})\right](\Sr(x)).
		\end{align*}
		From the definition of $\vq_{\Sr}$ in \eqref{eq:choice_qC}, we have
		\begin{align*}
			\left| \vPhi-\vq_{\Sr}-\vpip \circ (\vPhi-\vq_{\Sr})\right| & = \left| \vpipp (\vPhi - \vq_{\Sr})\right| \\[2mm]
			& = \left| \vpipp (\vPhi - \vq) - \sum_{\alpha=1}^{d-n} \left( \fint_{\Sr} (\vPhi-\vq)\cdot \vn_{\alpha,\Sr}\, d\vol_{g_{\vPhi}} \right)\, \vn_{\alpha,\Sr} \right|\\[2mm]
			& \leq \sum_{\alpha=1}^{d-n} \left|(\vPhi-\vq)\cdot\vn_{\alpha,\Sr} - \fint_{\Cr} (\vPhi-\vq)\cdot\vn_{\alpha,\Sr} \, d\vol_{g_{\vPhi}}\right|.
		\end{align*}
		By Lemma \ref{lm:Lipschitz_graph} and \eqref{eq:vol_good_slice}, we can cover $\vPhi(\Sr)$ by a finite number $I$ of graphs $(\Gc_i)_{1\leq i\leq I}$ over balls of diameter comparable to $\rho$, with $I\leq C(n,d,V)$. For each of these graph, we apply Claim \ref{cl:graph_good_slice} and obtain 
		\begin{align*}
			\sup_{x,y\in \Gc_i} \left|  (\vPhi(x)-\vq)\cdot\vn_{\alpha,\Sr}  -  (\vPhi(y)-\vq)\cdot\vn_{\alpha,\Sr}  \right|\leq C\, \eps\, \rho.
		\end{align*}
		We obtain the result by summing over all $i$ and using the bound $I\leq C(n,d,V)$.
	\end{proof}

	As a corollary of Claim \ref{cl:smallness_n_dot_Phi}, we obtain that the slice is localized close to the intersection of the round sphere $\s^{d-1}(\vq,\rho)$ and the affine plane $\Pc_{\Sr} = \vq_{\Sr}+\Pc$ (defined in \eqref{eq:def_planes}):
	\begin{align}\label{eq:ThickSlice}
		\vPhi(\Sr) \subset \left\{ 
		\vx\in \s^{d-1}(\vq,\rho) : \dist(\vx,\Pc_{\Sr} ) \leq C\, \eps\, \rho
		\right\}.
	\end{align}
	We will prove that $\Pc_{\Sr}$ is far from the north and south poles by showing that $\vq_{\Sr}$ cannot be too far from $\vq$, which is a consequence of Lemma \ref{lm:Lipschitz_graph}, that is to say the bound \eqref{eq:II_slice}.
	
	\begin{claim}\label{cl:slice_centered}
		There exists $\delta\in(0,1)$ depending only on $n$, $d$ and $V$ such that up to reducing $\eps_0$, we have $|\vq-\vq_{\Sr}|\leq \delta\, \rho$.
	\end{claim}
	
	\begin{proof}
		Thanks to Lemma \ref{lm:Lipschitz_graph} (with $\kappa$ and $\lambda$ depending only on $n$, $d$ and $V$), we have that 
		\begin{align*}
			\diam\left(\vPhi(\Sr)\right) \geq \sup_{p\in\Sr} \diam\left( \vPhi(\Sr(p)) \right) \geq \frac{\lambda}{2}\rho.
		\end{align*}
		Therefore, we must have 
		\begin{align*} 
			\vPhi(\Sr)\setminus \B^d\left(\vq\pm\rho\frac{\vq-\vq_{\Sr}}{|\vq-\vq_{\Sr}|},\frac{\lambda}{8}\rho \right) \neq \emptyset.
		\end{align*} 
		Indeed, the above ball has diameter $\frac{\lambda}{4} \rho$. Since $\vq-\vq_{\Sr} \in \Pc^{\perp}$, we obtain, for $\eps_0>0$ small enough, 
		\begin{align*}
			\left\{ 
			\vx\in \s^{d-1}(\vq,\rho) : \dist(\vx,\Pc_{\Sr} ) \leq C\, \eps\, \rho
			\right\} \cap \B^d\left(\vq\pm\rho\frac{\vq-\vq_{\Sr}}{|\vq-\vq_{\Sr}|},\frac{\lambda}{16}\rho \right) = \emptyset.
		\end{align*}
	\end{proof}
	
	As a corollary, we prove that $\vPhi(\Sr)$ is diffeomorphic to $\Sc_{\Pc}\coloneqq \Pc \cap \s^{d-1}(\vq,\rho)$. To do so, we consider the map $\iota\coloneqq \frac{ \vpip - \vq_{\Sr} }{| \vpip - \vq_{\Sr}|}\circ\vPhi\colon \Sr \to \Sc_{\Pc}$. We first show the following Claim \ref{cl:projection_equator} that this map is well-defined by Claim \ref{cl:slice_centered}. In Claim \ref{cl:projection_equator}, we then show that the map $\iota$ is locally injective. 
	
	\begin{claim}\label{cl:projection_equator}
		The map $\iota\colon \Sr\to \Sc_{\Pc}\coloneqq \Pc \cap \s^{d-1}(\vq,\rho)$ sending a point $x\in\Sr$ to the closest point in $\Sc_{\Pc}$ of $\vPhi(x)$ is well-defined and is an immersion.
	\end{claim}
	
	\begin{proof}
		We have the following description of $\iota$. Given $x\in \Sr$, we describe the set $\vPhi(\Sr(x))$ (we defined the set $\Sr(x)$ in Claim \ref{cl:graph_good_slice} and the set $\Sc_{x,\Sr}$ in \eqref{eq:domain_local_graph}) as follows:
		\begin{align}\label{eq:slice_graph}
			\vPhi(\Sr(x)) = \left\{ \vq_{\Sr} + \vy+\vu_x(\vy) : \vy\in \Sc_{x,\Sr}  \right\}.
		\end{align}
		Thanks to Claim \ref{cl:slice_centered}, if $\vy\in \Sc_{x,\Sr}$ it holds
		\begin{align*}
			\rho = \left|\vPhi-\vq\right| & \leq |\vy| + \left|\vu_x(\vy)+\vq_{\Sr}-\vq \right|   \leq |\vy| + \big(\delta + C(n,d,V)\, \eps\big) \rho.
		\end{align*}
		If $\eps>0$ is small enough, we obtain 
		\begin{align}\label{eq:norm_y}
			|\vy| \geq \frac{ 1-\delta}{2} \rho>0.
		\end{align}
		Thus, we have $0 \notin \Sc_{x,\Sr}$. Hence the map $\iota_{|\Sr(x)}$ is given by:
		\begin{align*}
			\forall \vz\in \Sc_{x,\Sr},\qquad \iota\big(\vq_{\Sr}+ \vz+\vu_x(\vz)\big)\coloneqq  \vq_{\Sr} + \rho \frac{\vz}{|\vz|} \in \Sc_{\Pc}.
		\end{align*}
		We deduce from Claim \ref{cl:graph_good_slice} that two points $\vy,\vz\in \Sc_{x,\Sr}$ close enough cannot be proportional. Otherwise, there would exist a point $\vy\in \Sc_{x,\Sr}$ such that the radial derivative of $\vPhi$ at $\vy$ exists and is given in radial coordinates $(r,\theta)\in(0,+\infty)\times \s^{n-1}$ by 
		\begin{align*}
			\dr_r\vPhi = \theta + \dr_r \vu_x(r,\theta).
		\end{align*}
		Differentiating the constraint $\rho^2 = \left|\vPhi-\vq\right|^2$, we obtain 
		\begin{align*}
			0 = \big(\vq_{\Sr}-\vq + \vy + \vu_x(r,\theta) \big) \cdot \big(\theta +\dr_r \vu_x(r,\theta) \big).
		\end{align*}
		Since $\vy=r\, \theta$ is orthogonal to $(\vq_{\Sr}-\vq)$ and to $\vu_x$, we obtain 
		\begin{align*}
			|\vy| \leq 2\, \rho\, |\g \vu_x| \leq C\, \eps\, \rho.
		\end{align*}
		This is in contradiction with \eqref{eq:norm_y} for $\eps_0>\eps>0$ small enough. Hence the map $\iota$ is locally injective. Moreover, the map $d\iota\colon T\Sr\to T_{\iota} \Sc_{\Pc}$ is injective. Indeed, we have
		\begin{align*}
			d\iota = \frac{d\vpi_{\Pc}\circ d\vPhi}{ |\vpi_{\Pc}\circ\vPhi - \vq_{\Sr}| } - \frac{\vpi_{\Pc}\circ\vPhi - \vq_{\Sr} }{ |\vpi_{\Pc}\circ\vPhi - \vq_{\Sr}|^2  } \scal{ \vpi_{\Pc}\circ\vPhi - \vq_{\Sr} }{ d\vpi_{\Pc}\circ d\vPhi }.
		\end{align*}
		If $X\in T\Sr$, then we have 
		\begin{align*}
			d\vpi_{\Pc}\circ d\vPhi(X) = d\vPhi(X) - \sum_{\alpha=1}^{d-n} \scal{\vn_{\alpha,\Sr} }{ d\vPhi(X) }\, \vn_{\alpha,\Sr}.
		\end{align*}
		On one hand, we have by Claim \ref{cl:osc_n} that
		\begin{align*}
			\left| d\vpi_{\Pc}\circ d\vPhi(X) \right| \geq (1-C(n,d,V)\, \eps) |d\vPhi(X)| = (1-C(n,d,V)\, \eps) |X|_{g_{\vPhi}}.
		\end{align*}
		On the other hand, since all the derivatives are taken along $\Sr$, it holds 
		\begin{align*}
			\left| \scal{ \vpi_{\Pc}\circ\vPhi - \vq_{\Sr} }{ d\vpi_{\Pc}\circ d\vPhi } \right| & = \left| \scal{ \vPhi - \sum_{\alpha=1}^{d-n} \scal{\vn_{\alpha,\Sr} }{ \vPhi }\, \vn_{\alpha,\Sr} }{d\vPhi(X) - \sum_{\alpha=1}^{d-n} \scal{\vn_{\alpha,\Sr} }{ d\vPhi(X) }\, \vn_{\alpha,\Sr} } \right| \\[2mm]
			& \leq C(n,d)\, \sum_{\alpha=1}^{d-n}  \left| \scal{\vn_{\alpha,\Sr} }{ d\vPhi(X) } \right|.
		\end{align*}
		Thanks to Claim \ref{cl:osc_n}, we obtain 
		\begin{align*}
			\left| \scal{ \vpi_{\Pc}\circ\vPhi - \vq_{\Sr} }{ d\vpi_{\Pc}\circ d\vPhi } \right| \leq C(n,d,V)\, \eps\, |X|_{g_{\vPhi}}.
		\end{align*}
		Coming back to $d\iota$, we obtain 
		\begin{align}\label{eq:non_degeneracy_diota}
			|d\iota(X)| \geq \frac{1-C(n,d,V)\, \eps}{|\vpi_{\Pc}\circ\vPhi - \vq_{\Sr}|}\, |X|_{g_{\vPhi}}.
		\end{align}
		For $\eps$ small enough, we deduce that $d\iota$ is injective. Since $\Sr$ and $\Sc_{\Pc}$ have the same dimension, we deduce that $d\iota$ is bijective. Since $\iota$ is $C^{\infty}$, the inverse function theorem implies that $\iota$ is an immersion.
	\end{proof}

	We now prove that $\iota$ is a covering map. Since $\Sr$ and $\Sc_{\Pc}$ have the same dimension, Claim \ref{cl:projection_equator} implies that $\iota(\Sr)$ is an immersed submanifold of codimension 0 in $\Sc_{\Pc}$. By \cite[Proposition 5.21]{lee2013}, the submanifold $\iota(\Sr)$ is embedded in $\Sc_{\Pc}$. Moreover, the differential of $\iota$ has full rank. Hence $\iota$ is locally surjective, see for instance \cite[Theorem 4.12]{lee2013}. Therefore, the set $\iota(\Sr)$ is open in $\Sc_{\Pc}$. It is also closed since $\iota$ is continuous and $\Sr$ is compact. Hence, the map $\iota\colon \Sr\to \Sc_{\Pc}$ a surjective submersion and thus a covering map. \\
	
	We now prove that $\iota$ is bijective\footnote{One can show that $\iota$ is a bijection more directly. First we show that any path $\gamma\colon [0,1]\to \Sc_{\Pc}$ can be lifted by $\iota$ to a path in $\Sr$ thanks to \eqref{eq:non_degeneracy_diota} (the set of $t>0$ such that $\gamma|_{[0,t]}$ can be lifted is an open and closed set in $[0,1]$ since $\iota$ is a local diffeomorphism). Moreover, the preimage of each point is a union of isolated points. Hence, if $\iota$ were not bijective, then we could find a closed path $\gamma\colon \s^1\to\Sc_{\Pc}$ that would be the image of a non-closed path in $\Sr$. Since $\Sc_{\Pc}$ is simply connected, $\gamma$ is homotopic to a point. This homotopy can be lifted by $\iota$ to a homotopy in $\Sr$. We would obtain that the image by $\iota$ of a continuous path in $\Sr$ is a given point in $\Sc_{\Pc}$, which is impossible since $\iota$ is a local diffeomorphism. }. Since $\Sr$ is compact, the number of sheets is constant, see for instance \cite[page 61]{hatcher2002}. By \cite[Proposition 1.32]{hatcher2002}, the number of sheets is given by the index of $\left(\vpip\circ\vPhi\right)_* [\pi_1(\Sr)]$ in $\pi_1(\Sc_{\Pc})$. Since $\pi_1(\Sc_{\Pc})=\{0\}$ (in dimension $n-1\geq 3$, the sphere $\s^{n-1}$ is simply connected), this index is $1$ and we obtain that $\iota$ is injective (see also \cite[Proposition A.79]{lee2013}). Therefore we must have that $\iota\colon \Sr\to \Sc_{\Pc}$ is a diffeomorphism. Hence, the map $\vPhi\colon \Sr\to \vPhi(\Sr)$ is injective, since if $x,y\in \Sr$ are such that $\vPhi(x)=\vPhi(y)$, then by definition of $\iota$, we also have $\iota(x)=\iota(y)$, so that $x=y$. Since $\vPhi$ is an immersion, the map $\vPhi\colon \Sr\to \vPhi(\Sr)$ is a diffeomorphism.\\
	
	Consequently, the set $\vPhi(\Sr)$ is globally the graph of a function $\vu\colon  \Sc_{\Sr}\subset \Pc_{\Sr}\to \Pc^{\perp}$, with $\Sc_{\Sr} \coloneq \pi_{\Pc}\circ\vPhi(\Sr)$ is diffeomorphic to $\Sc_{\Pc}$.
	We now record the estimates on $\vu$. The $L^{\infty}$ estimate follows from \eqref{eq:ThickSlice}. The estimate on the $L^{\infty}$-norm of $\g \vu$ follows from Claim \ref{cl:graph_good_slice}. Thus, we have 
	\begin{align}\label{eq:est_u_v1}
		\rho^{-1} \|\vu\|_{L^{\infty}(\Sc_{\Sr})} + \|\g \vu\|_{L^{\infty}(\Sc_{\Sr})} \leq C\, \eps.
	\end{align}
	The estimate on the Hessian of $\vu$ follows from Lemma \ref{lm:A_graph} and \eqref{eq:Ln_good_slice}.\\

	In order to complete the proof of Proposition \ref{pr:global_graph_slice}, we change the description of the slice, since the set $\Sc_{\Sr}$ might become irregular as $\vPhi$ approximate a weak immersion. Instead of describing $\vPhi\big|_{\Sr}$ as a graph over $\Sc_{\Sr}$, we will describe it as a deformation of the round $(n-1)$-sphere $\Sc_{\Pc} = \Pc_{\Sr}\cap \s^{d-1}(\vq,\rho)$. More precisely, we consider the map $\vphi\colon \Sc_{\Pc}\to \R^d$ such that
	\begin{align}\label{eq:vPhi_graph}
		\vPhi(\Sr) = \left\{ \vtheta + \vphi(\vtheta)  : \vtheta\in \Sc_{\Pc} \right\}.
	\end{align}
	By \eqref{eq:est_u_v1}, there exists a vector field $\vv\colon \Sc_{\Pc} \to T \s^{d-1}$ such that the map $\vphi$ is given by
	\begin{align}\label{eq:init_vj}
		\begin{cases} 
			\displaystyle \vphi(\vtheta) = \exp^{\s^{d-1}(\vq,\rho)}_{\vtheta}\left(\vv(\vtheta)\right)-\vtheta, \\[5mm]
			\displaystyle (\rho\, \eps)^{-1} \|\vv \|_{L^{\infty}(\Sc_{\Pc})} + \eps^{-1} \|\g \vv \|_{L^{\infty}(\Sc_{\Pc})} + \rho^{\frac{1}{n}}\, \|\g^2 \vv \|_{L^n(\Sc_{\Pc})} \leq C(n,d,V).
		\end{cases}
	\end{align}
	On $\s^{d-1}(\vq,\rho)$, the exponential map is given by
	\begin{align*}
		\exp^{\s^{d-1}(\vq,\rho)}_{\vtheta}\left(\vv(\vtheta) \right) = \vq+ \rho\left[ \cos\left( \frac{ |\vv(\vtheta)| }{\rho} \right)\, \frac{ \vtheta -\vq }{\rho} + \sin\left( \frac{ |\vv(\vtheta)| }{\rho} \right) \frac{\vv(\vtheta)}{|\vv(\vtheta)|} \right].
	\end{align*}
	We know show that the estimates \eqref{eq:est_u_v1} implies that $\vphi$ is small in $W^{1,\infty}$ and bounded in $W^{2,n}$.
	\begin{claim}\label{cl:est_perturbation1}
		It holds
		\begin{align*}
			(\rho\, \eps)^{-1} \|\vphi\|_{L^{\infty}(\Sc_{\Pc})} + \eps^{-1} \|\g \vphi\|_{L^{\infty}(\Sc_{\Pc})} + \rho^{\frac{1}{n}}\, \|\g^2 \vphi\|_{L^n(\Sc_{\Pc})} \leq C(n,d,V).
		\end{align*}
	\end{claim}
	
	\begin{proof}
		Up to dilation and translation, we can assume that $\vq=0$ and $\rho=1$. We expand $\cos$ and $\sin$ in Taylor expansion:
		\begin{align*}
			\exp^{\s^{d-1}}_{\vtheta}\left( \vv(\vtheta) \right) = \vtheta + \sum_{k=1}^{\infty} \frac{(-1)^k}{(2k)!} |\vv(\vtheta)|^{2k}\, \vtheta + \sum_{k=0}^{\infty} \frac{(-1)^k}{(2k+1)!} |\vv(\vtheta)|^{2k}\, \vv(\vtheta).
		\end{align*}
		Therefore, it holds
		\begin{align*}
			\vphi(\vtheta) = \sum_{k=1}^{\infty} \frac{(-1)^k}{(2k)!} |\vv(\vtheta)|^{2k}\, \vtheta + \sum_{k=0}^{\infty} \frac{(-1)^k}{(2k+1)!} |\vv(\vtheta)|^{2k}\, \vv(\vtheta).
		\end{align*}
		Claim \ref{cl:est_perturbation1} follows from \eqref{eq:init_vj}.
	\end{proof}

	\subsection{Extension of the good slice}\label{sec:extension}
	
	In this section, we extend each connected component of a good slice in a global graph over $\R^n$ minus a compact set using the procedure introduced in \cite{MarRiv2025}. Since \underline{we work with $C^{\infty}$-immersions in this section}, the graphs and the diffeomorphisms are actually of class $C^{\infty}$, but we record only the estimates required to pass to the limit in an approximation of a weak immersion by $C^{\infty}$-immersions. We prove the following result.
	
	\begin{proposition}\label{pr:extension}
		With the notations of Proposition \ref{pr:global_graph_slice}, we consider a connected component $\Sr\subset S_{\rho}$ and the function $\vphi_{\Sr}\in W^{1,\infty}\cap W^{2,n}(\Sc_{\Pc};\R^d)$ given by \Cref{it:Graph}. We can extend the map $\vphi_{\Sr}$ into a map $\vpsi_{\Sr}\in W^{1,\infty}\cap W^{2,n}_0(\Pc_{\Sr}\setminus B_{\Pc_{\Sr}}(\vq,\rho);\R^d)$\footnote{Since $\vPhi\in \Imm(\Sigma;\R^d)$, the extension $\vpsi_{\Sr}$ is also of class $C^{\infty}$.} satisfying $\vpsi_{\Sr}=0$ on $\Pc_{\Sr}\setminus B_{\Pc_{\Sr}}(\vq_{\Sr},2\rho)$ in addition to the following properties. 
		On each connected component of $\vPhi^{-1}(\s^{d-1}(\vq,\rho))$, we glue a copy of $\R^n\setminus \B^n_1$ using the sets
		\begin{align*} 
			\Gc_{\Sr}\coloneqq \Big\{ \vq_{\Sr} + \vx+\vpsi_{\Sr}(\vx): \vx\in \Pc_{\Sr}\setminus B_{\Pc_{\Sr}}(\vq,\rho) \Big\}.
		\end{align*} 
		We obtain a weak immersion $\vPsi \in \I_{0,(n,2)}(\tilde{\Sigma};\R^d)$ with image $\tilde{M}$ given by
		\begin{align*}
			\begin{cases}
				\displaystyle \tilde{\Sigma} \coloneqq \vPhi^{-1}\big( \B^d(\vq,\rho)\big)\cup \bigcup_{\substack{\Sr\subset S_{\rho}\\ \text{connected component} }} (\R^n\setminus \B^n_1) , \\[10mm]
				\displaystyle\tilde{M} \coloneqq \big(M\cap \B^d(\vq,\rho)\big) \cup \bigcup_{\substack{\Sr\subset S_{\rho}\\ \text{connected component} }} \Gc_{\Sr}.
			\end{cases}
		\end{align*}
		We have $\vPsi = \vPhi$ on $\vPhi^{-1}\big( \B^d(\vq,\rho)\big)$ and the estimate
		\begin{align*}
			\left\|\vII_{\vPsi} \right\|_{L^{(n,2)}\left(\tilde{\Sigma},g_{\vPsi} \right)} \leq C(n)\, \eps^{\frac{1}{2}}.
		\end{align*}
	\end{proposition}

	\begin{proof} 
		In order to apply \cite[Theorem 6.5]{MarRiv2025}, we need to prove the boundary estimates. To do so, we start by summarizing our setting. We fix a connected component $\Sr\subset S_{\rho}$. Let $\ve_1,\ldots,\ve_d$ be the canonical basis of $\R^d$. Up to a rotation, dilation and translation, we assume that
		\begin{align}\label{eq:normalization_vn}
			\vq =0,\qquad  \text{ and }\qquad  \rho=1, \qquad \text{ and } \qquad \forall \alpha\in\{1,\ldots,d-n\},\quad \vn_{\alpha,\Sr} = \ve_{n+\alpha}. 
		\end{align}
		By \Cref{cl:slice_centered}, there exists $\delta\in(0,1)$ depending only on $n$, $d$ and $V$ such that  
		\begin{align}\label{eq:qC_centered} 
				|\vq_{\Sr}|\leq \delta, \qquad \text{ and } \qquad 
				\vq_{\Sr}^{\ 1} = \cdots = \vq_{\Sr}^{\ n}=0. 
		\end{align}
		We will denote $\Bc(\vx,r) \coloneqq B_{\Pc_{\Sr}}(\vx,r)$. By \Cref{pr:global_graph_slice} and the normalizations \eqref{eq:normalization_vn}-\eqref{eq:qC_centered}, we have (for $\eps>0$ small enough)
		\begin{align*}
			\sup_{x\in \Sr}  \sum_{\alpha=n+1}^d \left|\vPhi(x)^{\alpha} - \vq_{\Sr}^{\ \alpha} \right| \leq C\, \eps.
		\end{align*}
		Let $\theta = (\theta_1,\ldots,\theta_{n-1})$ denotes coordinates on $\Sc_{\Pc}$. We define the following vector field
		\begin{align}\label{eq:def_vtau}
			\vtau(\theta) = \vPhi_* \left( \g^{g_{\vPhi}} |\vPhi| \Big|_{\Sr}\right) = \proj_{T\vPhi(\Sigma)}\left(\frac{\vPhi}{|\vPhi|}\right) = \proj_{T\vPhi(\Sigma)}\left(\vPhi \right).
		\end{align}
		We have in particular
		\begin{align}\label{eq:vtau_orthogonal}
			\vtau \cdot \dr_{\theta_i} \vPhi = \vPhi \cdot \dr_{\theta_i} \vPhi  = \dr_{\theta_i}\left( \frac{|\vPhi|^2}{2} \right) = 0.
		\end{align}
		We now estimate $\vtau$.
		
		\begin{claim}
			It holds
			\begin{align*}
				\left\| \vtau - \Id \right\|_{L^{\infty}(\Sc_{\Pc})} + \|\g \vtau - \g^{\Sc_{\Pc}} \Id \|_{L^n(\Sc_{\Pc})} \leq C(n,d,V)\, \eps.
			\end{align*}
		\end{claim}
		
		\begin{proof}
			First rotate, translate and dilate $\vPhi$ so that the graph is over $\Pc_{\Sr} = \R^n\times \{0\}^{d-n}$, i.e. $\vq_{\Cr}=0$. By \eqref{eq:qC_centered}, this will change the constants in the estimates only by constants depending on $n$, $d$ and $V$.\\
			For $\vtheta\in \Sc_{\Pc}$, we have 
			\begin{align*}
				\vtau(\vtheta) = \proj_{T\vPhi}(\vtheta +\vphi(\vtheta)).
			\end{align*}
			By Claim \ref{cl:est_perturbation1}, it holds
			\begin{align*}
				|\vtau(\vtheta) - \vtheta| & \leq C(n,d,V)\, \eps + \left| \proj_{\left(T\vPhi\right)^{\perp}}(\vtheta) \right|   \leq C(n,d,V)\, \eps + \sum_{\alpha=1}^{d-n} \left| \vtheta \cdot \vn_{\alpha} \right|   \leq C(n,d,V)\, \eps + \sum_{\alpha=1}^{d-n} \left| \vtheta \cdot (\vn_{\alpha} -\vn_{\alpha,\Sr})\right|.
			\end{align*}
			By \Cref{cl:osc_n}, we obtain 
			\begin{align*} 
				\|\vtau - (\vtheta-\vq_{\Sr})\|_{L^{\infty}(\Sc_{\Pc})}\leq C(n,d,V)\, \eps.
			\end{align*} 	
			We also have 
			\begin{align*}
				\g^{\Sc_{\Pc}}\vtau = -\left(\g \proj_{(T\vPhi)^{\perp}} \right)(\vtheta+\vphi) + \g^{\Sc_{\Pc}} \vtheta + \g^{\Sc_{\Pc}} \vphi.
			\end{align*}
			By \eqref{eq:Lnr_good_slice} and \Cref{cl:est_perturbation1}, we obtain the estimate
			\begin{align*}
				\|\g^{\Sc_{\Pc}}\vtau -  \g^{\Sc_{\Pc}} \vtheta\|_{L^n(\Sc_{\Pc})} \leq C(n,d,V)\, \eps.
			\end{align*}
		\end{proof}

		We now obtain Proposition \ref{pr:extension} by direct application of \cite[Theorem 6.5]{MarRiv2025}.
	\end{proof}

	\subsection{The interior of a good slice is the union of topological balls}\label{sec:Topology}

	In this section, we complete the proof of \Cref{th:Construction_chart} by proving that the connected components of $\vPhi^{-1}\big( \B^d(\vq,\rho)\big)$ are topological balls. To do so, we adapt the proof of \cite[Theorem 2]{chern1957} to find a Morse function with few critical points. We obtain the topology of the Euclidean ball thanks to the proof of Reeb theorem. Since \underline{we work with $C^{\infty}$-immersions in this section}, the diffeomorphism is actually of class $C^{\infty}$. 
	
	\begin{proposition}\label{pr:Topology}
		With the notation of Proposition \ref{pr:extension}, there exists $\eps_1>0$ depending only on $n$, $d$ and $V$ such that the following holds. Assume that 
		\begin{align}\label{eq:choice_ball}
			\Er\Big( \vPhi; \vPhi^{-1}\big( \B^d(\vq,2\rho)\big) \Big) \leq \eps_1^n,\qquad \text{ and }\qquad  \vol_{g_{\vPhi}}\left(  \vPhi^{-1}\big( \B^d(\vq,1)\big) \Big)  \right) \leq V.
		\end{align}
		Let $\Cr$ be a connected component of $\vPhi^{-1}\big(\B^d(\vq,\rho)\big)$.	Then $\Cr$ is diffeomorphic to $\B^n$.
	\end{proposition}
	
	\begin{proof}
		We split the proof in several steps. In Step 1, we define the notations for the inversions of $\vPhi$. In Step 2, we find a suitable inversion allowing the existence of a suitable Morse function. In Step 3, we use this Morse function to identify two possible topologies for the connected components of $\vPhi^{-1}\big(\B^d(\vq,\rho)\big)$. In Step 4, we use Morse theory to show that the topology must be the one of the ball. In Step 5, we conclude.\\
		
		\textit{Step 1: Inverting $\vPsi$ to obtain a closed submanifold.}\\
		Consider a point $\vQ\in \R^d\setminus \tilde{M}$ and the inversion $\iota_{\vQ}$ in $\R^d$ with respect to the ball $\B^d\big(\vQ,\frac{1}{2}\dist(\vQ;\B^d(\vq,2\rho) \big)$. Since the graphs $\Gc_{\Sr}$ are flat $n$-planes outside of a compact set, the immersion $\vpsi_{\vQ} \coloneqq \iota_{\vQ} \circ\vPsi$ can be considered as an immersion of the following domain:
		\begin{align*}
			\hat{\Sigma} \coloneqq \vPhi^{-1}\big( \B^d(\vq, \rho)\big)\cup \bigcup_{\substack{\Sr\subset S_{\rho}\\ \text{connected component} }} \s^n_+, 
		\end{align*}
		where $\s^n_+$ is a half-sphere. The manifold $\hat{\Sigma}$ is obtained from the manifold $\tilde{\Sigma}$ introduced in Proposition \ref{pr:extension} by replacing the copies of $\R^n\setminus \B^n_1$ by copies of $\s^n_+$. We consider $\hat{\Sigma}_0\subset \hat{\Sigma}$ a fixed connected component. We denote
		\begin{align*}
			\hat{\Sigma}_1 \coloneqq  \hat{\Sigma}_0\cap \vPsi^{-1}\big( \B^d(\vq,2\rho)\big),\qquad \hat{\Sigma}_2 \coloneqq \hat{\Sigma}_0\setminus \hat\Sigma_1.
		\end{align*}
		In the domain $\Sigma$, the set $\dr \hat{\Sigma}_1$ is diffeomorphic to $\vPhi^{-1}(\s^{d-1}(\vq,\rho))$, but in the target $\R^d$, the set $\vPsi(\dr \hat{\Sigma}_1)$ consists in a finite union of round $(n-1)$-dimensional spheres. \\
		
		\textit{Step 2: Construction of a suitable Morse function.}\\
		The manifold $\hat{\Sigma}_2$ is diffeomorphic to a disjoint union of half-spheres, and $\vPsi(\hat{\Sigma}_2)$ is totally geodesic in $\R^d$. Let $N\left(\vpsi_{\vQ} \right)$ be the unit normal bundle of $\vpsi_{\vQ}\big|_{\hat{\Sigma}_0}$ and $F_{\vQ}\colon N\left(\vpsi_{\vQ} \right)\to \s^{d-1}$ be the map assigning to a given vector $\xi\in N\left(\vpsi_{\vQ}\right)$ the corresponding unit vector in $\s^{d-1}$. Then, we have 
		\begin{align}\label{eq:total_det}
			\int_{N\left(\vpsi_{\vQ} \right)} \left| \det \vII_{\vpsi_{\vQ}}(\xi) \right|\, d\xi = \int_{\hat{\Sigma}_0} \left| (F_{\vQ})^*d\vol_{\s^{d-1}} \right|.
		\end{align}
		As explained in the proof of Theorem 2 in \cite{chern1957}, almost all of the directions $\ve\in \s^{d-1}$ are such that the following map is a Morse function on $\hat{\Sigma}_0$:
		\begin{align} \label{eq:Morse_function}
			h_{\ve,\vQ}\colon x\in \hat{\Sigma}_0 \mapsto \ve\cdot \vpsi_{\vQ}(x).
		\end{align} 
		We claim that we can choose $\vQ_0\in \R^d\setminus \tilde{M}$ and $\ve_0\in \s^{d-1}$ such that the map $h_{\ve_0,\vQ_0}$ has exactly two critical points per connected component of $\hat{\Sigma}_2$, none of them lying in $\hat{\Sigma}_1$. The crucial observation is that the map $h_{\ve,\vQ}$ has a critical point if and only if the vector $\ve$ lies in $F_{\vQ}^{-1}(\s^{d-1})$.\\    
		
		Let $I\geq 1$ be the number of connected component of $\hat{\Sigma}_2$. First we show that we can choose $\vQ_0\in \R^d\setminus \tilde{M}$ such that the quantity \eqref{eq:total_det} is close to the one of $I$ disjoint round $n$-spheres, meaning $2I\Hr^{d-1}(\s^{d-1})$:
		\begin{align}
			& \bullet \qquad \int_{N \left(\vpsi_{\vQ_0}\big|_{\hat{\Sigma}_1} \right)} \left| \det\II_{\vpsi_{\vQ_0}}(\xi) \right|\, d\xi \leq C(n,d,V)\, \eps_1 \int_{N \left(\vpsi_{\vQ_0}\big|_{\hat{\Sigma}_2} \right)} \left| \det\II_{\vpsi_{\vQ_0}}(\xi)\right|\, d\xi,\label{eq:choice_Q} \\[3mm]
			& \bullet\qquad \left| 2I\, \Hr^{d-1}(\s^{d-1}) - \int_{N\left( \vpsi_{\vQ_0} \right)} \left|\det\vII_{\vpsi_{\vQ_0}}(\xi) \right|\, d\xi \right|\leq  C(n,d,V)\, \eps_1. \label{eq:ineq_Chern}
		\end{align}
		The second fundamental form of $\vpsi_{\vQ}$ can be computed thanks to \cite[Proposition 1.2.1]{kuwert2012}, we have for every $\vQ\in \R^d\setminus \tilde{M}$,
		\begin{align*}
			\left| \vII_{\vpsi_{\vQ}} \right|_{g_{\vpsi_{\vQ}}}^2 \leq \frac{C}{|\g\iota_{\vQ}|^2}\left( \left| \vII_{\vPsi}\right|^2_{g_{\vPsi}} + \frac{1}{ |\vPsi-\vQ|^2} \right) \qquad \text{on }\hat{\Sigma}_1.
		\end{align*}
		We estimate the $L^1$-norm of the determinant of $\vII_{\vpsi_{\vQ}}$ in a direction $\xi\in N(\vpsi_{\vQ})$ by the $L^n$-norm of $\vII_{\vpsi_{\vQ}}$. We obtain
		\begin{align}
			\int_{N \left(\vpsi_{\vQ}\big|_{\hat{\Sigma}_1} \right)} \left| \det\vII_{\vpsi_{\vQ}}(\xi) \right|\, d\xi & \leq C \int_{\hat{\Sigma}_1} \left| \vII_{\vPsi}(\xi) \right|^n + \frac{1}{ \left|\vPsi-\vQ \right|^n} \, d\vol_{g_{\vPsi}}   \leq C\, \left( \eps^n_1 +\frac{\vol_{g_{\vPsi}}(\hat{\Sigma}_1) }{ \dist \left(\vQ,\tilde{M}\cap \B^d(\vq,\rho) \right)^n}\right). \label{eq:Energy_inversion}
		\end{align}
		Thanks to Proposition \ref{pr:Extrinsic_Hdiff}, we obtain
		\begin{align*}
			\forall R>\rho + |\vq|,\quad \forall \vQ \in \R^d\setminus \left(\B^d(0,R)\cup \tilde{M}\right),\qquad  \int_{N \left(\vpsi_{\vQ}\big|_{\hat{\Sigma}_1} \right)} \left| \det\vII_{\vpsi_{\vQ}}(\xi) \right|\, d\xi \leq C\, \left( \eps^n_1 +\frac{ \rho^n }{ R^n }\right).
		\end{align*}
		We now focus on $\hat{\Sigma}_2$. If $|\vQ|\geq R$ then then the set $\vpsi_{\vQ}(\hat{\Sigma}_2)$ is the union of $I$ open sets of round $n$-dimensional spheres of the form $\s^{n-1}(\vx,r)\setminus \B^d\left(\vy, 2\rho\right)$ for some $\vy\in \s^{n-1}(\vx,r)$ and $r\geq C\, \dist\big(\vQ, \B^d(\vq,2\rho)\big)$. Since the $L^1$-norm of $\det\vII(\xi)$ is scale-invariant, we obtain the following estimates. Taking $|\vQ|\geq \eps^{-1}_1$ and close enough to one of the graph $\Gc_{\Sr}$ (defined in Proposition \ref{pr:extension}), we obtain a point $\vQ_0\in \R^d\setminus \tilde{M}$ such that for some $C=C(n,d,V)>0$,
		\begin{align*}
			& \bullet \qquad \int_{N \left( \vpsi_{\vQ_0}\big|_{\hat{\Sigma}_1} \right)} \left| \det\vII_{\vpsi_{\vQ_0}}(\xi) \right|\, d\xi  \leq C\, \eps_1^n, \\[3mm]
			& \bullet \qquad \int_{N \left(\vpsi_{\vQ_0}\big|_{\hat{\Sigma}_2} \right)} \left| \det\vII_{\vpsi_{\vQ_0}}(\xi) \right|\, d\xi \geq \big( 1-C\, \eps_1\big)\, I \, \int_{N(\s^{n-1})} \left| \det \vII_{\s^{n-1}}(\xi) \right|\, d\xi = 2\, I\, (1-C\, \eps_1) \,\Hr^{d-1}(\s^{d-1}),\\[3mm]
			& \bullet \qquad \int_{N \left(\vpsi_{\vQ_0}\big|_{\hat{\Sigma}_2} \right)} \left| \det\vII_{\vpsi_{\vQ_0}}(\xi) \right|\, d\xi \leq I\, \int_{N(\s^{n-1})} \left| \det \vII_{\s^{n-1}}(\xi) \right|\, d\xi = 2\, I\, \Hr^{d-1}(\s^{d-1}).
		\end{align*}
		This is exactly \eqref{eq:choice_Q}-\eqref{eq:ineq_Chern}.\\

		Now that we have chosen $\vQ_0$, we choose $\ve_0$. We claim that for $\eps_1>0$ small enough, there exists $\ve_0\in \s^{d-1}$ such that $h_{\ve_0,\vQ_0}$ is a Morse function with exactly $2I$ critical points, all belonging to $\hat{\Sigma}_2$, each connected component carrying two of them. To do so, we decompose $\s^{d-1}$ is three disjoint measurable sets $U$, $V$ and $W$ as follows:
		\begin{align*}
			& U \coloneqq \left\{ \ve\in\s^{d-1} : \sharp F_{\vQ_0}^{-1}(\{\ve\}) \geq 3\, I\right\}, \\[3mm]
			& V \coloneqq \left\{ \ve \in\s^{d-1}\setminus U :  F_{\vQ_0}^{-1}(\{\ve\}) \cap \hat{\Sigma}_1 \neq \emptyset \right\} ,\\[3mm]
			& W \coloneqq \left\{ \ve\in\s^{d-1} :  F_{\vQ_0}^{-1}(\{\ve\}) \subset \hat{\Sigma}_2\quad \text{and }\quad   \sharp F_{\vQ_0}^{-1}(\{\ve \}) =2\, I\right\} .
		\end{align*}
		In order to show that $W$ is not empty, we show that $\Hr^{d-1}(U)$ and $\Hr^{d-1}(V)$ are small, so that $\Hr^{d-1}(W)$ is comparable to $\Hr^{d-1}(\s^{d-1})$. Hence, the set $W$ contains "most" of the points of $\s^{d-1}$.\\
		
		Since $\vpsi_{\vQ_0}(\hat{\Sigma}_0)\subset \R^d$ is closed, the map $F_{\vQ_0}$ is surjective (in other words, every direction $\ve\in \s^{d-1}$ is normal to $\vpsi_{\vQ_0}(\hat{\Sigma}_0)$ at some point). Thus, we have 
		\begin{align*}
			\int_{\hat{\Sigma}_0} \big| (F_{\vQ_0})^*d\vol_{\s^{d-1}}\big| & = \int_{\s^{d-1}} \sharp F_{\vQ_0} ^{-1} (\{\ve\})\ d\vol_{\s^{d-1}}(\ve) \\[3mm]
			& = \int_{U} \sharp F_{\vQ_0} ^{-1} (\{\ve\})\, d\vol_{\s^{d-1}}(\ve) + \int_{V} \sharp F_{\vQ_0} ^{-1} (\{\ve\})\, d\vol_{\s^{d-1}}(\ve) + \int_{W} \sharp F_{\vQ_0} ^{-1} (\{\ve \})\, d\vol_{\s^{d-1}}(\ve ).
		\end{align*}
		We first prove that $\Hr^{d-1}(V)$ is small. By definition of $V$, any point of $V$ is covered at most $(3I-1)$ times. Using \eqref{eq:choice_Q}, we obtain
		\begin{align*}
			\Hr^{d-1}(V) & \leq \int_{V} \sharp F_{\vQ_0} ^{-1} (\{\ve\})\ d\vol_{\s^{d-1}}(\ve) \\[2mm]
			& \leq (3I-1) \int_{V} \sharp\left( F_{\vQ_0} ^{-1} (\{\ve\}) \cap \hat{\Sigma}_1 \right)\ d\vol_{\s^{d-1}}(\ve) \\[2mm]
			& \leq (3I-1)\int_{\hat{\Sigma}_1} \big| (F_{\vQ_0})^*d\vol_{\s^{d-1}}\big| \leq  C(n,d,V)\, \eps_1\, I.
		\end{align*}
		Moreover, the volume bound \eqref{eq:vol_good_slice} together with \Cref{pr:global_graph_slice} implies that the number of $I$ is bounded from above by a constant depending only on $n$, $d$ and $V$. We end up with the following relations:
		\begin{align*}
			\begin{cases} 
				\displaystyle \Hr^{d-1}(V) \leq C(n,d,V)\, \eps_1,\\[3mm]
				\displaystyle C(n,d,V)\, \eps_1 + 2\, I\, \Hr^{d-1}(\s^{d-1}) \geq 3\, I\, \Hr^{d-1}(U) + 2\, I\, \Hr^{d-1}(W) , \\[3mm]
				\displaystyle \Hr^{d-1}(U) + \Hr^{d-1}(V)+ \Hr^{d-1}(W) = \Hr^{d-1}(\s^{d-1}).
			\end{cases}
		\end{align*}
		Plugging the last relation into the left-hand side of the second one, and then dividing by $I$, we obtain
		\begin{align*}
			C(n,d,V)\,  \eps_1 + 2\, \Hr^{d-1}(U) +2\, \Hr^{d-1}(W) \geq 3\, \Hr^{d-1}(U) + 2\, \Hr^{d-1}(W).
		\end{align*}
		Hence, we obtain 
		\begin{align}\label{eq:size_W} 
				\Hr^{d-1}(U\cup V) \leq  C(n,d,V)\, \eps_1,\qquad \text{ and } \qquad
				\Hr^{d-1}(W) \geq \Hr^{d-1}(\s^{d-1}) -C(n,d,V)\, \eps_1. 
		\end{align}
		For $\eps_1>0$ small enough, we must have a direction $\ve_0\in \s^{d-1}$ such that $h_{\ve_0,\vQ_0}$ is a Morse function with exactly $2I$ critical points, all belonging to $\hat{\Sigma}_2$. Since the $I$ connected components $\Cr_1,\ldots,\Cr_I$ of $\vpsi_{\vQ_0}(\hat{\Sigma}_2)$ are open sets of  round spheres $S_1,\ldots,S_I$, each function $h_{\ve_0,\vQ_0}\big|_{\Cr_i}$ extends to a Morse function $h_i$ on the full sphere $S_i$ of the form \eqref{eq:Morse_function} and thus has exactly 2 critical points. All the critical points of $h_i$ must be critical points of $h_{\ve_0,\vQ_0}\big|_{\Cr_i}$, otherwise the function $h_{\ve_0,\vQ_0}$ would have strictly less than $2I$ critical points. Hence, $h_{\ve_0,\vQ_0}$ has exactly 2 critical points on each set $\Cr_i$, each of them being a local extremum. \\
		
		\textit{Step 3: The boundary of every connected component $\Cr\subset \hat{\Sigma}_1$ has either one or two connected component.}\\
		By \cite[Theorem 3.1]{milnor1963}, any connected component $\Cr$ of $\hat{\Sigma}_1$ is strictly included in a trivial cobordism $Y^n = X^{n-1}\times [0,1]\subset \hat{\Sigma}$ (with $X$ closed and connected) containing no critical point of $h_{\ve_0,\vQ_0}$. Let $S\subset \dr \Cr$ be a connected component (diffeomorphic to $\s^{n-1}$) and consider $B\subset \hat{\Sigma}_2$ a connected component (diffeomorphic to $\B^n$) such that $\dr B=S$. Since $B$ contains 2 critical points of $h_{\ve_0,\vQ_0}$, $B$ cannot be fully contained in $Y$. Moreover, each connected component of $Y \setminus \Cr$ has a boundary in $\dr \Cr$. By Jordan--Brower Separation theorem, $\dr\Cr$ separates $Y$ into exactly two connected component, one containing $\Cr$ and the other one containing at least one connected component of $\dr Y=X\times \{0,1\}$. Since $\dr Y$ has exactly two connected component, $\dr \Cr$ has either one or two connected components. In the second case, $\Cr$ is diffeomorphic to a open set of the form $X\times (a,b)$, with $X\simeq \s^{n-1}$.\\

		\textit{Step 4: $\Cr$ cannot be diffeomorphic to $\s^{n-1}\times [0,1]$.}\\
		Assume $\Cr\simeq \s^{n-1}\times [0,1]$. Then the image of $\vpsi_{\vQ_0}$ consists in exactly two pieces of rounds spheres glued together by $\Cr$, thus we have that $\hat{\Sigma}$ is homeomorphic to $\s^n$. Moreover, the Morse function $h_{\ve_0,\vQ_0}$ has exactly two distinct strict local minima. Let $C_k$ be the number of critical points of $h_{\ve_0,\vQ_0}$ of index $k$. By Morse inequalities (see for instance Equation $(4_{\lambda})$ on page 30 in \cite{milnor1963}), we have 
		\begin{align*}
			C_1\geq C_0 + \dim H^1(\s^n) - \dim H^0(\s^n) = 1.
		\end{align*}
		In other words, $h_{\ve_0,\vQ_0}$ has at least one critical point of index 1. Since $\vpsi_{\ve_0,\vQ_0}(\hat{\Sigma}_2)$ consists in two pieces of round spheres, the map $h_{\ve_0,\vQ_0}|_{\hat{\Sigma}_2}$ has no critical point of index 1 (there is only extrema). Thus, $h_{\ve_0,\vQ_0}|_{\hat{\Sigma}_1}$ must have a critical point, which is impossible by construction.\\
		
		\textit{Step 5: $\Cr$ is diffeomorphic to a ball.}\\
		Since $\hat{\Sigma}_2$ consists in a disjoint union of open sets, all having a common boundary with $\Cr$, we deduce that $\hat{\Sigma}_2$ has only one connected component and is diffeomorphic to a ball $\B^n$. Consequently, the Morse function $h_{\ve_0,\vQ_0}$ has exactly 2 critical points. Following the proof of Reeb Theorem \cite[Theorem 4.1]{milnor1963}, we deduce that $\hat{\Sigma}_0\setminus \{p\}$ is diffeomorphic to a ball\footnote{The statement of Reeb theorem only shows that $\hat{\Sigma}$ is \textit{homeomorphic} to $\s^n$. However, the proof shows that if $m_1\coloneq h_{\ve_0,\vQ_0}(p)$ is the maximum value of $h_{\ve_0,\vQ_0}$ and $m_0\coloneq \min h_{\ve_0,\vQ_0}$, then the set $\hat{\Sigma}\setminus h_{\ve_0,\vQ_0}^{-1}([m_0,m_1-\delta])$ is diffeomorphic to $\B^n$ for any $\delta>m_1-m_0$. Indeed, by Theorem 3.1 p.12 in \cite{milnor1963}, the set $\hat{\Sigma}\setminus h_{\ve_0,\vQ_0}^{-1}([m_0,m_1-\delta])$ is diffeomorphic to $\hat{\Sigma}\setminus h_{\ve_0,\vQ_0}^{-1}([m_0,m_0+\delta])$. For $\delta>0$ small enough, the set $\hat{\Sigma}\setminus h_{\ve_0,\vQ_0}^{-1}([m_0,m_0+\delta])$ is diffeomorphic to $\B^n$ by Morse Lemma, see Lemma 2.2 p.6 in \cite{milnor1963}.}
		, where $p$ is the point reaching the maximum of $h_{\ve_0,\vQ_0}$. Hence, we have the diffeomorphisms $\Cr \simeq \hat{\Sigma}_0\setminus \B^n \simeq \hat{\Sigma}_0\setminus \{p\}$, which is also diffeomorphic to a ball.
	\end{proof}

	\subsection{Density of the connected components in the interior of good slices}\label{sec:density}
	
	In Proposition \ref{pr:Topology}, we proved that every connected component of $\vPhi^{-1}\big(\B^d(\vq, \rho)\big)$ is diffeomorphic to a ball, but without any information on the parametrization $\vPhi$ on these connected components. We now prove that for every connected components $\Cr\subset \vPhi^{-1}\big(\B^d(\vq,\rho)\big)$, the map $\vPhi\colon \Cr\to \B^d(\vq,\rho)$ is injective. The density of the associated varifold to $\vPhi(\Cr)$ is defined at every point and is equal to the number of preimages by $\vPhi$ since \underline{we work with $C^{\infty}$-immersions in 
		this section}. Hence, this amounts to prove that this density is equal to one at every point of $\vPhi(\Cr)$.

	\begin{proposition}\label{pr:Density}
		With the notations of Proposition \ref{pr:Topology}, we consider $\Cr$ a connected component of $\vPhi^{-1}(\B^d(\vq,\rho))$. Up to reducing $\eps_1$, the density of $\vPhi(\Cr)$ at any point $\vep_1\in \vPhi(\Cr)\setminus \s^{d-1}(\vq,\rho)$ is equal to one.
	\end{proposition}
	
	\begin{proof} 
		By \Cref{sec:extension} and \Cref{sec:Topology}, we can assume that $\Cr$ is identified with the ball $\B^n(0,\rho)$ and that $\vPhi\big|_{\B^n(0,\rho)}$ has already been extended to $\R^n$, with $\vPhi\big|_{\R^n\setminus \B^n(0,2\rho)}$ parametrizing a flat affine $n$-dimensional plane. We also have the global estimate
		\begin{align}\label{eq:small_extension}
			\begin{cases} 
				\displaystyle \left\| \vII_{\vPhi} \right\|_{L^{(n,2)}(\B^n(0,2\rho),g_{\vPhi})} \leq C(n,d,V)\, \eps_1, \\[4mm]
				\displaystyle \vII_{\vPhi} = 0 \qquad \text{ in }\R^n\setminus \B^n(0,2\rho).
			\end{cases} 
		\end{align}
		We now consider the monotonicity formula (see \eqref{eq:montonicity_balls}) for $0<s<t$ near a point $\vep_1\in \vPhi(\B^n(0,\rho))$,
		\begin{align*}
			\frac{\mu_{\vPhi}( \B^d(\vep_1,s))}{\Hr^n(\B^n)\, s^n} \leq \frac{\mu_{\vPhi}( \B^d(\vep_1,t))}{\Hr^n(\B^n)\, t^n} + \frac{1}{2}\int_{\B^d(\vep_1,t)} \frac{|\vH|^2}{|\vx-\vq|^{n-2}}\, d\mu_{\vPhi}(\vx) + \left( \frac{\mu(\B^d(\vep_1,s))}{\Hr^n(\B^n)\, s^n} \right)^{\frac{n-1}{n}} \frac{ \|\vH\|_{L^n(\B^d(\vep_1,s),\mu_{\vPhi})} }{ \Hr^n(\B^n)^{1/n} }.
		\end{align*}
		Letting $s\to 0$, we obtain 
		\begin{align*}
			\theta_{\vPhi(\Cr)}(\vep_1) \leq \frac{\mu_{\vPhi}( \B^d(\vep_1,t))}{\Hr^n(\B^n)\, t^n} + \frac{1}{2}\int_{\B^d(\vep_1,t)} \frac{|\vH|^2}{|\vx-\vq|^{n-2}}\, d\mu_{\vPhi}(\vx) .
		\end{align*}
		Letting $t\to +\infty$, we deduce from \eqref{eq:small_extension} and \Cref{it:Centre} in Proposition \ref{pr:global_graph_slice}, that the following estimate holds for some $\kappa>2$ depending only on $n$, $d$ and $V$
		\begin{align*}
			\theta_{\vPhi(\Cr)}(\vep_1) & \leq 1 + \frac{1}{2}\int_{\B^d(\vq,\kappa\, \rho)} \frac{|\vH|^2}{|\vx-\vq|^{n-2}}\, d\mu_{\vPhi}(\vx)  \leq 1+ \frac{1}{2}\, \left\| \frac{1}{|\cdot - \vq|} \right\|_{L^{(n,\infty)}\left(\vPhi(\R^n)\cap \B^d(\vq,\kappa\, \rho),\mu_{\vPhi} \right)}^{n-2}\, \|\vH \|_{L^{(n,2)}\left(\B^d(\vq,\kappa\, \rho), \mu_{\vPhi} \right)}^2 .
		\end{align*}
		By Proposition \ref{pr:Extrinsic_Hdiff}, we obtain 
		\begin{align*}
			\theta_{\vPhi(\Cr)}(\vep_1) \leq 1 + C(n,d,V)\, \eps_1^2.
		\end{align*}
		Up to reducing $\eps_1$, we obtain $\theta_{\vPhi(\Cr)}(\vep_1) \leq \frac{3}{2}$. This $\theta_{\vPhi(\Cr)}(\vep_1)$ is a non-zero integer, we obtain $\theta_{\vPhi(\Cr)}(\vep_1)=1$.
	\end{proof}
	
	\subsection{Coordinates inside a good slice}
	
	In this section, we prove \Cref{it:Coordinates} of \Cref{th:Construction_chart}, where \underline{the immersion $\vPhi$ is still assumed to be $C^{\infty}$}. We use the notations of \Cref{th:Construction_chart} with the assumption $r=1$: we consider $\Cr\subset \vPhi^{-1}\left(\B^d(\vq,1)\right)$ a connected component, diffeomorphic to $\B^n$. As proved in \Cref{sec:extension,sec:Topology}, we have a diffeomorphism $u\colon \B^n(0,1)\to \Cr$ such that the metric $g_{\vPhi\circ u}$ verifies on $\s^{n-1}$
	\begin{align}\label{eq:init_metric}
		\forall i,j\in\{1,\ldots,n\},\qquad \left\|  \left( g_{\vPhi\circ u}\right)_{ij} - \delta_{ij} \right\|_{L^{\infty}(\s^{n-1})} \leq C(n,d,V)\, \Er \big( \vPhi; \Cr \big)^{\frac{1}{n}}.
	\end{align}
	We also have 
	\begin{align}\label{eq:init_metric2}
		\forall i,j\in\{1,\ldots,n\},\qquad \left\|  \left( g_{\vPhi\circ u}\right)_{ij} - \left( g_{\Sc_{\Pc}} + dr^2 \right)_{ij} \right\|_{W^{1,n}(\s^{n-1})} \leq C(n,d,V).
	\end{align}
	Indeed, the metric $g_{\vPhi\circ u}$ is given by the following formula, using the notations of \Cref{sec:extension} with local coordinates $(\theta_1,\ldots,\theta_{n-1})$
	\begin{align*}
		\left( g_{\vPhi\circ u}|_{\s^{n-1}}\right)_{\theta_i\, \theta_j} =  \left( g_{\Sc_{\Pc}} \right)_{\theta_i\, \theta_j} + \dr_{\theta_i} \vtau \cdot\dr_{\theta_j} \vtau + \dr_{\theta_i} \vphi\, \dot{\otimes}\,  \left(  \dr_{\theta_j} \vtheta + \dr_{\theta_j} \vphi \right) + \dr_{\theta_i} \vtheta\, \dot{\otimes}\,  \dr_{\theta_j} \vphi .
	\end{align*}
	By \cite[Theorem 6.5]{MarRiv2025}, we obtain a diffeomorphism $f\colon \Or\to \B^n(0,1)$ such that $\vp\coloneq u\circ f\colon \Or\to \Cr$ verifies
	\begin{align}\label{eq:new_metric}
		\left\| \left(g_{\vPhi\circ \vp} \right)_{ij} - \delta_{ij} \right\|_{L^{\infty}(\Or)} + \left\| \dr_k \left( g_{\vPhi\circ \vp} \right)_{ij}  \right\|_{L^{(n,1)}(\Or)} +  \left\| \dr^2_{kl} \left(g_{\vPhi\circ \vp}\right)_{ij}  \right\|_{L^{\left( \frac{n}{2},1\right)}(\Or)} \leq C(n)\, \Er \big( \vPhi; \Cr \big)^{\frac{1}{n}}.
	\end{align}
	The regularity of $\dr \Or$ is a consequence of Remark 6.6 in \cite{MarRiv2025} and is obtained as follows. Combining \eqref{eq:new_metric} and \eqref{eq:init_metric}-\eqref{eq:init_metric2}, we obtain that the change of coordinates $ f= u^{-1}\circ \vp$ verifies for any $1\leq i,j,k\leq n$,
	\begin{align}
		& \left\| \dr_i f\cdot \dr_j f - \delta_{ij} \right\|_{L^{\infty}(\dr \Or)} \leq C(n,d,V)\, \Er \big( \vPhi; \Cr \big)^{\frac{1}{n}}. \label{eq:f1} 
	\end{align}
	If $\eps_0>\Er(\vPhi\circ\vp;\Or)>0$ is small enough (depending only on $n$, $d$ and $V$), we obtain the pointwise inequality between matrices
	\begin{align*}
		\frac{1}{2}\, \delta_{ij} < \dr_i f\cdot \dr_j f < 2\, \delta_{ij}.
	\end{align*}
	Thus, the inverse matrix also verifies the pointwise inequalities between matrices
	\begin{align*}
		\frac{1}{2}\, \delta_{ij} < \dr_i f^{-1}\cdot \dr_j f^{-1} < 2\, \delta_{ij}.
	\end{align*}
	Taking $i=j$, we deduce that 
	\begin{align*}
		\|\dr_i f\|_{L^{\infty}(\dr \Or)} + \|\dr_i f^{-1}\|_{L^{\infty}(\s^{n-1})}\leq 4.
	\end{align*}
    In order to obtain estimates on the second derivatives of $f$, we compare the Christoffel symbols of $g_{\vPhi\circ\vp}$ and $g_{\vPhi\circ u}$. Thanks to \cite[Equation (2.13)]{taylor2006}, we have 
    \begin{align*}
        \dr^2_{ij} (f^{-1})^k = \big( {^{g_{\vPhi\circ\vp}} \Gamma}_{ij}^p \big)\, \dr_p (f^{-1})^k - \big( {^{g_{\vPhi\circ u}} \Gamma^k_{\alpha\beta}} \big)\, \left( \dr_i (f^{-1})^{\alpha} \right)\, \left( \dr_j (f^{-1})^{\beta} \right).
    \end{align*}
	Consequently, the estimates \eqref{eq:init_metric2} and \eqref{eq:new_metric} imply that 
	\begin{align*}
		\|\dr^2_{ij} f^{-1}\|_{L^n(\dr \Or)} \leq C(n,d,V).
	\end{align*}
	Therefore, the map $f^{-1}\colon \s^{n-1}\to \dr\Or$ is a diffeomorphism of class $W^{2,n}(\s^{n-1})$ and by Sobolev injections, also of class $C^{1,\frac{1}{n}}$. In $\Or$, the map $f^{-1}\colon \B^n(0,1)\to \Or$ verifies $g_{\vPhi\circ u} = (f^{-1})^*g_{\vPhi\circ\vp}$. Therefore, we have
	\begin{align*}
		\|f^{-1}\|_{W^{1,n}(\B^n(0,1))} \leq C(n)\, \|g_{\vPhi}\|_{W^{1,n}(\Cr)}.
	\end{align*}

	\section{Atlas for weak immersions}\label{sec:Atlas}
	
	In this section, we iterate \Cref{th:Construction_chart} to construct an atlas on the domains of $C^{\infty}$-immersions with bounded energy. We then proceed by an approximation argument to obtain an atlas for weak immersions. Since the boundary conditions for these charts will be provided by \Cref{th:Construction_chart}, the regularity of the underlying differential structure is not obvious. We will prove that the differential structure of the atlas we construct is $C^1$-diffeomorphic to the original structure of class $C^{\infty}$ of the domain. More precisely, we prove the following result.
	
	\begin{theorem}\label{th:Atlas}
		Let $E,D>0$ and $\Sigma$ be a closed oriented manifold of class $C^{\infty}$ with even dimension $n\geq 4$. Consider $\vPhi\in \I_{\frac{n}{2}-1,2}(\Sigma;\R^d)$ such that $\Er(\vPhi)\leq E$ and $\diam(\vPhi(\Sigma))\leq D$. Then there exists $\eps_0>0$ depending only on $n$, $d$, $D$ and $E$ together with an atlas $(U_{\alpha},\vp_{\alpha})_{\alpha\in A}$ of $\Sigma$ satisfying the following properties:
		\begin{enumerate}
			\item\label{it:Harmonic} For each $\alpha\in A$, the chart $\vp_{\alpha}\colon \Or_{\alpha}\subset \R^n \to U_{\alpha}\subset \Sigma$ provides harmonic coordinates and satisfies the estimates of \Cref{it:Coordinates} in \Cref{th:Construction_chart}. In particular, the open set $U_{\alpha}$ is bi-Lipschitz homeomorphic to the Euclidean ball $\B^n$.
			
			\item\label{it:Boundary} For each $\alpha\in A$, there exist a ball $\B^d(\vep_{\alpha},r_{\alpha})$ such that $\vPhi\colon U_{\alpha}\to \B^d(\vep_{\alpha},r_{\alpha})$ is injective and 
			\begin{align*} 
				\Er\left( \vPhi; \vPhi^{-1}\big( \B^d(\vep_{\alpha},2r_{\alpha}) \big) \right)\leq \eps_0^n.
			\end{align*} 
			Moreover, there exists $\delta\in(0,1)$ depending only on $n$, $d$ and $E$, and an affine $n$-dimensional plane $\Pc_{\alpha}\subset \R^d$ such that $\vPhi(\dr U_{\alpha})$ is described by a map $\vphi_{\alpha}\colon \Sc_{\alpha}\to \R^d$, where $\Sc_{\alpha}\coloneq \Pc_{\alpha}\cap \s^{d-1}(\vep_{\alpha},r_{\alpha})$ satisfying the following properties 
			\begin{align*}
				\vPhi(\dr U_{\alpha}) = \left\{ \vtheta + \vphi_{\alpha}(\vtheta) : \vtheta \in \Sc_{\alpha}
				\right\},
			\end{align*}
			together with the following estimates
			\begin{align*}
				\begin{cases} 
					\eps_0^{-1} \left[ r_{\alpha}^{-1} \|\vphi_{\alpha}\|_{L^{\infty}(\Sc_{\alpha})} + \|\g \vphi_{\alpha}\|_{L^{\infty}(\Sc_{\alpha})} \right] + r_{\alpha}^{\frac{1}{n}} \|\g^2 \vphi_{\alpha}\|_{L^n(\Sc_{\alpha})} \leq C_0(n,d,E,D),\\[2mm]
					\dist(\vep_{\alpha},\Pc_{\alpha}) \leq \delta\, r_{\alpha}.
				\end{cases}
			\end{align*}
			
			\item\label{it:Transition} There exists a constant $C_1>0$ depending only on $n$ and $d$ such that the following holds. Let $\alpha,\beta\in A$ be such that $U_{\alpha}\cap U_{\beta}\neq \emptyset$. Then the transition charts $\vp_{\alpha\beta}\coloneqq \vp_{\alpha}^{-1}\circ\vp_{\beta}\colon \vp_{\beta}^{-1}(U_{\alpha}\cap U_{\beta})\to \vp_{\alpha}^{-1}(U_{\alpha}\cap U_{\beta})$ are of class $W^{3,\left( \frac{n}{2},1\right)}$ with the estimate
			\begin{align*}
				\begin{cases} 
					\displaystyle \left\| \vp_{\alpha\beta} \right\|_{L^{\infty}\left( \vp_{\beta}^{-1}(U_{\alpha}\cap U_{\beta})\right)} \leq 1, \\[3mm]
					\displaystyle \left\| \g \vp_{\alpha\beta} \right\|_{L^{\infty}\left( \vp_{\beta}^{-1}(U_{\alpha}\cap U_{\beta})\right)} \leq 1 + C_1\, \eps_0,\\[3mm]
					\displaystyle  \left\| \dr_i \vp_{\alpha\beta}\cdot \dr_j \vp_{\alpha\beta} - \delta_{ij} \right\|_{L^{\infty}\left( \vp_{\beta}^{-1}(U_{\alpha}\cap U_{\beta})\right)} \leq C_1\, \eps_0,\\[3mm]
					\displaystyle  \left\| \g^2 \vp_{\alpha\beta}\right\|_{L^{(n,1)}\left( \vp_{\beta}^{-1}(U_{\alpha}\cap U_{\beta}),\geu\right)} \leq C_1\, \eps_0,\\[3mm]
					\displaystyle  \left\| \g^3 \vp_{\alpha\beta}\right\|_{L^{\left( \frac{n}{2},1\right)}\left( \vp_{\beta}^{-1}(U_{\alpha}\cap U_{\beta}),\geu\right)} \leq C_1\, \eps_0.
				\end{cases}
			\end{align*}
			
			\item\label{it:Regularity} The induced differential structure is $C^1$-diffeomorphic to the initial $C^{\infty}$ structure of $\Sigma$.
			
			\item\label{it:Size} The cardinal of $A$ is bounded from above as long as the following number is bounded from above:
			\begin{align*}
				\Ec(\vPhi)\coloneqq \Er(\vPhi)\, \diam(\vPhi(\Sigma))^n\, \inf\left\{ \frac{1}{r^n} : \forall \vx\in \R^d,\quad \Er\Big( \vPhi;\vPhi^{-1}\big(\B^d(\vx,r)\big) \Big) \leq \eps_0 \right\}.
			\end{align*}
		\end{enumerate}
	\end{theorem}
	
	The main difference between \Cref{th:Atlas} and \Cref{th:Construction_chart} is the assumption on the volume. Thanks to \eqref{eq:diameter_energy_volume}, if $\Sigma$ is closed, then we have 
	\begin{align}\label{eq:est_Volume}
		\sup_{\vx\in\R^d} \vol_{g_{\vPhi}} \left( \vPhi^{-1}\big(\B^d(\vx,1)\big) \right) \leq \vol_{g_{\vPhi}}(\Sigma) \leq C(n,d)\, \diam(\vPhi(\Sigma))\, \Er(\vPhi).
	\end{align}
	Hence, if $\diam(\vPhi(\Sigma)) \leq D$ is fixed, then we can apply \Cref{th:Construction_chart}. However, for what follows, we will need to consider the case where $\Sigma$ is open. We record the following version of \Cref{th:Atlas}.
	
	\begin{theorem}\label{th:Atlas2}
		Let $V>0$ and $\Sigma$ be an oriented manifold (possibly with boundary) of class $C^{\infty}$ and having even dimension $n\geq 4$. Consider $\vPhi\in \I_{\frac{n}{2}-1,2}(\Sigma;\R^d)$ such that 
		\begin{align}\label{hyp:vol}
			\sup_{\vx\in\R^d} \vol_{g_{\vPhi}} \left( \vPhi^{-1}\big(\B^d(\vx,1)\big) \right) \leq V.
		\end{align}
		Given $\vartheta>0$, we consider
		\begin{align*}
			\Sigma_{\vartheta} \coloneqq \left\{ x\in \Sigma : \dist_{\R^d}\left(\vPhi(x) , \vPhi(\dr \Sigma)\right) > \vartheta \right\}.
		\end{align*}
		Then there exists $\eps_0>0$ depending only on $n$, $d$ and $V$ together with an atlas $(U_{\alpha},\vp_{\alpha})_{\alpha\in A}$ of $\Sigma_{\vartheta}$ satisfying the following properties:
		\begin{enumerate}
			\item\label{it:Harmonic2} For each $\alpha\in A$, the map $\vp_{\alpha}\colon \Or_{\alpha}\to U_{\alpha}\subset \Sigma$ provides harmonic coordinates and satisfy the estimates of \Cref{it:Coordinates} in \Cref{th:Construction_chart}. In particular, the open set $U_{\alpha}$ is bi-Lipschitz homeomorphic to the Euclidean ball $\B^n$.
			
			\item\label{it:Boundary2} For each $\alpha\in A$, there exist a ball $\B^d(\vep_{\alpha},r_{\alpha})$ such that $\vPhi\colon U_{\alpha}\to \B^d(\vep_{\alpha},r_{\alpha})$ is injective and 
			\begin{align*} 
				\Er\left( \vPhi; \vPhi^{-1}\big( \B^d(\vep_{\alpha},2r_{\alpha}) \big) \right)\leq \eps_0^n.
			\end{align*} 
			Moreover, there exists $\delta\in(0,1)$ depending only on $n$, $d$ and $E$, and an affine $n$-dimensional plane $\Pc_{\alpha}\subset \R^d$ such that $\vPhi(\dr U_{\alpha})$ is described by a map $\vphi_{\alpha}\colon \Sc_{\alpha}\to \R^d$, where $\Sc_{\alpha}\coloneq \Pc_{\alpha}\cap \s^{d-1}(\vep_{\alpha},r_{\alpha})$ satisfying the following properties 
			\begin{align}\label{eq:pseudo_graph}
				\vPhi(\dr U_{\alpha}) = \left\{ \vtheta + \vphi_{\alpha}(\vtheta) : \vtheta \in \Sc_{\alpha}
				\right\},
			\end{align}
			together with the following estimates
			\begin{align*}
				\begin{cases} 
					\eps_0^{-1} \left[ r_{\alpha}^{-1} \|\vphi_{\alpha}\|_{L^{\infty}(\Sc_{\alpha})} + \|\g \vphi_{\alpha}\|_{L^{\infty}(\Sc_{\alpha})} \right] + r_{\alpha}^{\frac{1}{n}} \|\g^2 \vphi_{\alpha}\|_{L^n(\Sc_{\alpha})} \leq C_0(n,d,E,D),\\[2mm]
					\dist(\vep_{\alpha},\Pc_{\alpha}) \leq \delta\, r_{\alpha}.
				\end{cases}
			\end{align*}
			
			\item\label{it:Transition2} There exists a constant $C_2>0$ depending only on $n$ and $d$ such that the following holds. Let $\alpha,\beta\in A$ be such that $U_{\alpha}\cap U_{\beta}\neq \emptyset$. Then the transition charts $\vp_{\alpha,\beta}\coloneqq \vp_{\alpha}^{-1}\circ\vp_{\beta}\colon \vp_{\beta}^{-1}(U_{\alpha}\cap U_{\beta})\to \vp_{\alpha}^{-1}(U_{\alpha}\cap U_{\beta})$ are of class $W^{3,\left( \frac{n}{2},1\right)}$ with the estimate
			\begin{align*}
				\begin{cases} 
					\displaystyle \left\| \vp_{\alpha,\beta} \right\|_{L^{\infty}\left( \vp_{\beta}^{-1}(U_{\alpha}\cap U_{\beta})\right)} \leq 1, \\[3mm]
					\displaystyle \left\| \g \vp_{\alpha,\beta} \right\|_{L^{\infty}\left( \vp_{\beta}^{-1}(U_{\alpha}\cap U_{\beta})\right)} \leq 1 + C_2\, \eps_0,\\[3mm]
					\displaystyle  \left\| \dr_i \vp_{\alpha,\beta}\cdot \dr_j \vp_{\alpha,\beta} - \delta_{ij} \right\|_{L^{\infty}\left( \vp_{\beta}^{-1}(U_{\alpha}\cap U_{\beta})\right)} \leq C_2\, \eps_0,\\[3mm]
					\displaystyle  \left\| \g^2 \vp_{\alpha,\beta}\right\|_{L^{(n,1)}\left( \vp_{\beta}^{-1}(U_{\alpha}\cap U_{\beta}),\geu\right)} \leq C_2\, \eps_0,\\[3mm]
					\displaystyle  \left\| \g^3 \vp_{\alpha,\beta}\right\|_{L^{\left( \frac{n}{2},1\right)}\left( \vp_{\beta}^{-1}(U_{\alpha}\cap U_{\beta}),\geu\right)} \leq C_2\, \eps_0.
				\end{cases}
			\end{align*}
			
			\item\label{it:Regularity2} The induced differential structure is $C^1$-diffeomorphic to the initial $C^{\infty}$ structure of $\Sigma$.
			
			\item\label{it:Size2} The cardinal of $A$ is bounded from above as long as the following number is bounded from above:
			\begin{align*}
				\Ec_{\vartheta}(\vPhi)\coloneqq \vol_{g_{\vPhi}}(\Sigma_{\vartheta})\, \inf\left\{\frac{1}{r^n} : \forall \vx\in \R^d,\quad \Er\Big( \vPhi;\vPhi^{-1}\big(\B^d(\vx,r)\big)\cap \Sigma_{\vartheta} \Big) \leq \eps_0 \right\}.
			\end{align*}
		\end{enumerate}
	\end{theorem}
	
	\Cref{th:Atlas} is a direct corollary of \Cref{th:Atlas2} and \eqref{eq:diameter_energy_volume}.	We first prove \Cref{th:Atlas2} for immersions of class $C^{\infty}$ and then for weak immersions, by an approximation method. 
	
	\subsection{Construction of the atlas for smooth immersions of closed manifolds}\label{sec:Atlas_smooth_case}
	
	In this section, we prove \Cref{th:Atlas} for a given \underline{$C^{\infty}$-immersion} $\vPhi\colon \Sigma\to \R^d$ satisfying the assumptions of \Cref{th:Atlas}. \\
	
	By \eqref{eq:diameter_energy_volume}, there exists a constant $C=C(n,d)>0$ such that 
	\begin{align*}
		\vol_{g_{\vPhi}}(\Sigma) \leq C\, E\, D^n.
	\end{align*}
	Hence, the quantity $\eps_0$ defined in \Cref{th:Construction_chart} is chosen to depend only on $D$ and $E$. Given $x\in \Sigma$, we define $r_x>0$ to be the largest radius such that
	\begin{align*}
		\Er\Big(\vPhi;\vPhi^{-1}\big( \B^d(\vPhi(x),2r_x)\big) \Big) \leq \eps_0.
	\end{align*}
	Thanks to \Cref{th:Construction_chart}, there exists $\rho_x\in[r_x,2r_x]$ and a point $\vep_x\in \B^d\left( \vPhi(x), \frac{r_x}{100}\right)$ such that the connected component containing $x$ is diffeomorphic to the unit ball $\B^n$. We consider the following covering of $M\coloneqq \vPhi(\Sigma)$:
	\begin{align*}
		M\subset \bigcup_{x\in \Sigma} \B^d\left( \vep_x,\rho_x \right).
	\end{align*}
	We extract a Besicovitch covering (that we can assume to be finite since $M$ is compact):
	\begin{align}\label{eq:Besicovitch}
		M\subset \bigcup_{\alpha\in A} \B^d(\vep_{x_{\alpha}},\rho_{x_{\alpha}}).
	\end{align}
	Let $x\in \Sigma$. There exist $N=N(d)\in \N$ and $I(x)\in\{1,\ldots,N\}$ balls $\B^d\left(\vep_{x_{\alpha_1}},\rho_{x_{\alpha_1}} \right),\ldots, \B^d \left( \vep_{x_{\alpha_{I(x)}}}, \rho_{x_{\alpha_{I(x)}}} \right)$ with $\alpha_1,\ldots,\alpha_{I(x)}\in A$ such that 
	\begin{align*}
		\vPhi(x) \in \bigcap_{i=1}^{I(x)} \B^d\left( \vep_{x_{\alpha_i}}, \rho_{x_{\alpha_i}} \right).
	\end{align*}
	For each $i\in\{1,\ldots,I(x)\}$, we consider $U_{x,i}\subset \Sigma$ the connected component of $\vPhi^{-1}\left( \B^d(\vep_{x_{\alpha_i}},\rho_{x_{\alpha_i}})\right)$ containing $x$. Thanks to \Cref{th:Construction_chart}, we can find coordinates $\vp_{x,\alpha_i}\colon \Or_{x,i}\to U_{x,i}$ satisfying \Cref{it:Harmonic,it:Boundary,it:Regularity}. We now check \Cref{it:Size}.
	
	\begin{claim}
		It holds $\sharp A \leq C(n,d)\, \Ec(\vPhi)$.
	\end{claim}
	\begin{proof}
		Since \eqref{eq:Besicovitch} is a locally finite covering, it holds
		\begin{align*}
			\vol_{g_{\vPhi}}(\Sigma) \geq C(d)\, \sum_{\alpha\in A} \vol_{g_{\vPhi}}\left( \vPhi^{-1}\big(\B^d(x_{\alpha},\rho_{x_{\alpha}})\big) \right).
		\end{align*}
		By \Cref{pr:Lower_Extrinsic}, we obtain 
		\begin{align*}
			\vol_{g_{\vPhi}}(\Sigma) \geq C(n,d)\, \left( \inf_{\alpha\in A} \rho_{x_{\alpha}}^n \right)\, \sharp A.
		\end{align*}
		By \eqref{eq:diameter_energy_volume}, it holds
		\begin{align*}
			\sharp A \leq C(n,d)\left( \sup_{\alpha\in A} \rho_{x_{\alpha}}^{-n} \right) E\, D^n.
		\end{align*}
	\end{proof}
	
	It remains to check \Cref{it:Transition}. Consider $\alpha,\beta\in A$ such that $U_{\alpha}\cap U_{\beta}\neq \emptyset$. To simplify the notations, we denote
	\begin{align*}
		\begin{cases} 
			V_{\alpha}\coloneqq  \vp_{\alpha}^{-1}(U_{\alpha}\cap U_{\beta}),\\[2mm]
			h_{\alpha} \coloneqq (\vp_{\alpha})^* g_{\vPhi}.
		\end{cases} \qquad \text{ and } \qquad 
		\begin{cases}
			V_{\beta} \coloneqq \vp_{\beta}^{-1}(U_{\alpha}\cap U_{\beta}), \\[2mm]
			h_{\beta} \coloneqq (\vp_{\beta})^* g_{\vPhi}.
		\end{cases}
	\end{align*}
	By \eqref{eq:bound_coordinates}, we have 
	\begin{align}\label{eq:hb}
		\left\| \g^2 h_{\beta} \right\|_{L^{\left(\frac{n}{2},1\right)}(\B^n(0,r_{\beta}),\geu)} + \left\| \g h_{\beta} \right\|_{L^{(n,1)}(\B^n(0,r_{\beta}),\geu)} + \|h_{\beta} - g_{eucl} \|_{L^{\infty}(\B^n(0,r_{\beta}),\geu)} \leq C\, \Er(\vPhi,U_{\beta})^{\frac{2}{n}}.
	\end{align}
	We also have
	\begin{align}\label{eq:ha}
		\left\| \g^2 h_{\alpha} \right\|_{L^{\left(\frac{n}{2},1\right)}(\B^n(0,r_{\alpha}),\geu)} + \left\| \g h_{\alpha} \right\|_{L^{(n,1)}(\B^n(0,r_{\alpha}),\geu)} + \|h_{\alpha} - \geu \|_{L^{\infty}(\B^n(0,r_{\alpha}),\geu)} \leq C\, \Er(\vPhi,U_{\alpha})^{\frac{2}{n}}.
	\end{align}
	By definition, the map $\vp_{\alpha,\beta} \coloneqq \vp_{\alpha}^{-1}\circ\vp_{\beta}$ is an isometry from $\left(V_{\beta}, h_{\beta} \right)$ into $\left( V_{\alpha}, h_{\alpha} \right)$. Since $V_{\alpha}\subset \Or_{\alpha}$ and $V_{\beta}\subset \Or_{\beta}$, we obtain from \eqref{eq:bound_domain}
	\begin{align*}
		\begin{cases} 
			\|\vp_{\alpha,\beta}\|_{L^{\infty}(V_{\beta})} \leq C(n,d,D)\, r_{\alpha}\leq C(n,d,D),\\[2mm]
			\|\vp_{\beta,\alpha}\|_{L^{\infty}(V_{\alpha})} \leq C(n,d,D)\, r_{\beta}\leq C(n,d,D).\\[2mm]
		\end{cases}
	\end{align*}
	By \eqref{eq:ha}, we have for any $1\leq i,j\leq n$
	\begin{align}\label{eq:Lips_0}
		\left\| h_{\beta}\left( \dr_i \vp_{\alpha,\beta}, \dr_j \vp_{\alpha,\beta}\right) - \delta_{ij} \right\|_{L^{\infty}(V_{\beta})} \leq C\, \Er(\vPhi,U_{\alpha})^{\frac{2}{n}}.
	\end{align}
	By \eqref{eq:hb}, we obtain
	\begin{align}\label{eq:Lips_transition}
		\left\| \g \vp_{\alpha,\beta} \right\|_{L^{\infty}(V_{\beta},\geu) } \leq C\, \left\| \g \vp_{\alpha,\beta} \right\|_{L^{\infty}(V_{\beta},h_{\beta}) } \leq C\left( 1+ \Er(\vPhi,U_{\alpha})^{\frac{2}{n}} + \Er(\vPhi,U_{\beta})^{\frac{2}{n}} \right).
	\end{align}
	From \eqref{eq:Lips_0}, we deduce that 
	\begin{align}\label{eq:Isom_transition}
		\forall i,j\in\{1,\ldots,n\},\qquad \left\| \delta_{\lambda\mu}\, (\dr_i \vp_{\alpha,\beta}^{\lambda})\, (\dr_j \vp_{\alpha,\beta}^{\mu}) - \delta_{ij} \right\|_{L^{\infty}(V_{\beta})}\leq C\, \left( \Er(\vPhi,U_{\alpha})^{\frac{2}{n}} + \Er(\vPhi,U_{\beta})^{\frac{2}{n}} \right).
	\end{align}
	We obtain that if $\Er(\vPhi,U_{\alpha}) + \Er(\vPhi,U_{\beta}) \leq \eps$ is small enough (depending only on $n$ and $d$), then the matrix $\left( \dr_i\vp_{\alpha,\beta}\cdot \dr_j \vp_{\alpha,\beta} \right)_{1\leq i,j\leq n}$ is uniformly close to the identity. In particular, it is invertible with inverse uniformly close to the identity as well. 
    Thanks to \cite[Equation (2.13)]{taylor2006}, we have 
    \begin{align}\label{eq:hessian_vp}
        \dr^2_{ij} (\vp_{\alpha,\beta})^k = \big( {^{h_{\beta}} \Gamma}_{ij}^p \big)\, \dr_p (\vp_{\alpha,\beta})^k - \big( {^{h_{\alpha}} \Gamma^k_{pq}} \big)\, \left( \dr_i (\vp_{\alpha,\beta})^p \right)\, \left( \dr_j (\vp_{\alpha,\beta})^q \right).
    \end{align}
	Combining \eqref{eq:Lips_transition}-\eqref{eq:Isom_transition} and \eqref{eq:ha}-\eqref{eq:hb}, we obtain that 
	\begin{align*}
		\forall i,j\in \{1,\ldots,n\},\qquad \|\dr^2_{ij} \vp_{\alpha,\beta} \|_{L^{(n,1)}(V_{\beta},\geu)}\leq C\, \eps.
	\end{align*}
	Differentiating \eqref{eq:hessian_vp}, we obtain 
	\begin{align*}
		\|\g^3 \vp_{\alpha,\beta} \|_{L^{\left( \frac{n}{2},1\right)}(V_{\beta},\geu)} \leq C\, \eps.
	\end{align*}
	This proves \Cref{it:Transition}.

	\subsection{Construction of the atlas for smooth immersions of open manifolds}\label{sec:Atlas_smooth_case2}
	
	In this section, we prove \Cref{th:Atlas2} for any given $C^{\infty}$ immersion $\vPhi\colon \Sigma \to \R^d$ satisfying the assumptions of \Cref{th:Atlas2}. \\
	
	We consider the quantity $\eps_0$ defined in \Cref{th:Construction_chart}. Given $x\in \Sigma_{\vartheta}$, we define $r_x\in(0,\vartheta)$ to be the largest radius such that
	\begin{align*}
		\Er\Big(\vPhi;\vPhi^{-1}\big( \B^d(\vPhi(x),2r_x)\big) \Big) \leq \eps_0.
	\end{align*}
	Thanks to \Cref{th:Construction_chart}, there exists $\rho_x\in[r_x,2r_x]$ and a point $\vep_x\in \B^d\big( \vPhi(x), \frac{r_x}{100}\big)$ such that the connected component containing $x$ is diffeomorphic to the unit ball $\B^n$. We define $U_x\subset \Sigma_{\vartheta}$ to be this connected component. We consider the following covering of $M_{\vartheta}\coloneqq \vPhi(\Sigma_{\vartheta})$:
	\begin{align*}
		M_{\vartheta}\subset \bigcup_{x\in \Sigma} \B^d(\vep_x,\rho_x).
	\end{align*}
	We extract a Besicovitch covering (that we can assume to be finite since $\overline{ M_{\vartheta}}$ is compact):
	\begin{align}\label{eq:Besicovitch2}
		M_{\vartheta}\subset \bigcup_{\alpha\in A} \B^d(\vep_{x_{\alpha}},\rho_{x_{\alpha}}).
	\end{align}
	Thanks to \Cref{th:Construction_chart}, we can find coordinates $\vp_{\alpha}\colon \Or_{\alpha}\to U_{x_{\alpha}}$ such that $\vPhi\circ\vp_{\alpha} (\dr\Or_{\alpha})$ is a graph satisfying \Cref{it:Harmonic2,it:Boundary2}. Since $\vPhi$ is of class $C^{\infty}$, the metric $g_{\vPhi}$ as well and the coordinates $\vp_{\alpha}$ also. Hence, \Cref{it:Regularity2} is verified. We now check \Cref{it:Size2}.
	
	\begin{claim}
		It holds $\sharp A \leq C(n,d)\, \Ec_{\vartheta}(\vPhi)$.
	\end{claim}
	\begin{proof}
		Since \eqref{eq:Besicovitch2} is a locally finite covering, it holds
		\begin{align*}
			\vol_{g_{\vPhi}}(\Sigma_{\vartheta}) \geq C(d)\, \sum_{\alpha\in A} \vol_{g_{\vPhi}}\left( \vPhi^{-1}\big(\B^d(x_{\alpha},\rho_{x_{\alpha}})\big) \right).
		\end{align*}
		By \Cref{pr:Lower_Extrinsic}, we obtain 
		\begin{align*}
			\vol_{g_{\vPhi}}(\Sigma_{\vartheta}) \geq C(n,d)\, \left( \inf_{\alpha\in A} \rho_{x_{\alpha}}^n \right)\, \sharp A.
		\end{align*}
	\end{proof}
	
	It remains to check \Cref{it:Transition2}. The proof is exactly the same as in \Cref{sec:Atlas_smooth_case} since we avoided the boundary.
	
	\subsection{Construction of the atlas for weak immersions}\label{sec:Altas_weak_imm}
	
	The proofs of \Cref{th:Atlas} and \Cref{th:Atlas2} for weak immersions are very similar and consist only in passing to the limit from an approximation by smooth immersions to a given weak immersion. Hence, we choose to write it only for closed manifolds.\\
	
	\subsubsection{Setting.}\label{sec:Setting}
	In this section, we fix a weak immersion $\vPhi\in \I_{\frac{n}{2}-1,2}\big(\Sigma^n;\R^d\big)$ satisfying the assumptions of \Cref{th:Atlas}. From \cite{Lan2025} we know the existence of $(\vPhi_k)_{k\in\N}\subset \Imm\big(\Sigma;\R^d \big)$  converging to $\vPhi$ in the strong $W^{\frac{n}{2}+1,2}(\Sigma,g_{\vPhi})$-topology\footnote{Indeed, if $(V_{\alpha},\psi_{\alpha})_{\alpha\in A}$ is a locally finite atlas of $\Sigma$ and $(\eta_{\alpha})_{\alpha\in A}$ is an associated partition of unity, we have $\vPhi = \sum_{\alpha} \eta_{\alpha}\vPhi$. We extend each $\eta_{\alpha}\vPhi$ by zero to $\R^n$. We consider $(\chi_{\eps})_{\eps>0}$ a mollifier. We can take $\vPhi_k = \sum_{\alpha}  \chi_{1/k}*(\eta_{\alpha}\vPhi)$ thanks to \cite{Lan2025}
		Theorem IV.21 page 62.}.
	Hence the metrics $g_{\vPhi_k}$ converge strongly to $g_{\vPhi}$ in $W^{\frac{n}{2},2}(\Sigma)$ and remain bounded in $L^{\infty}$, i.e. there exist a $C^{\infty}$ metric $h$ on $\Sigma$ and a constant $\Lambda>0$ such that 
	\begin{align}\label{eq:g_bounded}
		\begin{cases} 
			\displaystyle \forall k\in\N,\qquad \Lambda^{-1}\, h \leq g_{\vPhi_k} \leq \Lambda\, h, \\[3mm]
			\displaystyle g_{\vPhi_k}\xrightarrow[k\to +\infty]{} g_{\vPhi} \qquad \text{strongly in }W^{\frac{n}{2},2}(\Sigma,h).
		\end{cases} 
	\end{align}
	Therefore, the convergence in $W^{\frac{n}{2}+1,2}(\Sigma,g_{\vPhi})$ is equivalent to the convergence in $W^{\frac{n}{2}+1,2}(\Sigma,h)$. Since $h$ is a $C^{\infty}$-metric, we will often refer to this background metric instead of $g_{\vPhi}$. Because we have a strong convergence, we can assume that (up to a subsequence), the following estimates hold for all $k\geq 0$
	\begin{align}\label{eq:Unif_Int_Approx}
		\begin{cases}
			\displaystyle \Er\left(\vPhi_k;\Sigma \right) \leq 2\, E,\\[3mm]
			\displaystyle \sup\left\{\frac{1}{r} : \forall \vx\in \R^d,\quad \Er\Big( \vPhi_k;\vPhi_k^{-1}\big(\B^d(\vx,r)\big) \Big) \leq \eps_0 \right\} \leq 2\ \sup\left\{\frac{1}{r} : \forall \vx\in \R^d,\quad \Er\Big( \vPhi;\vPhi^{-1}\big(\B^d(\vx,r)\big) \Big) \leq \eps_0 \right\},\\[5mm]
			\displaystyle \diam\left( \vPhi_k(\Sigma) \right) \leq 2\, D.
		\end{cases}
	\end{align}

	For each $k\in \N$, \Cref{sec:Atlas_smooth_case} provides an atlas $(U_{\alpha}^k,\vp_{\alpha}^k)_{k\in\N, \alpha\in A_k}$ of $\Sigma$ such that \Cref{it:Boundary,it:Harmonic,it:Regularity,it:Size,it:Transition} of \Cref{th:Atlas} are satisfied for $\vPhi_k$. Thanks to \Cref{it:Size} and \eqref{eq:Unif_Int_Approx}, we can assume that the sets $A_k$ have constant cardinal and thus, are all equal: $A_k = A$. For each $\alpha,\beta\in A$ and $k\in \N$, we obtain the following properties:
	\begin{enumerate}
		\item\label{it:Charts} Each chart $\vp_{\alpha}^k\colon \left( \Or_{\alpha}^k,\geu\right)\to U^k_{\alpha}\subset ( \Sigma, h)$ is uniformly bi-Lipschitz and $g_{\vPhi_k\circ \vp_{\alpha}^k}$ is uniformly bounded in $W^{2,\left( \frac{n}{2},1\right)}$ thanks to \eqref{eq:g_bounded} and \Cref{it:Transition} of \Cref{th:Atlas}. Moreover, $r_{\alpha}^k$ is bounded from below thanks to \eqref{eq:Unif_Int_Approx}.
		
		\item\label{it:Graphs} Each map $\vPhi_k \circ \vp_{\alpha}^k \colon \dr \Or_{\alpha}^k \to \s^{d-1}(\vq_{\alpha}^{\ k},r_{\alpha}^k)$ is described as in \eqref{eq:pseudo_graph} through a map $\vphi_{\alpha}^{\ k} \in W^{1,\infty}\cap W^{2,n}(\Sc_{\alpha}^k;\R^d)$ where $\Sc_{\alpha}^k \coloneqq \Pc_{\alpha}^k\cap \B^d(\vq_{\alpha}^{\ k}, r_{\alpha}^k)$ for some affine $n$-dimensional plane $\Pc_{\alpha}^k$ such that 
		\begin{enumerate}
			\item\label{it:dist_plane} $\dist(\vq_{\alpha}^{\ k},\pr_{\alpha}^k) \leq \delta\, r_{\alpha}^k$ for some fixed $\delta=\delta(n,d)\in(0,1)$,
			\item\label{it:est_graph} we have the estimate 
			\begin{align*}
				\eps_0^{-1}\left[ (r_{\alpha}^k)^{-1} \|\vphi_{\alpha}^{\ k}\|_{L^{\infty}(\Sc_{\alpha}^k)} + \|\g \vphi_{\alpha}^{\ k}\|_{L^{\infty}(\Sc_{\alpha}^k)} \right] + (r_{\alpha}^k)^{\frac{1}{n}} \|\g^2 \vphi_{\alpha}^{\ k}\|_{L^n(\Sc_{\alpha}^k)} \leq C.
			\end{align*}
		\end{enumerate}
		
		\item The following transition function is uniformly bounded in $W^{3,\left( \frac{n}{2},1\right)}_{\geu}$:
		\begin{align*}
			\vp^k_{\beta,\alpha}\coloneqq (\vp_{\beta}^k)^{-1}\circ\vp_{\alpha}^k\colon (\vp_{\alpha}^k)^{-1} \left( U^k_{\alpha}\cap U^k_{\beta} \right)\to (\vp_{\beta}^k)^{-1}\left( U^k_{\alpha}\cap U^k_{\beta} \right).
		\end{align*}
	\end{enumerate}
	
	Up to a subsequence, we can assume that 
	\begin{align*}
		\forall \alpha \in A,\qquad \left( \vq_{\alpha}^{\ k}, r_{\alpha}^k \right) \xrightarrow[k\to +\infty]{} \left( \vq_{\alpha},r_{\alpha}\right).
	\end{align*}
	In other words, the balls $\B^d(\vq_{\alpha}^{\ k},r_{\alpha}^k)$ converge in Hausdorff distance to some ball $\B^d(\vq_{\alpha},r_{\alpha})$.

	\subsubsection{Convergence of the sets $\Or_{\alpha}^k$.}\label{sec:Oalpha}

	Thanks to \eqref{eq:g_bounded} and \Cref{it:Coordinates} of \Cref{th:Construction_chart}, we obtain that the bi-Lipschitz homeomorphisms $(f_{\alpha}^k)^{-1}\colon \B^n(0,r_{\alpha}^k)\to \Or_{\alpha}^k$ verify the uniform estimate
	\begin{align*}
		\forall k\in\N,\qquad \left\| (f_{\alpha}^k)^{-1} \right\|_{W^{2,n}(\dr \B^n(0,r_{\alpha}^k))} \leq C_0.
	\end{align*}
	Up to a subsequence, the sequence $(f_{\alpha}^k)^{-1}|_{\dr \B^n(0,r_{\alpha}^k)}$ converges weakly in $W^{2,n}$ and strongly in $C^{1,\frac{1}{2n}}$ to a limiting bi-Lipschitz map $F_{\alpha}\colon \dr \B^n(0,r_{\alpha})\to \R^n$. Moreover, it holds 
	\begin{align*}
		|\det \g F_{\alpha}| = \lim_{k\to +\infty} \left| \det \g (f_{\alpha}^k)^{-1} \right| = \lim_{k\to +\infty} \frac{1}{\left| \det \g f_{\alpha}^k \right| }\geq \frac{1}{C_0^n}.
	\end{align*}
	Therefore, $F_{\alpha}$ is a diffeomorphism of class $W^{2,n}(\dr \B^n(0,r_{\alpha}))$. In particular, the open sets $\Or_{\alpha}^k$ converge in the Hausdorff topology to some open set $\Or_{\alpha}\subset \R^n$ with boundary diffeomorphic to $\s^{n-1}$, with a diffeomorphism of class $W^{2,n}$ verifying the estimates \eqref{eq:bound_domain}. We summarize
	\begin{align}\label{eq:conv_diffeo}
		\begin{aligned} 
			& (f_{\alpha}^k)^{-1} \big( (r_{\alpha}^k)^{-1}\cdot \big) \xrightarrow[k\to +\infty]{} (f_{\alpha})^{-1}\big( r^{-1}\cdot \big)  \\[2mm]
			& \qquad \text{weakly in }W^{2,n}(\s^{n-1}) \text{ and strongly in }C^{1,\frac{1}{2n}}(\s^{n-1}).
		\end{aligned} 
	\end{align}
	Consequently, the maps $\tilde{f}_{\alpha}^k\coloneq f_{\alpha}^k\circ\left( \frac{r_{\alpha}^k }{ r_{\alpha} } \, f_{\alpha}^{-1} \right)$ verify 
	\begin{align}\label{eq:conv_diffeo2}
		\begin{aligned} 
			& \tilde{f}_{\alpha}^k\xrightarrow[k\to +\infty]{} \Id_{\dr\Or_{\alpha}}  \\[2mm]
			& \qquad \text{weakly in }W^{2,n}(\dr\Or_{\alpha}) \text{ and strongly in }C^{1,\frac{1}{2n}}(\dr\Or_{\alpha}).
		\end{aligned} 
	\end{align}
	
	\subsubsection{Weak convergence of the metrics $g_{\vPhi_k\circ\vp_{\alpha}^k}$.} By \eqref{eq:bound_coordinates}, the coefficients $\big( g_{\vPhi_k\circ\vp_{\alpha}^k} \big)_{ij}$ are uniformly bounded in $W^{2,\left(\frac{n}{2},1\right)}(\Or_{\alpha}^k)$. Thus, up to a subsequence, we have a metric $g_{\alpha}\in W^{2,\left(\frac{n}{2},1\right)}(\Or_{\alpha})$ such that 
	\begin{align}\label{eq:galpha}
		\begin{aligned} 
			& \big( g_{\vPhi_k\circ\vp_{\alpha}^k}\big)_{ij} \xrightarrow[k\to +\infty]{} \big( g_{\alpha} \big)_{ij} \\[2mm]
			&\qquad \text{ weakly in }W_{\loc}^{2,\left(\frac{n}{2},1\right)}(\Or_{\alpha},\geu) \text{ and strongly in }L^p_{\loc}(\Or_{\alpha}) \text{ for any }1\leq p<+\infty.
		\end{aligned} 
	\end{align}
	
	\subsubsection{Weak convergence of the charts $\vp_{\alpha}^k$.} 
	
	By \Cref{it:Charts} and \eqref{eq:g_bounded}, we obtain that for each $\alpha\in A$, the sequence $\vp_{\alpha}^k\colon \Or_{\alpha}^k\to (\Sigma,h)$ will converge in the $C^{0,\frac{1}{2}}$-topology to a bi-Lipschitz map $\vp_{\alpha}\colon \Or_{\alpha}\to (\Sigma,h)$. Since $\vp_{\alpha}$ is bi-Lipschitz, its inverse $\vp_{\alpha}^{-1}\colon U_{\alpha}\coloneqq \vp_{\alpha}(\Or_{\alpha}) \to \Or_{\alpha}$ is continuous, so that $U_{\alpha} = \left[\vp_{\alpha}^{-1}\right]^{-1}(\Or_{\alpha})$ is an open set of $\Sigma$. Moreover, we have $(\vp_{\alpha}^k)^*g_{\vPhi_k} = g_{\vPhi_k\circ\vp_{\alpha}}$. Since $g_{\vPhi_k}$ converges strongly to $g_{\vPhi}$ in the $W^{\frac{n}{2},2}(\Sigma,h)$ topology, we obtain 
	\begin{align}\label{eq:Unif_vp}
		\|d\vp_{\alpha}^k\|_{L^{\infty}(\Or_{\alpha}^k)} + \|\vp_{\alpha}^k\|_{W^{3,\frac{n}{2}}(\Or_{\alpha}^k)} + \|d(\vp_{\alpha}^k)^{-1}\|_{L^{\infty}(U_{\alpha}^k)} + \|(\vp_{\alpha}^k)^{-1}\|_{W^{3,\frac{n}{2}}(U_{\alpha}^k)} \leq C(n,\vPhi). 
	\end{align}
	
	\subsubsection{Weak convergence of the immersions.} Since $(\vPhi_k)_{k\in\N}$ converges strongly to $\vPhi$ in the $W^{\frac{n}{2}+1,2}(\Sigma,h)$-topology and using that $\s^{d-1}(\vq_{\alpha}^{\ k},r_{\alpha}^k)\cap \vPhi_k(\Sigma)$ are good slices (i.e. satisfy \Cref{it:est_graph,it:dist_plane}), the following convergences hold. In $\Or_{\alpha}$, we have
	\begin{align}\label{eq:lim_Phialpha1}
		\begin{aligned} 
			&  \vPhi_k\circ\vp_{\alpha}^k \xrightarrow[k\to +\infty]{} \vPhi\circ \vp_{\alpha} \colon \Or_{\alpha}\to \B^d(\vq_{\alpha},r_{\alpha}) \\[2mm]
			& \qquad \text{weakly in }W_{\loc}^{3,\frac{n}{2}}(\Or_{\alpha}) \text{ and strongly in }C^{0,\frac{1}{2}}_{\loc}(\Or_{\alpha}).
		\end{aligned} 
	\end{align}
	On $\dr \Or_{\alpha}$, we have uniformly bi-Lipschitz diffeomorphisms $\tilde{f}_{\alpha}^k\coloneq f_{\alpha}^k\circ \left( \frac{r_{\alpha}^k}{ r_{\alpha} }\, f_{\alpha}^{-1} \right)\colon \Or_{\alpha}\to \Or_{\alpha}^k$ given by \eqref{eq:conv_diffeo2} converging to the identity map in the weak $W^{2,n}(\s^{n-1})$ and the strong $C^{1,\frac{1}{2n}}$ topologies. Hence it holds
	\begin{align}\label{eq:lim_Phialpha2}
		\begin{aligned} 
			& \vPhi_k\circ\vp_{\alpha}^k\circ \tilde{f}_{\alpha}^k \xrightarrow[k\to +\infty]{} \vPhi\circ\vp_{\alpha} \colon \s^{n-1}(0,r_{\alpha})\to \s^{d-1}(\vq_{\alpha},r_{\alpha})\\[2mm]
			&\qquad   \text{weakly in }W^{2,n}(\s^{n-1}(0,r_{\alpha})) \text{ and strongly in }C^{0,\frac{1}{2n}}(\s^{n-1}(0,r_{\alpha})).
		\end{aligned} 
	\end{align} 
	Thus, we obtain from \Cref{it:Graphs} that the intersections $\vPhi(\Sigma)\cap \s^{d-1}(\vq_{\alpha},r_{\alpha})$ are given by a union of sets of the form \eqref{eq:pseudo_graph} satisfying the same estimates.\\
	
	\subsubsection{Weak convergence of the transition functions.}
	Consider now $\alpha,\beta\in A$ such that $U_{\alpha}\cap U_{\beta}\neq \varnothing$. Since $U_{\alpha}\cap U_{\beta}$ is non-empty subset of $\vPhi^{-1}\left( \B^d(\vq_{\alpha}, r_{\alpha}) \cap \B^d(\vq_{\beta}, r_{\beta})  \right)$ by \eqref{eq:lim_Phialpha1}, we have 
	\begin{align*}
		\lim_{k\to +\infty} \frac{|\vq_{\alpha}^{\ k} - \vq_{\beta}^{\ k}|}{r_{\alpha}^k + r_{\beta}^k}< 1.
	\end{align*}
	Moreover, we have the convergence
	\begin{align*}
		(\vp_{\alpha}^k)^{-1}\xrightarrow[k\to +\infty]{} \vp_{\alpha}^{-1} \qquad \text{ strongly in }C^{0,\frac{1}{2}}_{\loc}(\Or_{\alpha}).
	\end{align*}
	Hence the open sets $U_{\alpha}^k$ converge in Hausdorff distance to $U_{\alpha}$. Therefore, the open sets $\left( U_{\alpha}^k\cap U_{\beta}^k \right)_{k\in\N}$ converge to the open set $U_{\alpha}\cap U_{\beta}$ in the Hausdorff topology\footnote{Here we use that $U_{\alpha}\cap U_{\beta}$ is a non-empty open set. For $\delta>0$, we denote $B_{\delta}(U)$ the $\delta$-neighbourhood of $U$. If $x\in U_{\alpha}\cap U_{\beta}$, then there exists a radius $s>0$ such that $\B^n(x,s)\Subset U_{\alpha}\cap U_{\beta}$. Hence, we have $\B^n(x,s/2)\subset B_{\delta}(U_{\alpha}^k\cap U_{\beta}^k)$ for $k$ large enough depending on $x$, $s$ and $\delta$. Hence, it holds 
		\begin{align*}
			U_{\alpha}\cap U_{\beta} \subset \liminf_{k\to +\infty} B_{\delta}(U_{\alpha}^k\cap U_{\beta}^k) = \bigcup_{k\geq 0} \bigcap_{\ell \geq k} B_{\delta}(U_{\alpha}^{\ell}\cap U_{\beta}^{\ell})\subset B_{\delta}\left( \liminf_{k\to +\infty} U_{\alpha}^k \cap U_{\beta}^k \right). 
		\end{align*}
		Conversely, with $\limsup_{k\to +\infty} U_{\alpha}^k\cap U_{\beta}^k = \bigcap_{k\geq 0} \bigcup_{\ell\geq k} U_{\alpha}^{\ell}\cap U_{\beta}^{\ell}$, it holds
		\begin{align*}
			\limsup_{k\to +\infty} U_{\alpha}^k\cap U_{\beta}^k \subset  B_{\delta}(U_{\alpha}\cap U_{\beta}).
		\end{align*}
		Hence, we obtain that the Hausdorff distance between $U_{\alpha}^k\cap U_{\beta}^k$ and $U_{\alpha}\cap U_{\beta}$ converges to zero.
	}. Since the $(\vp_{\alpha}^k)^{-1}$ are uniformly Lipschitz, we obtain 
	\begin{align}
		& \dist_H\left( (\vp_{\alpha}^k)^{-1}\left( U_{\alpha}^k\cap U_{\beta}^k\right),  \vp_{\alpha}^{-1}\left( U_{\alpha}\cap U_{\beta}\right) \right) \nonumber  \\[3mm]
		& \leq \dist_H\left( (\vp_{\alpha}^k)^{-1}\left( U_{\alpha}^k\cap U_{\beta}^k\right),  (\vp_{\alpha}^k)^{-1}\left( U_{\alpha}\cap U_{\beta}\right) \right) + \dist_H\left( (\vp_{\alpha}^k)^{-1} \left( U_{\alpha}\cap U_{\beta}\right),  \vp_{\alpha}^{-1}\left( U_{\alpha}\cap U_{\beta}\right) \right) \nonumber \\[3mm]
		& \leq C\ \dist_H\left(  U_{\alpha}^k\cap U_{\beta}^k,   U_{\alpha}\cap U_{\beta} \right) + \dist_H\left( (\vp_{\alpha}^k)^{-1} \left( U_{\alpha}\cap U_{\beta}\right),  \vp_{\alpha}^{-1}\left( U_{\alpha}\cap U_{\beta}\right) \right) \nonumber \\[1mm]
		& \xrightarrow[k\to +\infty]{} 0. \label{eq:Hausdorff}
	\end{align}
	Hence, for any open set $V\Subset \vp_{\alpha}^{-1}\left( U_{\alpha}\cap U_{\beta}\right)$, there exists $k_0=k_0(V,\alpha,\beta)\in \N$ such that for any $k\geq k_0$, we have $V\Subset (\vp_{\alpha}^k)^{-1}\left( U_{\alpha}^k\cap U_{\beta}^k\right)$. Thanks to \Cref{it:Transition}, the transition charts $\vp_{\alpha,\beta}\coloneqq \vp_{\alpha}^{-1}\circ\vp_{\beta}$ and $\vp_{\alpha,\beta}^k\coloneqq (\vp_{\alpha}^k)^{-1}\circ\vp_{\beta}$ satisfy
	\begin{align}\label{eq:bound_Transition}
		\left\| \vp_{\alpha,\beta} \right\|_{W^{3,\left(\frac{n}{2},1\right)}\left( V,\geu\right) } \leq \liminf_{k\to +\infty} \left\| \vp^k_{\alpha,\beta} \right\|_{W^{3,\left(\frac{n}{2},1\right)}\left( V,\geu\right) } \leq C\, \eps_0.
	\end{align}
	Given $\ell\in \N$, we can choose $V=V(\ell)$ such that $\Hr^n\left(  \vp_{\alpha}^{-1}\left( U_{\alpha}\cap U_{\beta}\right) \setminus V \right)\leq \ell^{-1}$. Letting $\ell\to +\infty$, we obtain 
	\begin{align*}
		\left\| \vp_{\alpha,\beta} \right\|_{W^{3,\left(\frac{n}{2},1\right)}\left(  \vp_{\alpha}^{-1}\left( U_{\alpha}\cap U_{\beta}\right),\geu\right) }  \leq C\, \eps_0.
	\end{align*}
	Moreover, the transition charts are uniformly bi-Lipschitz. Up to a subsequence, we obtain the following convergence:
	\begin{align}\label{eq:conv_transition}
		\vp^k_{\alpha,\beta} \xrightarrow[k\to +\infty]{} \vp_{\alpha,\beta} \qquad \text{strongly in }C^{0,\frac{1}{2}}_{\loc}(\vp_{\beta}^{-1}(U_{\alpha}\cap U_{\beta})) \text{ and weakly in }W^{3,\left(\frac{n}{2},1\right)}.
	\end{align}
	Consequently, we have an atlas satisfying \Cref{it:Boundary,it:Size,it:Transition}. It remains to check \Cref{it:Harmonic,it:Regularity}. To do so, we need to improve the convergence of the charts. 
	
	\subsubsection{Strong convergence.} Since the charts $\vp_{\alpha}^k$ provides harmonic coordinates, the map $(\vp_{\alpha}^k)^{-1}\colon U_{\alpha}^k\to \R^n$ is a harmonic map. Hence, we have 
	\begin{align}\label{eq:vpa}
		\lap_{g_{\vPhi_k}} (\vp_{\alpha}^k)^{-1} =0 \qquad \text{in }U_{\alpha}^k.
	\end{align}
	Since $g_{\vPhi_k}$ converges strongly to $g_{\vPhi}$ in $W^{\frac{n}{2},2}(\Sigma,h)$, we obtain the following convergence.
	\begin{claim}\label{cl:weak_to_strong}
		It holds 
		\begin{align*}
			(\vp_{\alpha}^k)^{-1} \xrightarrow[k\to +\infty]{} \vp_{\alpha}^{-1} \qquad \text{ strongly in } W^{3,\left(\frac{n}{2},2\right)}_{\loc}(U_{\alpha},h).
		\end{align*}
		We have $g_{\alpha} = (\vp_{\alpha})^*g_{\vPhi}$, where $g_{\alpha}$ is defined in \eqref{eq:galpha} and 
		\begin{align}\label{eq:strong_g}
			g_{\vPhi_k\circ\vp_{\alpha}^k} \xrightarrow[k\to +\infty]{} g_{\vPhi\circ\vp_{\alpha}} \qquad \text{ strongly in }W^{2,\left(\frac{n}{2},1\right)}_{\loc}( \Or_{\alpha} ,\geu) .
		\end{align}
		The transition functions $\vp^k_{\alpha,\beta} \coloneqq (\vp_{\alpha}^k)^{-1}\circ\vp_{\beta}^k$ and $\vp_{\alpha,\beta} \coloneqq \vp_{\alpha}^{-1}\circ\vp_{\beta}$ verify
		\begin{align}\label{eq:strong_Transition}
			\vp^k_{\alpha,\beta}  \xrightarrow[k\to +\infty]{} \vp_{\alpha,\beta} \qquad \text{ strongly in }W^{3,\left(\frac{n}{2},1\right)}_{\loc}(\vp_{\beta}^{-1}(U_{\alpha}\cap U_{\beta})).
		\end{align}
	\end{claim}
	\begin{proof}
		Given a fixed ball $B_h(x,s)\Subset U_{\alpha}$, we also have $B_h(x,s)\Subset U_{\alpha}^k$ for $k$ large enough. On any open set of $B_h(x,s)$ carrying fixed coordinates, we have 
		\begin{align*}
			\lap_{g_{\vPhi}} (\vp_{\alpha}^k)^{-1} = (\det g_{\vPhi})^{-\frac{1}{2}}\ \dr_i \left[ \left(g_{\vPhi}^{ij}\, \sqrt{\det g_{\vPhi}} - g_{\vPhi_k}^{ij} \, \sqrt{\det g_{\vPhi_k}}\right) \dr_j (\vp_{\alpha}^k)^{-1} \right].
		\end{align*} 
		We multiply by $\sqrt{\det g_{\vPhi}}$ and we obtain a system in divergence form
		\begin{align}\label{eq:eq_harm}
			\dr_i\left[ \sqrt{\det g_{\vPhi}}\ g_{\vPhi}^{ij} \ \dr_j (\vp_{\alpha}^k)^{-1} \right] =  \dr_i \left[ \left(g_{\vPhi}^{ij}\, \sqrt{\det g_{\vPhi}} - g_{\vPhi_k}^{ij} \, \sqrt{\det g_{\vPhi_k}}\right) \dr_j (\vp_{\alpha}^k)^{-1} \right].
		\end{align} 
		Let $V\Subset W\Subset U_{\alpha}$ be open sets with $C^{\infty}$ boundary and $\xi\in C^{\infty}_c(W;[0,+\infty))$ such that $\xi\geq 1$ in $V$. Given $1\leq a,b\leq n$, we consider the following Hodge decomposition in $L^{\frac{n}{n-2}}(W)$: there exist $\lambda_{ab}\in W^{1,\frac{n}{n-2}}(W;\R^n)$ and $\mu_{ab}\in W^{1,\frac{n}{n-2}}(W;\R^n\otimes\Lambda^1\R^n)$ such that 
		\begin{align*}
			\left| d(\dr^2_{ab} (\vp_{\alpha}^k)^{-1} - \dr^2_{ab} (\vp_{\alpha})^{-1}) \right|_{g_{\vPhi}}^{\frac{n}{2}-2} \ d(\dr^2_{ab} (\vp_{\alpha}^k)^{-1} - \dr^2_{ab} (\vp_{\alpha})^{-1}) = d\lambda_{ab} + d^{*_{\geu}}\mu_{ab} \qquad \text{ in }W.
		\end{align*}
		We have the estimates following from \eqref{eq:Unif_vp} and \eqref{eq:g_bounded}
		\begin{align}\label{eq:est_Hodge}
			\| \lambda_{ab} \|_{W^{1,\frac{n}{n-2}}(W)} + \| \mu_{ab} \|_{W^{1,\frac{n}{n-2}}(W)} \leq C\, \left\| d(\dr^2_{ab} (\vp_{\alpha}^k)^{-1} - \dr^2_{ab} (\vp_{\alpha})^{-1}) \right\|_{L^{\frac{n}{2}}(W)}^{\frac{n}{2}-1} \leq C.
		\end{align}
		We obtain by \eqref{eq:g_bounded}
		\begin{align*}
			& \int_V \left| \g\left( \dr^2_{ab} (\vp_{\alpha}^k)^{-1} - \dr^2_{ab} (\vp_{\alpha})^{-1} \right) \right|^{\frac{n}{2}}_{\geu}\ dx \\[2mm]
			& \leq C(\Lambda) \int_V \left| \g\left( \dr^2_{ab} (\vp_{\alpha}^k)^{-1} - \dr^2_{ab} (\vp_{\alpha})^{-1} \right) \right|^{\frac{n}{2}}_{g_{\vPhi}}\ d\vol_{g_{\vPhi}} \\[2mm]
			& \leq C(\Lambda) \int_W \xi\ \left| \g\left( \dr^2_{ab} (\vp_{\alpha}^k)^{-1} - \dr^2_{ab} (\vp_{\alpha})^{-1} \right) \right|^{\frac{n}{2}}_{g_{\vPhi}}\ d\vol_{g_{\vPhi}} \\[2mm]
			& \leq C(n,\Lambda) \int_W \xi\ \scal{ \dr^2_{ab}d\left(  (\vp_{\alpha}^k)^{-1} - (\vp_{\alpha})^{-1} \right)  }{ d\lambda_{ab} + d^{*_{\geu}} \mu_{ab} }_{g_{\vPhi}} \ d\vol_{g_{\vPhi}}.
		\end{align*}
		By integration by parts, we obtain 
		\begin{align}
			& \int_V \left| \g\left( \dr^2_{ab} (\vp_{\alpha}^k)^{-1} - \dr^2_{ab} (\vp_{\alpha})^{-1} \right) \right|^{\frac{n}{2}}_{\geu}\ dx \nonumber \\[2mm]
			& \leq C(n,\Lambda) \int_W \scal{\dr^2_{ab} d \left(  (\vp_{\alpha}^k)^{-1} - (\vp_{\alpha})^{-1} \right)}{F_k}_{\geu}\ dx \label{eq:T1} \\[2mm]  
			& \qquad + C(n,\Lambda) \scal{ d^{*_{g_{\vPhi}}} \dr^2_{ab}d \left(  (\vp_{\alpha}^k)^{-1} - (\vp_{\alpha})^{-1} \right) }{ \xi\, \lambda_{ab} }_{\Dr'(W),\Dr(W)}, \label{eq:T2} 
		\end{align}
		where the term $F_k$ in \eqref{eq:T1} verifies the pointwise estimate
		\begin{align*}
			|F_k|_{\geu} \leq C(n,\Lambda) \left( |d\xi|_{\geu} + |\xi|\, |\g g_{\vPhi}|_{\geu} \right)\ \left( |\lambda_{ab}| + |\mu_{ab}|_{\geu} \right).
		\end{align*}
		The term \eqref{eq:T1} converges to zero by \eqref{eq:g_bounded}, \eqref{eq:Unif_vp} and \eqref{eq:est_Hodge}. Indeed, the term $|\g g_{\vPhi}|_{\geu}$ is bounded in $L^p$ for any $p<+\infty$, the term $\left(\dr^2_{ab} d \left(  (\vp_{\alpha}^k)^{-1} - (\vp_{\alpha})^{-1} \right) \right)$ converges weakly to 0 in $L^{\frac{n}{2}}$. The term $\left( |\lambda_{ab}| + |\mu_{ab}|_{\geu} \right)$ is bounded in $W^{1,\frac{n}{n-2}}\hookrightarrow L^{\frac{n}{n-3}}$, and thus converges strongly (up to a subsequence) in $L^q$ for any $1\leq q < \frac{n}{n-3}$. Hence the term $F_k$ converges strongly in $L^{\frac{n}{n-2}}(W)$ to some limit $F\in L^{\frac{n}{n-2}}(W)$. Thus, the term \eqref{eq:T1} converges to 0, up to a subsequence.\\
		
		We now show that the term \eqref{eq:T2} also converges to $0$. By \eqref{eq:eq_harm}, it holds 
		\begin{align}
			&  \scal{ d^{*_{g_{\vPhi}}}\, \dr^2_{ab}\, d \left(  (\vp_{\alpha}^k)^{-1} - (\vp_{\alpha})^{-1} \right) }{ \xi\, \lambda_{ab} }_{\Dr'(W),\Dr(W)} \nonumber \\[2mm]
			& = \int_W \scal{\g^3 \left(  (\vp_{\alpha}^k)^{-1} - (\vp_{\alpha})^{-1} \right) }{ G_k }_{\geu} + \scal{\g^2\left(  (\vp_{\alpha}^k)^{-1} - (\vp_{\alpha})^{-1} \right) }{H_k}_{\geu} dx \label{eq:T21} \\[2mm]
			& \qquad + \scal{ \dr^2_{ab}\left( \dr_i \left[ \left(g_{\vPhi}^{ij}\, \sqrt{\det g_{\vPhi}} - g_{\vPhi_k}^{ij} \, \sqrt{\det g_{\vPhi_k}}\right) \dr_j (\vp_{\alpha}^k)^{-1} \right] \right) }{ \xi\ \lambda_{ab}}_{\Dr'(W),\Dr(W)}, \label{eq:T22}
		\end{align}
		where the terms $G_k$ and $H_k$ verify the following pointwise estimates
		\begin{align*}
			& |G_k|_{\geu} \leq C(n,\Lambda)\, \xi \,  |\lambda_{ab}|\,  |\g g_{\vPhi}|_{\geu}, \\[2mm]
			& |H_k|_{\geu} \leq C(n,\Lambda)\, \xi \, |\lambda_{ab}| \, ( 1+ |\g g_{\vPhi}|_{\geu} + |\g^2 g_{\vPhi}|_{\geu}).
		\end{align*}
		As previously, we have that, up to a subsequence, $G_k$ converges strongly in $L^{\frac{n}{n-2}}$ and $H_k$ converges strongly in $L^{\frac{n}{n-1}}$. Hence the term \eqref{eq:T21} converges to 0, up to a subsequence. The term \eqref{eq:T22} is given by 
		\begin{align*}
			& \scal{ \dr^2_{ab}\left( \dr_i \left[ \left(g_{\vPhi}^{ij}\, \sqrt{\det g_{\vPhi}} - g_{\vPhi_k}^{ij} \, \sqrt{\det g_{\vPhi_k}}\right) \dr_j (\vp_{\alpha}^k)^{-1} \right] \right) }{ \xi\ \lambda_{ab}}_{\Dr'(W),\Dr(W)} \\[2mm]
			& = -\int_W \dr_{b}\left( \dr_i \left[ \left(g_{\vPhi}^{ij}\, \sqrt{\det g_{\vPhi}} - g_{\vPhi_k}^{ij} \, \sqrt{\det g_{\vPhi_k}}\right) \dr_j (\vp_{\alpha}^k)^{-1} \right] \right) \ \dr_a\left(  \xi\ \lambda_{ab} \right)\ dx \\[2mm]
			& \leq C(n,\Lambda)\  \|\xi\|_{W^{1,\infty}(W)}\, \|\lambda_{ab} \|_{W^{1,\frac{n}{n-2}}(W)}\\[2mm]
			& \qquad \times \left( \left\| \g^2 (g_{\vPhi} - g_{\vPhi_k}) \right\|_{L^{\frac{n}{2}}(W,\geu)} \, \| \g (\vp_{\alpha}^k)^{-1} \|_{L^{\infty}(W)} + \left\| \g (g_{\vPhi} - g_{\vPhi_k}) \right\|_{L^n(W,\geu)} \, \| \g^2 (\vp_{\alpha}^k)^{-1} \|_{L^n(W)}  \right) .
		\end{align*}
		By \eqref{eq:g_bounded}, \eqref{eq:Unif_vp} and \eqref{eq:est_Hodge}, we obtain that the term \eqref{eq:T22} converges to 0. Hence, both \eqref{eq:T21} and \eqref{eq:T22} converges to 0, which implies that \eqref{eq:T2} converges to 0. Since we already proved that \eqref{eq:T1} converges to 0, we obtain 
		\begin{align}\label{eq:strong_cv_chart}
			(\vp_{\alpha}^k)^{-1} \xrightarrow[k\to +\infty]{} (\vp_{\alpha})^{-1} \qquad \text{strongly in }W^{3,\frac{n}{2}}_{\loc}(U_{\alpha}).
		\end{align}

		Coming back to the limit \eqref{eq:galpha}, we obtain that this limit is also strong in any $L^p$-space, so that $g_{\alpha} = (\vp_{\alpha})^*g_{\vPhi} \in W^{2,\left(\frac{n}{2},1\right)}(\B^n(0,r_{\alpha}))$ and $g_{\vPhi_k\circ\vp_{\alpha}^k}$ converges strongly in $W^{2,\frac{n}{2}}$ to $g_{\vPhi\circ\vp_{\alpha}}$. Moreover, the charts $\vp_{\alpha}^k$ provide harmonic coordinates. Hence, we have by \cite[Lemma 11.2.6]{petersen2016} that
		\begin{align}\label{eq:harm_coord}
			-\frac{1}{2} g_{\vPhi_k\circ\vp_{\alpha}^k}^{ij}\, \dr^2_{ij}\, \big( g_{\vPhi_k\circ\vp_{\alpha}^k}\big)_{\alpha\beta} = \Ric_{\alpha\beta}^{g_{\Phi_k\circ\vp_{\alpha}^k}} + Q_{\alpha\beta}\left( g_{\vPhi_k\circ\vp_{\alpha}^k}, \dr g_{\vPhi_k\circ\vp_{\alpha}^k} \right),
		\end{align}
		where the last term is quadratic in the coefficients $\dr_k (g_{\vPhi_k\circ\vp_{\alpha}^k})_{ij}$, that is to say, we have the estimate
		\begin{align*}
			\left| Q_{\alpha\beta}\left( g_{\vPhi_k\circ\vp_{\alpha}^k}, \dr g_{\vPhi_k\circ\vp_{\alpha}^k} \right) \right| \leq C(n)\, \sum_{1\leq i,j,k\leq n} \left| \dr_k \big( g_{\vPhi_k\circ\vp_{\alpha}^k} \big)_{ij} \right|^2.
		\end{align*}
		We rewrite \eqref{eq:harm_coord} in divergence form
		\begin{align}\label{eq:harm_coord2}
			-\frac{1}{2}\, \dr_i\left( g_{\vPhi_k\circ\vp_{\alpha}^k}^{ij}\, \dr_j\, \big( g_{\vPhi_k\circ\vp_{\alpha}^k}\big)_{\alpha\beta}\right) = \Ric_{\alpha\beta}^{g_{\Phi_k\circ\vp_{\alpha}^k}} + \tilde{Q}_{\alpha\beta}\left( g_{\vPhi_k\circ\vp_{\alpha}^k}, \dr g_{\vPhi_k\circ\vp_{\alpha}^k} \right),
		\end{align}
		where $\tilde{Q}$ also verifies
		\begin{align*}
			\left| \tilde{Q}_{\alpha\beta}\left( g_{\vPhi_k\circ\vp_{\alpha}^k}, \dr g_{\vPhi_k\circ\vp_{\alpha}^k} \right) \right| \leq C(n)\, \sum_{1\leq i,j,k\leq n} \left| \dr_k \big( g_{\vPhi_k\circ\vp_{\alpha}^k} \big)_{ij} \right|^2.
		\end{align*}
		Since the coefficients $g_{\vPhi_k\circ\vp_{\alpha}^k}^{ij}$ verify \eqref{eq:bound_coordinates}, we deduce from \cite[Theorem 2.4]{byun2015} that the Green kernel associated to the operator $-\frac{1}{2}\, \dr_i\left( g_{\vPhi_k\circ\vp_{\alpha}^k}^{ij}\, \dr_j(\cdot) \right)$ is uniformly bounded, seen as an operator from $L^{\left(\frac{n}{2},1\right)}\to W^{2,\left(\frac{n}{2},1\right)}$.
		Since $\vII_{\vPhi_k}$ converges strongly to $\vII_{\vPhi}$ in $L^{(n,2)}(\Sigma)$, $\g \vp_{\alpha}^k$ converges a.e. to $\g \vp_{\alpha}$ and $\g\vp_{\alpha}^k$ is uniformly bounded in $L^{\infty}$, we also have that $\Ric_{\alpha\beta}^{g_{\Phi_k\circ\vp_{\alpha}^k}}$ converges strongly to $\Ric_{\alpha\beta}^{g_{\vPhi\circ\vp_{\alpha}}}$ in $L^{\left( \frac{n}{2},1\right)}_{\loc}(\B^n(0,r_{\alpha}))$ by dominated convergence. Hence, the metrics $g_{\vPhi_k\circ\vp_{\alpha}^k}$ converge strongly in $W^{2,\left(\frac{n}{2},1\right)}_{\loc}$ to $g_{\vPhi\circ\vp_{\alpha}}$.\\
		
		Consider $\alpha,\beta\in A$ such that $U_{\alpha}\cap U_{\beta}\neq \varnothing$. Since both $U_{\alpha}$ and $U_{\beta}$ are open sets, the intersection $U_{\alpha}\cap U_{\beta}$ is also an open set. To obtain estimate on the transition functions, we use \eqref{eq:hessian_vp} or \cite[Equation (2.13)]{taylor2006}, stating that the transition maps $\vp^k_{\alpha,\beta} \coloneqq (\vp_{\alpha}^k)^{-1}\circ\vp_{\beta}^k$ verify
		\begin{align}\label{eq:hessian_vpk}
            \dr^2_{ij} (\vp_{\alpha,\beta}^k)^{\ell} = \left( {^{g_{\vPhi\circ \vp_{\beta}^k}} \Gamma}_{ij}^p \right)\, \dr_p (\vp^k_{\alpha,\beta})^{\ell} - \left( {^{g_{\vPhi\circ \vp_{\alpha}^k}} \Gamma}_{ij}^{\ell} \right)\, \left( \dr_i (\vp^k_{\alpha,\beta})^p \right)\, \left( \dr_j (\vp_{\alpha,\beta}^k)^q \right).
        \end{align}
        By \eqref{eq:strong_cv_chart}, the sequence $\vp_{\alpha,\beta}^k$ converges strongly in the $W^{3,\frac{n}{2}}_{\loc}$-topology to $\vp_{\alpha,\beta}\coloneq \vp_{\alpha}^{-1}\circ\vp_{\beta}$. Hence for any ball $\B^n(y,s)\Subset (\vp_{\beta}^k)^{-1}\left( U^k_{\alpha}\cap U^k_{\beta}\right)\cap (\vp_{\beta}^{\ell})^{-1}\left( U^{\ell}_{\alpha}\cap U^{\ell}_{\beta}\right) \cap \vp_{\beta}^{-1}\left( U_{\alpha}\cap U_{\beta}\right)$ (and $k\leq\ell$ large enough), we obtain 
        \begin{align}
            \left\| \g^2 (\vp_{\alpha,\beta}^k) - \g^2 (\vp^{\ell}_{\alpha,\beta}) \right\|_{L^{(n,1)}(\B^n(y,s),\geu)} & \leq C(n,d,E)\, \left\|  {^{g_{\vPhi\circ \vp_{\beta}^k}} \Gamma}_{ij}^p -  {^{g_{\vPhi\circ \vp_{\beta}^{\ell}}} \Gamma}_{ij}^p \right\|_{L^{(n,1)}(\B^n(y,s),\geu)} \nonumber \\[2mm]
             & \qquad +C(n,d,E)\, \eps_0\, \left\| \g (\vp_{\alpha,\beta}^k) - \g (\vp^{\ell}_{\alpha,\beta}) \right\|_{L^{\infty}(\B^n(y,s),\geu)} \label{eq:Hess_transition_est}
        \end{align}
        Up to reducing $\eps_0$ in \eqref{eq:Unif_Int_Approx} (still depending only on $n$, $d$ and $E$), we obtain by Sobolev injections that 
        \begin{align*}
            \left\| \g (\vp_{\alpha,\beta}^k) - \g (\vp^{\ell}_{\alpha,\beta}) \right\|_{L^{\infty}(\B^n(y,s),\geu)}  \leq C(n,d,E)\, \left\|  {^{g_{\vPhi\circ \vp_{\beta}^k}} \Gamma}_{ij}^p -  {^{g_{\vPhi\circ \vp_{\beta}^{\ell}}} \Gamma}_{ij}^p \right\|_{L^{(n,1)}(\B^n(y,s),\geu)} .
        \end{align*}
        By \eqref{eq:strong_g}, we obtain that $(\vp_{\alpha,\beta}^k)_{k\in\N}$ converges strongly in $W^{1,\infty}_{\loc}$ to $\vp_{\alpha,\beta}$. Coming back to \eqref{eq:Hess_transition_est}, we obtain that this convergence also holds in $W^{2,(n,1)}_{\loc}$. Differentiating once more \eqref{eq:hessian_vpk}, we obtain the strong convergence in $W^{3,\left(\frac{n}{2},1\right)}_{\loc}$. 
	\end{proof}
	
	\subsubsection{Convergence of the densities.}\label{sec:Densities}

	We now show that each map $\vPhi_{\alpha}\coloneq \vPhi\circ\vp_{\alpha} \colon \Or_{\alpha}\to \B^d(\vq_{\alpha},r_{\alpha})$ is injective. Let $\vep\in \vPhi_{\alpha}(\Or_{\alpha})$ and $s_0>0$ such that $\B^d(\vep,s_0)\Subset \B^d(\vq_{\alpha},r_{\alpha})$. As done in \Cref{sec:extension}, we can extend each $\vPhi_k\circ\vp_{\alpha}$ to an immersion $\vPsi_{\alpha,k}\colon \R^n\to \R^d$ such that 
	\begin{align}\label{eq:est_extension}
		\left\| \vII_{\vPsi_{\alpha,k}} \right\|_{L^{(n,2)}\left(\R^n,g_{\vPsi_{\alpha,k}} \right)} \leq C(n,d,D)\, \Er\left(\vPhi_k; \vPhi_k^{-1}\big( \B^d(\vq_{\alpha}^{\ k}, r_{\alpha}^k)\big) \right)^{\frac{1}{2n}}.
	\end{align}
	Moreover the immersion $\vPsi_{\alpha,k}\colon \R^n\setminus \B^n(0,2\, r_{\alpha}^k)\to \R^d$ is an embedding parametrizing a flat affine $n$-dimensional plane. We consider the monotonicity formula applied to $\vPsi_{\alpha,k}$. For any $0<s<t<+\infty$, using the notation \eqref{eq:measure}, it holds (see \eqref{eq:montonicity_balls})
	\begin{align}\label{eq:monotonicity_extension}
		\begin{aligned} 
			\frac{\mu_{\vPsi_{\alpha,k}} (\B^d(\vep,s)) }{ \Hr^n(\B^n)\, s^n} & \leq \frac{\mu_{\vPsi_{\alpha,k}} (\B^d(\vep,t)) }{ \Hr^n(\B^n)\, t^n} + \frac{1}{2\, \Hr^n(\B^n)}\int_{\B^d(\vep,t)} \frac{|\vH_{\vPsi_{\alpha,k}}|^2 }{|\vx-\vep|^{n-2}}\, d\mu_{\vPsi_{\alpha,k}}(\vx)  \\[2mm]
			& + \left(\frac{\mu_{\vPsi_{\alpha,k}} (\B^d(\vep,s)) }{ \Hr^n(\B^n)\, s^n}\right)^{\frac{n-1}{n}} \left\|\vH_{\vPsi_{\alpha,k}} \right\|_{L^n\left(\B^d(\vep,s),\mu_{\vPsi_{\alpha,k}} \right)}.
		\end{aligned} 
	\end{align}
	For $t$ large enough, the second terms verifies 
	\begin{align*}
		& \frac{1}{2\, \Hr^n(\B^n)}\int_{\B^d(\vep,t)} \frac{|\vH_{\vPsi_{\alpha,k}}|^2 }{|\vx-\vep|^{n-2}}\, d\mu_{\vPsi_{\alpha,k}}(\vx) \\[2mm]
		& = \frac{1}{2\, \Hr^n(\B^n)}\int_{\B^d(\vep,2\, r_{\alpha}^k)} \frac{|\vH_{\vPsi_{\alpha,k}}|^2 }{|\vx-\vep|^{n-2}}\, d\mu_{\vPsi_{\alpha,k}}(\vx)  \\[2mm]
		& \leq C(n)\, \left\| \frac{1}{|\cdot -\vep|} \right\|_{L^{(n,\infty)}(\B^d(\vep,r_{\alpha}^k),\mu_{\vPsi_{\alpha,k}})}\, \left\| \vII_{\vPsi_{\alpha,k}} \right\|_{L^{(n,2)}(\R^n,g_{\vPsi_{\alpha,k}})}^2.
	\end{align*}
	By \Cref{pr:Extrinsic_Hdiff} and \eqref{eq:Lninfty}, we obtain 
	\begin{align*}
		\frac{1}{2\, \Hr^n(\B^n)}\int_{\B^d(\vep,t)} \frac{|\vH_{\vPsi_{\alpha,k}}|^2 }{|\vx|^{n-2}}\, d\mu_{\vPsi_{\alpha,k}}(\vx) \leq C(n,d,D)\, \eps_0^{\frac{1}{2n}}.
	\end{align*}
	Letting $t\to +\infty$ in \eqref{eq:monotonicity_extension}, we obtain 
	\begin{align*}
		\frac{\mu_{\vPsi_{\alpha,k}} (\B^d(\vep,s)) }{ \Hr^n(\B^n)\, s^n} \leq 1 + C(n,d,D)\, \eps_0^{\frac{1}{2n}} + \eps_0^{\frac{1}{2n}}\, \left(\frac{\mu_{\vPsi_{\alpha,k}} (\B^d(\vep,s)) }{ \Hr^n(\B^n)\, s^n}\right)^{\frac{n-1}{n}}.
	\end{align*}
	For $s<s_0$, we obtain the uniform estimate
	\begin{align*}
		\frac{\mu_{\vPhi_k\circ\vp_{\alpha}} (\B^d(\vep,s)) }{ \Hr^n(\B^n)\, s^n} \leq 1 + C(n,d,D)\, \eps_0^{\frac{1}{2n}} + \eps_0^{\frac{1}{2n}}\, \left(\frac{\mu_{\vPhi_k\circ\vp_{\alpha}} (\B^d(\vep,s)) }{ \Hr^n(\B^n)\, s^n}\right)^{\frac{n-1}{n}}.
	\end{align*}
	Letting $k\to+\infty$, we obtain 
	\begin{align*}
		\frac{\mu_{\vPhi_{\alpha}} (\B^d(\vep,s)) }{ \Hr^n(\B^n)\, s^n} \leq 1 + C(n,d,D)\, \eps_0^{\frac{1}{2n}} + \eps_0^{\frac{1}{2n}}\, \left(\frac{\mu_{\vPhi_{\alpha}} (\B^d(\vep,s)) }{ \Hr^n(\B^n)\, s^n}\right)^{\frac{n-1}{n}}.
	\end{align*}
	Letting $s\to 0$, we obtain 
	\begin{align*}
		\theta_{\vPhi_{\alpha}}(\vep) \leq 1 + C(n,d,D)\, \eps_0^{\frac{1}{2n}} + \eps_0^{\frac{1}{2n}}\, \theta_{\vPhi_{\alpha}}(\vep)^{\frac{n-1}{n}}.
	\end{align*}
	By Young inequality, we obtain for any $\eta\in(0,1)$,
	\begin{align*}
		\theta_{\vPhi_{\alpha}}(\vep) \leq 1 + C(n,d,D,\eta)\, \eps_0^{\frac{1}{2n}} + \eta\, \theta_{\vPhi_{\alpha}}(\vep).
	\end{align*}
	Thus, it holds
	\begin{align*}
		\theta_{\vPhi_{\alpha}}(\vep) \leq \frac{ 1 + C(n,d,D,\eta)\, \eps_0^{\frac{1}{2n}} }{\eta}.
	\end{align*}
	We choose $\eta>\frac{3}{4}$ such that $1/\eta < \frac{4}{3}$. Up to reducing $\eps_0$ (depending only on $n$, $d$ and $D$), we obtain 
	\begin{align*}
		\theta_{\vPhi_{\alpha}}(\vep) \leq \frac{5}{3} < 2.
	\end{align*}
	Since the density of weak immersions is always an integer (see for instance \cite[Theorem IV.15]{Lan2025}), we obtain $\theta_{\vPhi(U_{\alpha})}(\vep)=1$.
	
	\subsubsection{Diffeomorphisms.}\label{sec:Diffeo}

	This is a straightforward adaptation of the proof of Theorem IV.19 in \cite{Lan2025} (see also \cite[Proposition 3.2]{uhlenbeck1982}), using that \eqref{eq:strong_Transition} implies $C^1$ convergence thanks to the embedding $W^{2,(n,1)}\hookrightarrow C^1$ in dimension $n$. 
	\begin{claim}\label{cl:existence_diffeo}
		There exists a sequence of $C^1$-diffeomorphisms $F_k \colon (\Sigma,g_{\vPhi_k}) \to (\Sigma,g_{\vPhi})$ converging to the identity map in the $C^1$-topology.
	\end{claim}
	
	Let $(\varphi_\alpha^k,U_\alpha)_{\alpha\in A}$ be the sequence of atlases introduced in \Cref{sec:Setting} for $\vec{\Phi}_k$. We reformulate \Cref{cl:existence_diffeo} as follows.	We claim that for $k$ large enough, there exists an open cover $(V_{\alpha})_{1\leq \alpha\leq l}$ of $\Sigma$ with $V_\alpha\Subset U_\alpha$ and embeddings $f_{\alpha}^k \colon \vp_{\alpha}^{-1}(\overline{V_\alpha})\subset\R^n\rightarrow (\vp_{\alpha}^k)^{-1}(U_\alpha^k)\subset\R^n$ satisfying the following two properties
	\begin{align}
		&\bullet\  f_{\alpha}^k \xrightarrow[k\rightarrow\infty]{} \Id_{\R^n} \qquad \text{ in }C^1(\vp_{\alpha}^{-1}(\overline{V_\alpha});\R^n), \label{eq:conv_fk}\\[3mm]
		&\bullet\ \vp_{\beta,\alpha}^k\circ f_{\alpha}^k\circ \vp_{\alpha,\beta}= f_{\beta}^k \qquad \text{ on } \vp_{\beta}^{-1}(\overline{V_\alpha}\cap \overline{V_\beta}).\label{cocycle}
	\end{align}  
	Once these maps $f_{\alpha}^k$ are defined, we consider the map 
	\begin{align}\label{eq:def_Fk}
		F_k \coloneqq \vp_{\alpha}^k\circ f_{\alpha}^k\circ \vp_{\alpha}^{-1} \qquad \text{ on }\overline{V_\alpha}.
	\end{align}
	This definition is independent of $\alpha$ by \eqref{cocycle} and thus, extends to the whole manifold $\Sigma$. By construction, it holds $F_k\rightarrow \Id_{\Sigma}$ in the $C^1$-topology. Hence, it is a diffeomorphism for $k$ large enough.\\
	
	We now proceed to the proof of the claim above and the construction of $(V_{\alpha})_{\alpha}$ and $(f^k_{\alpha})_{\alpha,k}$. Let $(U_\alpha')_{1\leq \alpha\leq \ell}$ be an open cover of $\Sigma$ such that $U_\alpha'$ has smooth boundary and $U_\alpha'\Subset U_\alpha$. We first prove the following result.
	\begin{claim} \label{cl:diffeo_f}
		There exist open sets $(V_{\alpha,j})_{1\leq \alpha\leq j\le l}\subset \Sigma$ (we will write $V_{\alpha,\alpha-1}=V_{\alpha,\alpha}$), $k_0\in\N$ and $C^1$ maps $f_{\alpha}^k\colon \vp_{\alpha}^{-1}(\overline{V_{\alpha,\alpha}})\rightarrow (\vp_{\alpha}^k)^{-1} (U_\alpha^k)$ for $k\ge k_0$ such that for any $1\le j\le l$ and $1\le \alpha,\beta\le j$, it holds 
		\begin{align}\label{eq:fk} 
			\begin{cases}
				(i)\ V_{\alpha,j}\subset V_{\alpha,j-1}\subset U_{\alpha}',\\[3mm]
				(ii)\ \Sigma=\Big(\bigcup\limits_{\gamma\leq j} V_{\gamma,j}\Big)\cup \Big(\bigcup \limits_{\gamma>j} U'_{\gamma}\Big),\\[3mm]
				(iii)\ f^k_{\alpha}\xrightarrow[k\rightarrow\infty]{} \Id_{\R^n} & \text{in } C^1(\vp_{\alpha}^{-1}(\overline{V_{\alpha,\alpha}}),\R^n),\\[3mm]
				(iv)\ \vp_{\beta,\alpha}^k \circ f_{\alpha}^k\circ\vp_{\alpha,\beta}= f^k_{\beta} & \text{on }\vp_{\beta}^{-1}(\overline{V_{\alpha,j}}\cap \overline{V_{\beta,j}}).
			\end{cases}
		\end{align}
	\end{claim} 
	\begin{proof} 
		We proceed by induction on $j$. \\
		
		For $j=1$, we define $V_{1,1}=U_1'$ and $f_1^k\colon \vp_1^{-1}(\overline{U_1'})\rightarrow \R^n$ as the identity map. 
		For $k$ large enough, the image of $f^k_1$ is contained in $(\vp_1^k)^{-1}(U_1^k)$ since $(\vp^k_1)^{-1}\rightarrow \vp_1^{-1}$ in $C^{0,\frac{1}{2}}(\overline{U_1'},\R^n)$ and $U_1'\Subset U_1$. \\
		
		Let $1\le j_0\le l$, suppose we have constructed $(V_{\alpha, j})_{1\le \alpha\le j\le j_0}$ and $(f^k_{\alpha})_{1\le \alpha\le j_0}$ such that \eqref{eq:fk} hold for any $1\le j\le j_0$ and $1\le \alpha,\beta\le j$. We now construct $f^k_{j_0+1}$ and $(V_{\alpha,j_0+1})_{1\le \alpha\le j_0+1}$ using the induction hypothesis. For $1\le \alpha\le j_0$, we define 
		\begin{align} \label{eq:def_omega}
			\omega^k_{j_0}:=\vp^k_{j_0+1,\alpha}\circ f^k_{\alpha}\circ \vp_{\alpha,j_0+1} \qquad \text{on }\vp_{j_0+1}^{-1}\left( \overline{V_{\alpha,j_0}} \cap U_{j_0+1}\right).
		\end{align}  
		By our induction hypothesis $(iv)$ in \eqref{eq:fk} for $j\le j_0$, this map is well-defined on $\vp_{j_0+1}\Big(\big( \bigcup\limits_{\alpha\le j_0}\overline{V_{\alpha,j_0}}\big) \cap U_{j_0+1}\Big)$ and independent of $\alpha$. 
		By the induction hypothesis $(ii)$ in \eqref{eq:fk} for $j=j_0$, we have  
		\begin{align*} 
			\left(\Sigma\setminus \bigcup\limits_{\alpha\le j_0 } V_{\alpha,j_0} \right)\cap \left(\Sigma\setminus \bigcup\limits_{\alpha> j_0}U'_{\alpha} \right)=\varnothing.
		\end{align*} 
		Since both sets on the left-hand side of the above equality are closed sets, there exists a smooth cut-off $\xi_{j_0}\in C^1(\Sigma,g_{\vPhi})$\footnote{Since the transition chart are $C^1$ by \eqref{eq:strong_Transition}, the $C^1$-topology for the metric $g_{\vPhi}$ is well-defined.} such that $\xi_{j_0}=0$ on a neighbourhood of $\Sigma\setminus\bigcup\limits_{\alpha\le j_0} V_{\alpha,j_0}$ and  $\xi_{j_0}=1$ on a neighbourhood of $\Sigma\setminus \bigcup\limits_{\alpha> j_0}U'_{\alpha}$. We define
		\begin{align*}
			\begin{cases}
				V_{j_0+1,j_0+1} \coloneqq U_{j_0+1}'\Subset U_{j_0+1} ,\\[3mm]
				V_{\alpha,j_0+1} \coloneqq V_{\alpha,j_0}\cap \,\Int\Big\{x\in\Sigma:\xi_{j_0}(x)=1\Big\} \qquad \text{ if }\alpha\le j_0.
			\end{cases}
		\end{align*}
		This definition implies the property $(i)$ in \eqref{eq:fk} for $j=j_0+1$.\\
		
		\textit{Condition $(ii)$.}\\
		We now check the property $(ii)$ in \eqref{eq:fk}. We have 
		\begin{align*}
			\Big(\bigcup\limits_{\alpha\le j_0+1} V_{\alpha,j_0+1}\Big)\cup \Big(\bigcup \limits_{\alpha>j_0+1} U'_{\alpha}\Big)&=\Big(\bigcup\limits_{\alpha\le j_0} V_{\alpha,j_0+1}\Big)\cup \Big(\bigcup \limits_{\alpha>j_0} U'_{\alpha}\Big)\\
			&= \left( \Big(\bigcup\limits_{\alpha\le j_0} V_{\alpha,j_0}\Big)\cap\,\Int \Big\{x\in\Sigma:\xi_{j_0}(x)=1\Big\} \right) \cup \Big(\bigcup \limits_{\alpha>j_0} U'_{\alpha}\Big).
		\end{align*}
		By definition of $\xi_{j_0}$ we have 
		\begin{align*} 
			\Sigma\setminus\bigcup\limits_{\alpha> j_0}U'_{\alpha}\subset \Int \Big\{x\in\Sigma:\xi_{j_0}(x)=1\Big\}.
		\end{align*} 
		We then obtain 
		\begin{align*}
			\Big(\bigcup\limits_{\alpha\le j_0+1} V_{\alpha,j_0+1}\Big)\cup \Big(\bigcup \limits_{\alpha>j_0+1} U'_{\alpha}\Big)&\supset \Big( \Big(\bigcup\limits_{\alpha\le j_0} V_{\alpha,j_0}\Big)\setminus \bigcup\limits_{\alpha> j_0}U'_{\alpha}\Big) \cup \Big(\bigcup \limits_{\alpha>j_0} U'_{\alpha}\Big) =\Big(\bigcup\limits_{\alpha\le j_0} V_{\alpha,j_0}\Big)\cup \Big(\bigcup \limits_{\alpha>j_0} U'_{\alpha}\Big) =\Sigma.
		\end{align*}
		Hence, $(ii)$ in \eqref{eq:fk} holds for $j=j_0+1$. \\
		
		\textit{Condition $(iv)$.}\\
		We now check the condition $(iv)$ in \eqref{eq:fk}. Since $\xi_{j_0}$ vanishes on a neighborhood of $\Sigma\setminus\bigcup\limits_{\alpha\le j_0} V_{\alpha,j_0}$, we obtain from \eqref{eq:strong_Transition} that
		\begin{align*}
			\omega^k_{j_0}\in C^1\left(\vp_{j_0+1}^{-1}\left[ \Big( \bigcup\limits_{\alpha\le j_0}V_{\alpha,j_0}\Big) \cap U_{j_0+1} \right] \right).
		\end{align*}
		Moreover, we have $(\xi_{j_0} \circ\vp_{j_0+1})\,\omega_{j_0}^k \in C^1(\vp_{j_0+1}^{-1}(U_{j_0+1}))$. 
		Then we define 
		\begin{align*} 
			f^k_{j_0+1} \coloneqq (\xi_{j_0}\circ \vp_{j_0+1})\,\omega_{j_0}^k+(1-\xi_{j_0}\circ \vp_{j_0+1}) \, \Id\in C^1\left( \vp_{j_0+1}^{-1}(U_{j_0+1}),\R^n \right).
		\end{align*} 
		Let $\alpha\le j_0$. Since $\xi_{j_0}=1$ on $\overline{V_{\alpha,j_0+1}}$, the following holds on $\vp_{j_0+1}^{-1}(U_{j_0+1}\cap \overline{V_{\alpha,j_0+1}})$:
		\begin{align*} 
			f^k_{j_0+1}=\omega^k_{j_0}=(\vp^k_{j_0+1,\alpha})^{-1} \circ f^k_{\alpha}\circ (\vp_{\alpha,j_0+1})^{-1}.
		\end{align*} 
		Since we have $
		\overline{V_{j_0+1,j_0+1}}=\overline{U'_{j_0+1}}\subset U_{j_0+1}$, the condition $(iv)$ of \eqref{eq:fk} is thus satisfied for $\beta=j=j_0+1$, $\alpha\le j_0$. By using $(i)$ in \eqref{eq:fk} and the identity $\vp_{\alpha,\beta}^k=(\vp^k_{\beta,\alpha})^{-1}$, we then obtain $(iv)$ in \eqref{eq:fk} for $j=j_0+1$ and any  $\alpha,\beta\le j_0+1$.\\
		
		\textit{Condition $(iii)$.}\\
		Next we prove $f^k_{j_0+1}\rightarrow \Id$ in $C^1 \left( (\vp_{j_0+1})^{-1}(\overline{U'_{j_0+1}});\R^n \right)$ as $k\rightarrow\infty$. 
		Fixing $\alpha\le j_0$, let $W$ be an open set such that $\vp_{\alpha}^{-1} (\overline{V_{\alpha,j_0}}\cap \overline{U_{j_0+1}'})\subset W\Subset \vp_{\alpha}^{-1} (U_\alpha\cap U_{j_0+1})$. By \eqref{eq:strong_Transition}, we have
		\begin{align*}
			\vp^k_{j_0+1,\alpha}\xrightarrow[k\rightarrow\infty]{} \vp_{j_0+1,\alpha} \qquad \text{in } C^1(\overline{W}).
		\end{align*} 
		By our induction hypothesis $(iii)$ of \eqref{eq:fk} for $\alpha\le j_0$, we also have 
		\begin{align*}  
			f^k_{\alpha}\xrightarrow[k\rightarrow\infty]{} \Id   \qquad \text{strongly in }C^1 \left(\vp_{\alpha}^{-1}\left( \overline{V_{\alpha,j_0}}  \right) \right).
		\end{align*} 
		The definition of $\omega_{j_0}^k$ in \eqref{eq:def_omega} implies  
		\begin{align}\label{eq:cv_omega}
			\omega^k_{j_0}\xrightarrow[k\rightarrow\infty]{} \Id \qquad \text{ in } C^1 \left(\vp_{j_0+1}^{-1}(\overline{V_{\alpha,j_0}})\cap \overline{U'_{j_0+1}} \right).
		\end{align}  
		Hence we obtain
		\begin{align*}  
			f^k_{j_0+1} \xrightarrow[k\rightarrow\infty]{} \Id \qquad \text{strongly in }  C^1\left(\vp_{j_0+1}^{-1}\left[\Big(\bigcup_{\alpha\le j_0}\overline{V_{\alpha,j_0}}\Big)\cap \overline{U'_{j_0+1}} \right] \right).
		\end{align*} 
		On the other hand, we have $f^k_{j_0+1}= \Id$ on $\vp_{j_0+1}^{-1}\left(U_{j_0+1}\setminus \bigcup\limits_{\alpha\le j_0}V_{\alpha,j_0}\right)$. Therefore, it holds
		\begin{align*} 
			f^k_{j_0+1}\xrightarrow[k\rightarrow\infty]{} \Id \qquad \text{ in }  C^1 \left( \vp_{j_0+1}^{-1}( \overline{U'_{j_0+1}}) \right).
		\end{align*} 
		This is $(iii)$ of \eqref{eq:fk} for $j=j_0+1$. In particular, $f^k_{j_0+1}$ is a $C^1$-immersion on $\vp_{j_0+1}^{-1}(\overline{V_{j_0+1,j_0+1}})=\vp_{j_0+1}^{-1}(\overline{U'_{j_0+1}})$ for $k$ large enough. Consequently, there exists $k_0\in \N$ such that for any $k\ge k_0$ and $x,y\in \vp_{j_0+1}^{-1}(\overline{U'_{j_0+1}})$ it holds 
		\begin{align*} 
			\left| (f^k_{j_0+1}-\Id)(x)-(f^k_{j_0+1}-\Id)(y) \right|\le \frac 12 \,|x-y|.
		\end{align*} 
		Therefore, we obtain 
		\begin{align*} 
			\left| f^k_{j_0+1}(x)-f^k_{j_0+1}(y) \right|\ge \frac 12 \,|x-y|.
		\end{align*} 
		In particular, this implies $f^k_{j_0+1}$ is injective on $ \vp_{j_0+1}^{-1}(\overline{U'_{j_0+1}})$ for $k\ge k_0$. 
		Since $V_{j_0+1,j_0+1}=U'_{j_0+1}\Subset U_{j_0+1}$ and $\vp^k_{j_0+1}\rightarrow\vp_{j_0+1}$ in $C^0(\overline{U'_{j_0+1}})$ as $k\rightarrow\infty$, for $k$ large enough we have 
		\begin{align*} 
			f^k_{j_0+1}\big(\vp_{j_0+1}^{-1}(\overline{V_{j_0+1,j_0+1}})\big)\subset \vp^k_{j_0+1}(U_{j_0+1}).
		\end{align*}  
		This completes the induction.
	\end{proof} 
	
	For $1\le \alpha\le l$, we define $V_\alpha \coloneqq V_{\alpha,l}$. Together with Claim \ref{cl:diffeo_f}, the constructions of the maps $(f^k_{\alpha})_{k,\alpha}$ and the refined atlas $(V_{\alpha})_{\alpha}$ are thus accomplished. We obtain the map $F_k$ by \eqref{eq:def_Fk} which converges to $\Id_{\Sigma}$ in the $C^1$ topology and such that $F_k^*g_{\vPhi_k}$ converges to $g_{\vPhi}$ in the $C^0(\Sigma)$ topology. Thus, Claim \ref{cl:existence_diffeo} is proved. The strong convergence in the $C^1$-topology implies that the identity map on $\Sigma$ realizes a diffeomorphism between the two differential structures.

	\section{Closure of the set of $C^{\infty}$-immersion under critical Sobolev norm}\label{sec:Compactness}

	In this section, we prove \Cref{th:compactness}. We consider a sequence $(\vPhi_k)_{k\in\N}\subset \Imm(\Sigma;\R^d)$ with $\Sigma$ being a closed orientable $n$-dimensional manifold with $n\geq 4$ even, and such that $\Er(\vPhi_k)<E$ for some fixed number $E>0$. Since $\Er$ is scale invariant, we can translate and dilate $\vPhi_k$ in order to have $0\in \vPhi_k(\Sigma)$ and $\diam(\vPhi_k(\Sigma))=1$. This transformation is the map $\Theta_k$ of \Cref{th:compactness}. We will work in the whole section under these assumptions. In \Cref{sec:Convergence}, we will show that the good sets \eqref{eq:good1} converge to a $C^1$ manifold and that the maps $\vPhi_k$ converge to some $\vPhi_{\infty}$. We shall also prove that the image of the bad sets \eqref{eq:bad1} is contained in a finite number of balls shrinking to their centres. In \Cref{sec:Singularities}, we will study the singularities of the limiting domain $\Sigma_{\infty}$. We will show that the preimage of a small ball centred on a singularity on the target consists in a finite union of $\B^n\setminus \{0\}$, each of them carrying coordinates in which the metric $g_{\vPhi_{\infty}}$ lies in $W^{2,\left(\frac{n}{2},1\right)}(\B^n(0,1))$.
	
	\subsection{Convergence away from a finite number of points}\label{sec:Convergence}
	
	The goal of this section is to prove the following result.
	
	\begin{theorem}\label{th:Convergence}
		Let $E>0$ and $\Sigma$ be a closed manifold of even dimension $n\geq 4$. Consider a sequence $(\vPhi_k)_{k\in\N}\subset \Imm(\Sigma;\R^d)$ such that
		\begin{align}\label{eq:Setting}
				\displaystyle\sup_{k\in\N}\ \Er(\vPhi_k) \leq E,\qquad \text{ and } \qquad 
				0\in \vPhi_k(\Sigma),\qquad \text{ and }
				\diam(\vPhi_k(\Sigma))=1.
		\end{align}
		Then there exist 
		\begin{itemize}
			\item a finite number of points $\vq_1,\ldots,\vq_I\in \R^d$,
			\item an $n$-dimensional manifold $\Sigma_{\infty}$ of class $W^{\frac{n}{2}+1,(2,1)}$ possibly not connected,
			\item a weak immersion $\vPhi_{\infty}\in \I_{\frac{n}{2}-1,2}(\Sigma_{\infty};\R^d)$,
		\end{itemize}
		such that for any $R>0$,  $(\Sigma\setminus \vPhi_k^{-1}(\B^d(\vq_1,R)\cup \cdots \cup \B^d(\vq_I,R)),g_{\vPhi_k})$ converges in the pointed Gromov--Hausdorff distance to $(\Sigma_{\infty}\setminus \vPhi_{\infty}^{-1}(\B^d(\vq_1,R)\cup \cdots \cup \B^d(\vq_I,R)),g_{\vPhi_{\infty}})$ and 
		\begin{align}
			& \vPhi_k  \xrightarrow[k\to+\infty]{} \vPhi_{\infty} \qquad \text{in }C^0_{\loc}(\Sigma_{\infty}) \label{eq:imm_limit} ,\\[2mm]
			& \overline{\vPhi_{\infty}(\Sigma_{\infty})} = \vPhi_{\infty}(\Sigma_{\infty})\cup \{\vq_1,\ldots,\vq_I\}, \label{eq:set_limit} \\[3mm]
			& \vol_{g_{\vPhi_k}}(\Sigma) \xrightarrow[k\to +\infty]{} \vol_{g_{\vPhi_{\infty}}}(\Sigma_{\infty}). \label{eq:vol_limit}
		\end{align}
		Moreover, the manifold $\left(\Sigma_{\infty},g_{\vPhi_{\infty}} \right)$ can be completed by adding a finite number of points to $\Sigma_{\infty}$ into a closed Riemannian manifold of class $W^{3,\left(\frac{n}{2},1\right)}$.
	\end{theorem}

	\subsubsection{Setting and notations.}
	For the rest of the section, we fix a sequence $(\vPhi_k)_{k\in\N}\subset \Imm(\Sigma;\R^d)$ satisfying the assumptions of \Cref{th:Convergence}. Thanks to \eqref{eq:diameter_energy_volume}, we obtain 
	\begin{align}\label{eq:vol_k}
		\vol_{g_{\vPhi}}(\Sigma) \leq C(n,d)\, E.
	\end{align}
	Given $k\in\N$, $x\in \Sigma$ and $r>0$, we will denote 
	\begin{align*}
		\Bc_k(x,r) \coloneqq \vPhi_k^{-1}\left( \B^d\left( \vPhi_k(x),r \right)  \right) \subset \Sigma.
	\end{align*}
		
	\subsubsection{Good set and bad set.}
	Given $x\in \Sigma$, there exists $r_{x,k}\in(0,1)$ such that 
	\begin{align}\label{eq:shreshold}
		\Er\left(\vPhi_k; \Bc_k(x,r_{x,k}) \right) = \frac{\eps_0}{2},
	\end{align}
	where $\eps_0$ is defined in \Cref{th:Construction_chart} with $V = C(n,d) E$ given by \eqref{eq:vol_k}. Given $r>0$ and $k\in\N$, we denote the "good" set and "bad" set of size $r$ for $\vPhi_k$ as
	\begin{align}
		& \Gr_k^r \coloneqq \bigcup \left\{ \Bc_k(x,r_{x,k}) : x\in\Sigma,\ r_{x,k}\geq r\right\}, \label{eq:def_Good}\\[3mm]
		& \Br_k^r \coloneqq \bigcup \left\{ \Bc_k(x,r_{x,r}) : x\in\Sigma,\ r_{x,k}< r\right\}.\label{eq:def_Bad}
	\end{align}
	\begin{remark}
		We will show that the good sets $\Gr^r_k$ converges as $k\to \infty$ and then $r\to 0$ to $\Sigma_{\infty}$ in the pointed Gromov--Hausdorff topology. We will show that the images of the bad sets $\Br^r_k$ converge in Hausdorff distance to a finite set of points in $\R^d$. In contrast with \cite{laurain2018,mondino2014}, we do not develop in this work, the blow-up procedure necessary to study the structure and the formation of these singularities. In dimension 2, it followed from the Deligne--Mumford compactness and the control of the conformal factor that one can keep track of the full topology of the domain, identify bubbles as topological spheres, and discuss the notion of "neck region" or "collar". There is no such description in higher dimensions, hence we rely on the behaviour of the image instead of the domain.
	\end{remark}
	We now prove that $\vPhi_k(\Br^r_k)$ is a small set for $r$ small, uniformly in $k$.
	
	\begin{claim}\label{cl:Bad_Small}
		There exists $I\in\N$ depending only on $n$, $d$ and $E$ such that for any $r>0$, the set $\vPhi_k(\Br^r_k)$ is included in at most $I$ Euclidean balls of $\R^d$ of radius $2r$. In particular, there exists a constant $C_0>0$ depending only on $n$, $d$ and $E$ such that
		\begin{align*}
			\forall r>0,\ \forall k\in\N, \qquad \vol_{g_{\vPhi_k}}(\Br_k^r) \leq C_0\, r^n.
		\end{align*}
		Moreover, each connected component of $\vPhi_k(\Br_k^r)$ has diameter at most $2Ir$.
	\end{claim}
	\begin{proof}
		Let $r>0$. For every $x\in \Br_k^r$, we have the following inclusions 
		\begin{align*}
			\vPhi_k\Big(\Bc_k(x,r_{x,k})\Big) \subset \B^d\big(\vPhi_k(x),r_{x,k}\big) \subset \B^d\big(\vPhi_k(x),r\big).
		\end{align*}
		From the inclusion $\overline{\vPhi_k(\Br^r_k)}\subset \bigcup_{x\in \Br^r_k} \B^d\big(\vPhi_k(x),2r\big)$, we can extract a Besicovitch covering 
		\begin{align*}
			\overline{\vPhi_k(\Br^r_k)}\subset \bigcup_{1\leq i\leq I_k} \B^d \left( \vPhi_k(x_i),2r \right).
		\end{align*}
		We obtain a bound on $I_k$ using \eqref{eq:shreshold}:
		\begin{align*}
			\frac{ I_k\, \eps_0}{2} & \leq \sum_{i=1}^{I_k} \Er\Big(\Phi_k; \Bc_k(x_i,2r)  \Big)   \leq C(d)\, \Er\Bigg(\vPhi_k; \vPhi_k^{-1}\Big[ \bigcup_{1\leq i\leq I_k} \B^d\big( \vPhi_k(x_i),2r \big) \Big] \Bigg)  \leq C(d)\ \Er\left(\vPhi_k; \Sigma  \right) \leq C(d)\, E.
		\end{align*}
		Hence, we have the uniform bound $I_k\leq C(d)\, \eps_0^{-1}\, E$. Thanks to \Cref{pr:Extrinsic_Hdiff}, we obtain 
		\begin{align*}
			\vol_{g_{\vPhi_k}}(\Br^r_k) \leq \sum_{i=1}^{I_k} \vol_{g_{\vPhi_k}}\left( \Bc_k(x_i,2r) \right) \leq C(n,d,E)\, r^n.
		\end{align*}
	\end{proof}
	
	\subsubsection{Convergence of the good set.}
	
	We now focus on the sets $\Gr^r_k$. We first show that $(\Gr^r_k,g_{\vPhi_k})$ converges in the pointed Gromov--Hausdorff topology to a Riemannian manifold of class $W^{\frac{n}{2}+1,\left( 2,1 \right)}$.
	
	\begin{lemma}\label{lm:Conv_Good}
		Let $r\in(0,\frac{1}{2})$.
		Up to a subsequence, the sequence $(\Gr^r_k,g_{\vPhi_k})_{k\in\N}$ converges to some $W^{\frac{n}{2}+1,\left( 2,1 \right)}$-Riemannian manifold $\Gr^r_{\infty}$ in the pointed Gromov--Hausdorff topology. The sequence $\vPhi_k\colon \Gr^r_k\to \R^d$ converges uniformly to some limiting weak immersion $\vPhi_{\infty}\in \I_{\frac{n}{2}-1,2}(\Gr^r_{\infty};\R^d)$.
	\end{lemma}
	
	The main difficulty here is the distinction between extrinsic and intrinsic balls. If the volume growth of extrinsic balls obtained in Propositions \ref{pr:Extrinsic_Hdiff} and \ref{pr:Lower_Extrinsic} were true for intrinsic balls, then the pointed Gromov--Hausdorff convergence (Step 1 of the proof) would be a direct application of the standard criteria. The regularity of $\Gr_{\infty}$ (Step 2 of the proof) is obtained by generalizing the arguments developed in \Cref{sec:Altas_weak_imm}.
	
	\begin{remark}
		A straightforward adaptation of the construction of the diffeomorphisms in \Cref{sec:Diffeo} in this setting yields a sequence of maps $F_k^r\colon \Gr^r_k\to \Gr^r_{\infty}$ that converge weakly in $W^{3,\left(\frac{n}{2},1\right)}$ to the identity. It does not seem enough to conclude that $F^r_k$ is a homeomorphism for $k$ large enough.
	\end{remark}
	
	\begin{proof}
		The proof is an adaptation of \cite[Theorem 11.3.6]{petersen2016}. By \Cref{th:Atlas2}, we have an atlas of $\Gr^r_k$ satisfying uniform estimates, i.e. we have a set $\left(U_{\alpha}^k,\vp_{\alpha}^k \right)_{\alpha\in A_r^k}$ such that 
		\begin{enumerate}
			\item It holds $\Gr^r_k = \bigcup_{\alpha\in A_r^k} U_{\alpha}^k$.
			\item The charts $\vp_{\alpha}^k\colon \Or^k_{\alpha} \to U_{\alpha}^k$ satisfy all the properties of \Cref{th:Atlas2} for $\Sigma_{\vartheta} = \Gr^{r/2}_k$.
		\end{enumerate}
		The convergence of the open sets $\Or_{\alpha}^k$ to a $C^{1,\alpha}$ open set $\Or_{\alpha}^{\infty}$ (with a diffeomorphism of class $W^{2,n}(\s^{n-1}(0,r_{\alpha}^{\infty}))$ between $\s^{n-1}(0,r_{\alpha}^{\infty})$ and $\dr\Or_{\alpha}^{\infty}$ for some $r_{\alpha}^{\infty}\in[r,1]$), is proved in the same manner as in \Cref{sec:Oalpha}. Given $k\in \N$, we fix $p_k\in \Gr^{2r}_k$ a reference point.\\
		
		\textit{Step 1: Up to a subsequence, $\left( \Gr^r_k,\dist_{g_{\vPhi_k}},p_k \right)_{k\in\N}$ converge in the pointed Gromov--Hausdorff topology to some metric space $\left( \Gr^r_{\infty},\dist_{\infty},p_{\infty} \right)$.}\\
		By Proposition \ref{pr:charac_GH}, it suffices to show that for any fixed $R>0$ and $r>\eta>0$, the number of disjoint balls $B_{g_{\vPhi_k}}(x,\eta)$ of fixed radius $\eta$ contained in the ball $B_{g_{\vPhi_k}}(p_k,R)$ is bounded from above by a constant depending only on $n,d,E,r,R,\eta$, that is to say, independent of $x$ and $k$.\\
		
		For $\eps_0>0$ small enough (depending only on $n$, $d$ and $V$ by \Cref{it:Transition2} of \Cref{th:Atlas2}), each chart $\vp_{\alpha}^k \colon \Or_{\alpha}^k \to U_{\alpha}^k\subset \Gr^r_k$ is $\frac{3}{2}$-Lipschitz, with $r\leq r_{\alpha}^k\leq 1$ and $\dr \Or_{\alpha}^k$ is uniformly bi-Lipschitz to $\dr \B^n(0,r_{\alpha}^k)$. There exists $N_0=N_0(n)\in\N$ such that every set $\Or_{\alpha}^k$ can be covered by at most $N_0$ balls of fixed radius $\frac{1}{9} r$. Since $\vp_{\alpha}^k$ is $\frac{3}{2}$-Lipschitz, this implies that any ball $B_{g_{\vPhi_k}}(x,\frac{2}{3}r)\subset U_{\alpha}^k$ can be covered by $N_0$ balls (for the metric $g_{\vPhi_k}$) of fixed radius $\frac{1}{6}r$. \\
		
		We now show by induction on $\ell\in \N$, the following statement:
		\begin{align}\label{Rec}
			\text{Any ball $B_{g_{\vPhi_k}}(x,\ell r/3)\subset \Gr^r_k$ can be covered by $N_1^{\ell}$ balls of radius $\frac{1}{6}r$ for some $N_1 = N_1(n,d,E,r,R)\in \N$. }
		\end{align}
		For $\ell\in\{1,2\}$, we just proved it, assuming that $B_{g_{\vPhi_k}}(x,\frac{2}{3}r)$ is contained in some $U_{\alpha}^k$. If not, then $B_{g_{\vPhi}}(x,\frac{2}{3}r)$ is contained in at most $C(n,d)\, \Ec(\vPhi)$ (bounded from above by  $C(n,d,E)\, r^{-n}$) different charts, by \eqref{eq:Setting}, \Cref{it:Size2} of \Cref{th:Atlas2} and \eqref{eq:est_Volume}. We obtain \eqref{Rec} for $\ell\in\{1,2\}$.\\
		
		Assume now that \eqref{Rec} holds for all $\ell\in\{1,\ldots,\ell_0\}$ for some $\ell_0\in\N$. Consider a ball $B_{g_{\vPhi_k}}(x,(\ell_0+1)r/3)\subset \Gr^r_k$. By \eqref{Rec}, there exists $s\in\{ 1,\ldots, N_1^{\ell_0}\}$ and balls $\left( B_{g_{\vPhi_k}}(x_i,r/6) \right)_{1\leq i\leq s}$  such that 
		\begin{align*}
			B_{g_{\vPhi_k}}\left(x,\frac{\ell_0}{3}r \right) \subset \bigcup_{i=1}^s B_{g_{\vPhi_k}}\left(x_i, \frac{r}{6} \right).
		\end{align*}
		Hence, we have 
		\begin{align*}
			B_{g_{\vPhi_k}}\left(x,\frac{\ell_0+1}{3}r \right) \subset \bigcup_{i=1}^s B_{g_{\vPhi_k}}\left(x_i, \frac{r}{2} \right).
		\end{align*}
		By \eqref{Rec}, each ball $B_{g_{\vPhi_k}}\left(x_i, \frac{r}{2} \right)\subset B_{g_{\vPhi_k}}\left(x_i, \frac{2r}{3} \right)$ can be covered by $N_1$ balls of radius $\frac{r}{6}$. Hence, the ball $B_{g_{\vPhi_k}}\left(x,\frac{\ell_0+1}{3}r \right)$ can be covered by $s\cdot N_1 \leq N_1^{\ell_0+1}$ balls. Thus, the statement \eqref{Rec} is proved by induction.\\
		
		Consider now a ball $B_{g_{\vPhi_k}}(p_k,R)\subset \Gr^r_k$. Let $\ell\in \N$ be such that $\ell\, r/3\leq R\leq (\ell+1)r/3$. By \eqref{Rec}, we have
		\begin{align*}
			\vol_{g_{\vPhi_k}}\left( B_{g_{\vPhi_k}}(p_k,R) \right) \leq C(n,d)\, N_1^{\ell+1} \leq C(n,d,r,R,E).
		\end{align*}
		Let now $\eta\in (0,r)$ and consider $s$ disjoint balls $\left( B_{g_{\vPhi_k}}(x_i,\eta)\right)_{1\leq i\leq s}$ all contained in $B_{g_{\vPhi_k}}(p_k,R)$. Then we have 
		\begin{align*}
			C(n,d,r,R,E) \geq \sum_{i=1}^s \vol_{g_{\vPhi_k}} \left( B_{g_{\vPhi_k}}(x_i,\eta) \right).
		\end{align*}
		By definition of $r$, choice of $\eps_0$ and \Cref{th:Construction_chart}, each ball $B_{g_{\vPhi_k}}(x_i,\eta)$ is contained in a chart carrying harmonic coordinates satisfying the estimate \eqref{eq:bound_coordinates}. By choice of $\eps_0$, we have, for all $1\leq i\leq s$,
		\begin{align*}
			\vol_{g_{\vPhi_k}} \left( B_{g_{\vPhi_k}}(x_i,\eta) \right) \geq C(n)\, \eta^n.
		\end{align*}
		We obtain $s \leq C(n,d,r,R,E,\eta)$ and Step 1 is proved.\\
		
		\textit{Step 2: The metric space $(\Gr^r_{\infty},\dist_{\infty})$ is a manifold of class $W^{3,\left(\frac{n}{2},1\right)}$ and $\vPhi_k$ converges uniformly to a weak immersion.}\\
		Let $\ell\in \N$. 
		Thanks to Lemma \ref{lm:ArzelaAscoli}, we can assume that up to a subsequence, all the charts $\vp_{\alpha}^k\colon (\overline{\Or_{\alpha}^k},\geu)\to \overline{ U_{\alpha}^k }\subset (\Gr^r_k,\dist_{g_{\vPhi_k}})$ are uniformly bi-Lipschitz and thus, converge (up to a subsequence) uniformly to bi-Lipschitz maps $\vp_{\alpha}^{\infty}\colon (\overline{ \Or_{\alpha}^{\infty} } ,\geu)\to \overline{ U_{\alpha}^{\infty} }\subset (\Gr^r_{\infty},\dist_{\infty})$, where $\Or_{\alpha}^{\infty}$ is an open set of $\R^n$ satisfying \eqref{eq:bound_domain}. 
		We obtain an atlas of $\Gr^r_{\infty}$, which makes it a topological manifold.\\
		
		We now consider the metrics $(\vp_{\alpha}^k)^*g_{\vPhi_k}$ on $\Or_{\alpha}^k$. Thanks to \Cref{it:Harmonic2} of \Cref{th:Atlas2}, we have up to a subsequence
		\begin{align}\label{eq:def_galpha}
			(\vp_{\alpha}^k)^*g_{\vPhi_k} \xrightarrow[k\to +\infty]{} g_{\alpha} \qquad \text{ weakly in }W_{\loc}^{2,\left(\frac{n}{2},1\right)}(\Or_{\alpha}^{\infty}) \text{ and strongly }L^p_{\loc}(\Or_{\alpha}^{\infty}) \text{ for }1\leq p<+\infty.
		\end{align}
		Since the convergence is holds almost everywhere, we obtain that $\dist_{g_{\alpha}}$ and $(\vp_{\alpha}^{\infty})^*\dist_{\infty}$ agree (almost everywhere, and thus everywhere by continuity), that is to say the map $\vp_{\alpha}^{\infty}\colon (\Or_{\alpha}^{\infty},\dist_{g_{\alpha}}) \to (U_{\alpha}^{\infty},\dist_{\infty})$ is an isometry. Moreover, the maps $\vPhi_{\alpha,k}\coloneq \vPhi_k\circ \vp_{\alpha}^k\colon \overline{ \Or_{\alpha}^k }\to \overline{ \B^d(\vq_{\alpha}^{\ k},r_{\alpha}^k) }$ are uniformly Lipschitz by \eqref{eq:def_galpha}, with $r_{\alpha}^k\in[r,1]$, and uniformly bounded in $W^{2,\left(n,2\right)}_{\loc}$ since 
		\begin{align*}
			\dr^2_{ij}\vPhi_{\alpha,k} & = (\vII_{\vPhi_{\alpha,k}})_{ij} + \left( \dr^2_{ij}\vPhi_{\alpha,k}\cdot \dr_a \vPhi_{\alpha,k}\right)\, g_{\vPhi_k\circ\vp_{\alpha}^k}^{ab}\, \dr_b \vPhi_{\alpha,k}   = (\vII_{\vPhi_{\alpha,k}})_{ij} + \left({^{g_{\vPhi_{\alpha,k}}} \Gamma }_{ij}^b\right)\, \dr_b \vPhi_{\alpha,k} \in L^{(n,2)}(\Or_{\alpha}^k).
		\end{align*}
		Up to a subsequence, we can assume that both $\vq_{\alpha}^{\ k}$ and $r_{\alpha}^k$ converge toward some $\vq_{\alpha}^{\ \infty}\in \R^d$ and $r_{\alpha}^{\infty}\in[r,1]$. Hence up to a subsequence, the sequence $(\vPhi_{\alpha,k})_{k\in\N}$ uniformly converges to some $\vPhi_{\alpha,\infty}\in \I_{0,(n,2)}\left( \Or_{\alpha}^{\infty} ; \B^d(\vq_{\alpha}^{\ \infty},r_{\alpha}^{\infty}) \right) $. Following the arguments of \Cref{sec:Densities}, all the maps $\vPhi_{\alpha,k}$ and $\vPhi_{\alpha,\infty}$ are injective.\\
		
		Consider now $\alpha\neq \beta$ such that $U_{\alpha}^{\infty}\cap U_{\beta}^{\infty}\neq \varnothing$. Then we have
		\begin{align*}
			& (\vp_{\alpha}^k)^{-1}\left( U_{\alpha}^k\cap U_{\beta}^k \right) = \vPhi_{\alpha,k}^{-1}\left( \B^d(\vq_{\alpha}^{\ k}, r_{\alpha}^k)\cap \B^d(\vq_{\beta}^{\ k}, r_{\beta}^k) \right) \subset \Or_{\alpha}^k\subset \R^n, \\[2mm]
			& (\vp_{\alpha}^{\infty})^{-1}\left( U_{\alpha}^{\infty}\cap U_{\beta}^{\infty} \right) = \vPhi_{\alpha,\infty}^{-1}\left( \B^d(\vq_{\alpha}^{\ {\infty}}, r_{\alpha}^{\infty})\cap \B^d(\vq_{\beta}^{\ {\infty}}, r_{\beta}^{\infty}) \right) \subset \Or_{\alpha}^{\infty} \subset\R^n.
		\end{align*}
		Since the maps $\vPhi_{\alpha,k}\colon \Or_{\alpha}^k\to \vPhi_{\alpha,k}(\Or_{\alpha}^k)$ and $\vPhi_{\alpha,\infty}\colon \Or_{\alpha}^k\to \vPhi_{\alpha,k}(\Or_{\alpha}^{\infty})$ are uniformly continuous bijections, we have the following property\footnote{
			We show the last limit. We denote 
			\begin{align*}
				E^k\coloneq \vPhi_{\alpha,k}^{-1}\left( \B^d(\vq_{\alpha}^{\ k},r_{\alpha}^k) \cap \B^d(\vq_{\beta}^{\ k},r_{\beta}^k) \right), \qquad E^{\infty} \coloneq \vPhi_{\alpha,\infty}^{-1}\left(  \B^d(\vq_{\alpha}^{\ \infty},r_{\alpha}^{\infty}) \cap \B^d(\vq_{\beta}^{\ \infty},r_{\beta}^{\infty}) \right) .
			\end{align*}
			If $x\in E^{\infty}$, then there exists $s>0$ such that 
			\begin{align*}
				\vPhi_{\alpha,\infty}\left( \B^n(x,s)\right) \Subset \B^d(\vq_{\alpha}^{\ \infty},r_{\alpha}^{\infty}) \cap \B^d(\vq_{\beta}^{\ \infty},r_{\beta}^{\infty}) .
			\end{align*}
			Hence, for $k\geq 0$ large enough, it holds 
			\begin{align*}
				\vPhi_{\alpha,k}\left( \B^n(x,\frac{s}{2})\right) \Subset \B^d(\vq_{\alpha}^{\ k},r_{\alpha}^{k}) \cap \B^d(\vq_{\beta}^{\ k},r_{\beta}^{k}) .
			\end{align*}
			Hence, we have for any $\delta>0$ that $E^{\infty}\subset \liminf_{k\to +\infty} B_{\delta}(E^k)$, where $B_{\delta}(E^k)$ denotes the $\delta$-neighbourhood of $E^k$. Conversely, it holds 
			\begin{align*}
				\limsup_{k\to +\infty} E^k \subset B_{\delta}\left[ \limsup_{k\to +\infty} \vPhi^{-1}_{\alpha,k}\left( \B^d(\vq_{\alpha}^{\ \infty},r_{\alpha}^{\infty}) \cap \B^d(\vq_{\beta}^{\ \infty},r_{\beta}^{\infty}) \right) \right] \subset B_{\delta}(E^{\infty}).
			\end{align*}
			Since this is valid for any $\delta>0$, we obtain $\dist_H(E^k,E^{\infty})\xrightarrow[k\to +\infty]{}0$.
		}:
		
		\begin{align*}
			& \dist_H\left[ (\vp_{\alpha}^k)^{-1}\left( U_{\alpha}^k\cap U_{\beta}^k \right), (\vp_{\alpha}^{\infty})^{-1}\left( U_{\alpha}^{\infty}\cap U_{\beta}^{\infty} \right) \right] \\[2mm]
			& = \dist_H\left[ \vPhi_{\alpha,k}^{-1}\left( \vPhi_{\alpha,k}(\Or_{\alpha}^k)\cap \vPhi_{\beta,k}(\Or_{\beta}^k) \right) , \vPhi_{\alpha,\infty}^{-1}\left( \vPhi_{\alpha,\infty}(\Or_{\alpha}^{\infty})\cap \vPhi_{\beta,\infty}(\Or_{\beta}^{\infty}) \right) \right] \\[2mm]
			& = \dist_H\left[ \vPhi_{\alpha,k}^{-1}\left( \B^d(\vq_{\alpha}^{\ k},r_{\alpha}^k) \cap \B^d(\vq_{\beta}^{\ k},r_{\beta}^k) \right) , \vPhi_{\alpha,\infty}^{-1}\left(  \B^d(\vq_{\alpha}^{\ \infty},r_{\alpha}^{\infty}) \cap \B^d(\vq_{\beta}^{\ \infty},r_{\beta}^{\infty}) \right) \right] \\[2mm]
			& \xrightarrow[k\to +\infty]{} 0.
		\end{align*}
		This implies in particular that, for $k$ large enough, we have $U_{\alpha}^k\cap U_{\beta}^k\neq \varnothing$ and that for every open set $\Omega\Subset U_{\alpha}^{\infty}\cap U_{\beta}^{\infty}$, we have $\Omega\Subset U_{\alpha}^k\cap U_{\beta}^k$ for $k$ large enough.\\
		
		Thanks to \Cref{it:Transition2} of \Cref{th:Atlas2}, the transition functions $\vp_{\alpha,\beta}^k\coloneqq (\vp_{\alpha}^k)^{-1}\circ \vp_{\beta}^k$ are uniformly bounded in $W^{3,\left(\frac{n}{2},1\right)}_{\loc}\left((\vp_{\beta}^{\infty})^{-1}(U^{\infty}_{\alpha}\cap U^{\infty}_{\beta})\right)$. Up to a subsequence, we can assume that $\left( \vp_{\alpha,\beta}^k\right)_{k\in\N}$ converges strongly in $W^{1,p}_{\loc}$ for any $p\in[1,+\infty)$ and weakly in $W^{3,\left(\frac{n}{2},1\right)}_{\loc}$ to $\vp_{\alpha,\beta}^{\infty}$. Therefore, the topological manifold $\Gr^r_{\infty}$ is actually a $C^1$ manifold and the family of metrics $g_{\alpha}$ defined in \eqref{eq:def_galpha} actually defines a Riemannian metric $g_{r,\infty}$ on $\Gr^r_{\infty}$ of class $W^{2,\left(\frac{n}{2},1\right)}$. We obtain the following properties: 
		\begin{align}
			& \begin{aligned} \label{eq:cv_transition} 
			& \vp_{\alpha,\beta}^k\xrightarrow[k\to +\infty]{} \vp^{\infty}_{\alpha,\beta}  \\[2mm]
			& \qquad  \text{ strongly in }W^{1,p}_{\loc} \text{ for any }p\in[1,+\infty) \text{ and weakly in }W^{3,\left(\frac{n}{2},1\right)}_{\loc}\left((\vp_{\beta}^{\infty})^{-1}(U^{\infty}_{\alpha}\cap U^{\infty}_{\beta})\right), 
			\end{aligned} \\[3mm]
			& g_{\vPhi_k\circ\vp_{\alpha}^k} \xrightarrow[k\to +\infty]{} (\vp^{\infty}_{\alpha})^* g_{r,\infty} \qquad \text{ strongly in }L^p_{\loc} \text{ for any }p\in[1,+\infty) \text{ and weakly in }W^{2,\left(\frac{n}{2},1\right)}_{\loc}\left(\Or_{\alpha}^{\infty}\right). \label{eq:cv_metric}
		\end{align}
		Moreover, if $U_{\alpha}^{\infty}\cap U_{\beta}^{\infty}\neq \emptyset$, then we have $\vPhi_{\alpha,\infty} = \vPhi_{\beta,\infty}\circ \vp_{\beta,\alpha}^{\infty}$. Therefore, the collection of maps $(\vPhi_{\alpha,\infty})_{\alpha}$ defines a map $\vPhi_{\infty}\colon \Gr^r_{\infty}\to \R^d$ such that with $g_{r,\infty} = g_{\vPhi_{\infty}}$. Hence $g_{\vPhi_{\infty}}$ verifies the estimates \eqref{eq:bound_coordinates}. Thus we have $\vPhi_{\infty}\in \I_{0,(n,2)}(\Gr^r_{\infty};\R^d)$.\\

		\textit{Step 3: The manifold $(\Gr^r_{\infty},g_{r,\infty})$ is of class $W^{\frac{n}{2}+1,(2,1)}$ and we have $\vPhi_{\infty}\in \I_{\frac{n}{2}-1,2}(\Gr^r_{\infty};\R^d)$.}\\
		
		Let $x_k\in \Gr^r_k$ be converging to some $x_{\infty}\in \Gr^r_{\infty}$ and consider $\vp^k_{\alpha} \colon \Or_{\alpha}^k\to U_{\alpha}^k\subset \Gr^r_k$ be the harmonic coordinates for $\vPhi_k$ introduced at the beginning of the proof, with $x_k\in U_{\alpha}^k$ and $x_{\infty}\in U_{\alpha}^{\infty}$. Then $(\vp^k_{\alpha})_{k\in\N}$ converge uniformly to a chart $\vp^{\infty}_{\alpha}\colon \Or_{\alpha}^{\infty} \to \Gr^r_{\infty}$ providing harmonic coordinates as well for $g_{\vPhi_{\infty}\circ\vp_{\alpha}^{\infty}}$. 
		We bootstrap the level of convergence using the following two systems:
		\begin{align}\label{eq:system_Phi}
			\begin{cases}
				\displaystyle \lap_{g_{\vPhi_k\circ\vp_{\alpha}^k}} \left( \vPhi_k\circ\vp_{\alpha}^k\right) = n\, \vH_{\vPhi_k\circ\vp^k_{\alpha}} ,\\[3mm]
				\displaystyle -\frac{1}{2}\ g_{\vPhi_k\circ\vp_{\alpha}^k}^{ij}\ \dr^2_{ij}\, \left(g_{\vPhi_k \circ\vp_{\alpha}^k}\right)_{ab} = \Ric_{ab}^{\vPhi_k\circ\vp^k_{\alpha}} + Q_{ab}\left( g_{\vPhi_k\circ\vp^k_{\alpha}}, \dr\, g_{\vPhi_k\circ\vp^k_{\alpha}} \right).
			\end{cases}
		\end{align}
		In the second system, the term $Q_{ab}$ is polynomial in the coefficient $\big( g_{\vPhi_k\circ\vp^k_{\alpha}} \big)_{ab}$, $ g_{\vPhi_k\circ\vp^k_{\alpha}}^{ab}$ and $\dr_{\gamma} \big( g_{\vPhi_k\circ\vp^k_{\alpha}}\big)_{ab}$, with the pointwise estimate
		\begin{align*}
			\forall \alpha,\beta\in\{1,\ldots,n\},\qquad  \left| Q_{ab}\left( g_{\vPhi_k\circ\vp^k_{\alpha}}, \dr\, g_{\vPhi_k\circ\vp^k_{\alpha}} \right) \right| \leq C(n)\, \sum_{1\leq i,j,k\leq n} \left| \dr_{k} \big( g_{\vPhi_k\circ\vp^k_{\alpha}}\big)_{ij} \right|^2.
		\end{align*}
		The bootstrap argument goes as follows:\\
		If $\left(g_{\vPhi_k\circ\vp^k_{\alpha}}\right)_{ab}$ is bounded in $W^{k,\left(\frac{n}{k},1\right)}_{\loc}$ with $k<\frac{n}{2}$, $\vII_{\vPhi_k\circ\vp_{\alpha}^k}$ is bounded in $W^{k-1,\left(\frac{n}{k},2\right)}$ and $\vH_{\vPhi_k\circ\vp_{\alpha}^k}$ is bounded in $W^{\frac{n}{2}-1,2}_{\loc}$, then the second system in \eqref{eq:system_Phi} implies that it is also bounded in $W^{k+1,\left(\frac{n}{k+1},1\right)}_{\loc}$. The first system in \eqref{eq:system_Phi} implies that $\vPhi_k\circ\vp_{\alpha}^k$ is bounded in $W^{k+2,\left(\frac{n}{k+1},2\right)}_{\loc}$. Hence, $\vII_{\vPhi_k\circ\vp_{\alpha}^k}$ is bounded in $W^{k,\left(\frac{n}{k+1},2\right)}$ and the coefficients of $\Ric_{ab}^{\vPhi_k\circ\vp^k_{\alpha}}$ are bounded in $W^{k,\left(\frac{n}{k+2},1\right)}$ and we deduce that $\left(g_{\vPhi_k\circ\vp^k_{\alpha}}\right)_{ab}$ is bounded in $W^{k+2,\left(\frac{n}{k+2},1\right)}_{\loc}$.\\
		
		Eventually, we obtain that the maps $\vPhi_k\circ\vp_{\alpha}^k$ are bounded in $W^{\frac{n}{2}+1,2}_{\loc}$ and the matrices $\left(g_{\vPhi_k\circ\vp^k_{\alpha}}\right)_{ab}$ are uniformly bounded in the $W^{\frac{n}{2},(2,1)}_{\loc}(\Or_{\alpha}^{\infty})$-topology. Since the transition maps $\vp_{\alpha,\beta}^{\infty}$ are harmonic for the metric $g_{\vPhi_{\infty}\circ\vp_{\beta}^{\infty}}$, we obtain that every $\vp_{\alpha,\beta}^{\infty}$ lies in $W^{\frac{n}{2}+1,(2,1)}$. Hence the manifold $(\Gr^r_{\infty},g_{\vPhi_{\infty}})$ carries a differential structure of class $W^{\frac{n}{2}+1,(2,1)}$ and $\vPhi_{\infty}\in \I_{\frac{n}{2}-1,2}(\Gr^r_{\infty};\R^d)$.
	\end{proof}
	
	Since for each $k\in\N$ and $0<r_1<r_2<s_0$, we have the inclusion $\Gr^{r_2}_k\subset \Gr^{r_1}_k$, we obtain $\Gr^{r_2}_{\infty}\subset \Gr^{r_1}_{\infty}$. Hence we can define the following manifold:
	\begin{align}\label{eq:limits}
		\Sigma_{\infty} \coloneqq \bigcup_{r>0} \Gr^r_{\infty}.
	\end{align}
	In particular, we obtain that $(\vPhi_k)_{k\in\N}$ converges in $C^0_{\loc}$ to $\vPhi_{\infty}$. This completes the proof of \eqref{eq:imm_limit} in \Cref{th:Convergence}. We now justify that the limit \eqref{eq:vol_limit} is a consequence of Claim \ref{cl:Bad_Small} and \eqref{eq:def_galpha}. Indeed, we have 
	\begin{align}\label{eq:finite_volume}
		\vol_{g_{\vPhi_{\infty}}}(\Sigma_{\infty}) = \lim_{r\to 0} \vol_{g_{\vPhi_{\infty}}}(\Gr^r_{\infty})= \lim_{r\to 0} \lim_{k\to +\infty} \vol_{g_{\vPhi_k}}(\Gr^r_k) \leq C(n,d)\, E.
	\end{align}
	Moreover, we have for any $r>0$
	\begin{align*}
		\left| \vol_{g_{\vPhi_k}}(\Sigma) - \vol_{g_{\vPhi_{\infty}}}(\Sigma_{\infty}) \right| & \leq \vol_{g_{\vPhi_k}}(\Br^r_k) + \left| \vol_{g_{\vPhi_k}}(\Gr^r_k) - \vol_{g_{\vPhi_{\infty}}}(\Gr^r_{\infty}) \right| + \vol_{g_{\vPhi_{\infty}}}(\Sigma_{\infty}\setminus \Gr^r_{\infty}).
	\end{align*}
	By Claim \ref{cl:Bad_Small} and \eqref{eq:def_galpha}, it holds 
	\begin{align*}
		\limsup_{k\to +\infty} \left| \vol_{g_{\vPhi_k}}(\Sigma) - \vol_{g_{\vPhi_{\infty}}}(\Sigma_{\infty}) \right|  \leq C(n,d,E)\, r^n + \vol_{g_{\vPhi_{\infty}}}(\Sigma_{\infty}\setminus \Gr^r_{\infty}).
	\end{align*}
	Letting $r\to 0$, we deduce from \eqref{eq:finite_volume} that the last term of the right-hand side converges to $0$. We obtain 
	\begin{align*}
		\limsup_{k\to +\infty} \left| \vol_{g_{\vPhi_k}}(\Sigma) - \vol_{g_{\vPhi_{\infty}}}(\Sigma_{\infty}) \right| =0.
	\end{align*}
	
	\subsection{Analysis of the singularities}\label{sec:Singularities}
	The goal of this section is to describe the singularities of the manifold $\Sigma_{\infty}$ introduced in \eqref{eq:limits}. We prove the following result.
	
	\begin{theorem}\label{th:Singularities}
		With the notations of \Cref{th:Convergence}, we have the following description:
		\begin{enumerate}
		\item\label{it:Sing_M} There exist a finite number of points $\vep_1\ldots,\vep_I\in \R^d$ such that $\overline{  \vPhi_{\infty}(\Sigma_{\infty}) } = \vPhi_{\infty}(\Sigma_{\infty})\cup \{\vep_1,\ldots,\vep_I\}$.
			
			\item\label{it:Sing_Sigma} For every $i\in\{1,\ldots,I\}$, there exists $r_i>0$ such that any connected component $\Cr$ of $\vPhi_{\infty}^{-1}(\B^d(\vep_i,r_i)\setminus \{\vep_i\})$ such that $\vep_i\in \overline{\vPhi_{\infty}(\Cr)}$ is homeomorphic to $\B^n(0,r_i)\setminus \{0\}$. Moreover, there exist coordinates on $\B^n(0,r_i)\setminus \{0\}$ such that $(g_{\vPhi_{\infty}})_{ij} \in W^{2,\left(\frac{n}{2},1\right)}(\B^n(0,r_i))$. As a consequence, we can complete the metric space $(\Cr,\dist_{g_{\vPhi_{\infty}}})$ into a complete metric space $(\overline{\Cr},\dist_{g_{\vPhi_{\infty}}})$, where $\overline{\Cr}$ is obtained by adding a point to $\Cr$.
		\end{enumerate}
	\end{theorem}
	
	\begin{remark}
		The points $\vep_1,\ldots,\vep_I$ do not lie in $\vPhi_{\infty}(\Sigma_{\infty})$.
	\end{remark}
	
	\subsubsection{Singularities on the image.}
	We start by proving \Cref{it:Sing_M}. This is a direct consequence of Claim \
    \ref{cl:Bad_Small}.
	
	\begin{claim}
		There exists a finite number of points $\vep_1\ldots,\vep_I\in \R^d$ such that $\overline{  \vPhi_{\infty}(\Sigma_{\infty}) } = \vPhi_{\infty}(\Sigma_{\infty})\cup \{\vep_1,\ldots,\vep_I\}$.
	\end{claim}
	\begin{remark}\label{rk:existence_pi}
		The proof also shows that if $\vPhi_{\infty}(\Sigma_{\infty})$ is closed in $\R^d$, then the points $\vep_i$ do not exist and the bad set $\Br^r_k$ defined in \eqref{eq:def_Bad} is empty for $r$ small enough and $k$ large enough. 
	\end{remark}
	\begin{proof}
		By Claim \ref{cl:Bad_Small}, there exists a number $I\in\N$ and $\vep_{1,r,k},\ldots,\vep_{I,r,k}\in \vPhi_k(\Br^r_k)$ such that for all $k\in\N$ and all $r\in(0,s_0)$, it holds 
		\begin{align*}
			\vPhi_k(\Br^r_k) \subset \B^d(\vep_{1,r,k},2r)\cup \cdots \cup \B^d(\vep_{I,r,k},2r).
		\end{align*}
		Up to reducing $I$, we assume that $I$ is minimal for the above condition. Given $r\in(0,s_0)$ and $k$ large enough (depending on $r$), we obtain
		\begin{align*}
			\vPhi_k(\Sigma) \subset \vPhi_k(\Gr^r_k)\cup \B^d(\vep_{1,r,k},2r)\cup \cdots \cup \B^d(\vep_{I,r,k},2r).
		\end{align*}
		Since $\vPhi_k(\Sigma)\subset \B^d(0,2)$ by \eqref{eq:Setting}, we obtain that each $(\vep_{i,r,k})_{k\in\N}$ converges to some $\vep_{i,r}\in \B^d(0,4)$ up to a subsequence. Letting $k\to +\infty$, we obtain 
		\begin{align*}
			\overline{  \vPhi_{\infty}(\Sigma_{\infty}) } \subset \vPhi_{\infty}(\Gr^r_{\infty})\cup \B^d(\vep_{1,r},2r)\cup \cdots \cup \B^d(\vep_{I,r},2r).
		\end{align*}
		Sending $r\to 0$, we have that each $(\vep_{i,r})_{r>0}$ converges to some $\vep_i\in \B^d(0,4)$ up to a subsequence. We obtain 
		\begin{align*}
			\overline{  \vPhi_{\infty}(\Sigma_{\infty}) } \subset \vPhi_{\infty}(\Sigma_{\infty}) \cup \left\{\vep_1,\ldots,\vep_I \right\}.
		\end{align*}
		Moreover, for any $i\in\{1,\ldots,I\}$ and any $r>0$, the ball $\B^d(\vep_{i,r,k},3r)$ intersects $\vPhi_k(\Gr^r_k)$. Sending $k\to +\infty$, we deduce that the ball $\overline{ \B^d(\vep_{i,r},3r)}$ intersects $\overline{  \vPhi_{\infty}(\Sigma_{\infty}) }$. Since this is true for any $r>0$, we deduce that $\vep_i\in \overline{  \vPhi_{\infty}(\Sigma_{\infty}) }$ by letting $r\to 0$.
	\end{proof}
	
	\subsubsection{Singularities of $\Sigma_{\infty}$.}
	
	In this section, we prove that $\vPhi^{-1}_{\infty}(\B^d(\vep_i,r))$ is a disjoint union of topological annuli for $r>0$ small enough. To do so, we adapt the arguments of \Cref{sec:One_chart}.
	
	\begin{lemma}\label{cl:Topology_singularities}
		There exists $R>0$ such that for any $i\in\{1,\ldots,I\}$, the connected components $\Cr\subset \vPhi_{\infty}^{-1}(\B^d(\vep_i,R)\setminus \{\vep_i\})$ such that $\vep_i\in \overline{\vPhi_{\infty}(\Cr)}$ are all homeomorphic to pointed ball $\B^n(0,R)\setminus \{0\}$. Under this identification, we have harmonic coordinates with $\vPhi_{\infty}\in \I_{\frac{n}{2}-1,2}(\B^n(0,R)\setminus \{0\};\R^d)$ and such that the coefficients $(g_{\vPhi_{\infty}})_{ij}$ lie in $W^{2,\left(\frac{n}{2},1\right)}(\B^n(0,R)\setminus \{0\})$. In particular, the induced distance is continuous up to the origin and $(\B^n(0,R),g_{\vPhi_{\infty}})$ is a well-defined Riemannian manifold of class $W^{3,\left(\frac{n}{2},1\right)}$.
	\end{lemma}

	\begin{proof} 
		\textit{Step 1: Approximation by smooth immersions.}\\
		Let $i\in\{1,\ldots,I\}$. By Claim \ref{cl:Bad_Small}, there exists $\vep_{i,r,k}\in \vPhi_k(\Br^r_k)$ and a sequence $(r_{\ell})_{\ell\in\N}$ of positive numbers converging to 0, such that the following limit exist:
		\begin{align}
			\lim_{k\to +\infty} \vep_{i,r_{\ell},k}=\vep_{i,r_{\ell}}, \qquad \text{ and }\qquad  \lim_{\ell\to +\infty} \vep_{i,r_{\ell}} = \vep_i.
		\end{align}
		Let $\eta\in (0,1)$. Given $R>0$ such that $R<\min_{i\neq j} |\vep_i-\vep_j|$, there exists $\ell_0=\ell_0(\eta,R)$ such that for any $\ell\geq \ell_0$, there exists $k_0=k_0(\eta,R,\ell)$ such that 
		\begin{align}\label{eq:vq_close}
			\forall \ell\geq \ell_0,\ \forall k\geq k_0,\qquad \left|\vep_{i,r_{\ell},k} - \vep_i \right|\leq \frac{ \eta R}{10}.
		\end{align}
		Thus, given $V \coloneq \sup_{k\in\N} \vol_{g_{\vPhi_k}}(\Sigma) \in \big(0,C(n,d)\, E\big)$, there exists $R_0=R_0(\eps,V)>0$ such that for any $R\in(0,R_0)$ and any $\eta\in(0,1)$, we have\footnote{We can cover each annulus $\B^d(\vep_{i,r_{\ell},k}, 4^{-j}R)\setminus \B^d(\vep_{i,r_{\ell},k}, 4^{-j-1} R)$ by a number $C(d)>0$ of balls of radius $4^{-j-2}R$. The reciprocal image of each of these balls by $\vPhi_k$ lie in $\Gr^{4^{j-2}R}_k$ by \eqref{eq:vq_close} and Claim \ref{cl:Bad_Small} if $4^{-j-1}\geq \eta$, i.e. if $j< \log_4(1/\eta)$. } 
		\begin{align*}
			\forall \ell\geq \ell_0,\ \forall k\geq k_0,\qquad  
			\begin{cases} 
				\displaystyle \sup_{ 0 \leq j< \log_4(1/\eta)} \Er\left( \vPhi_k; \vPhi_k^{-1}\left(\B^d(\vep_{i,r_{\ell},k}, 4^{-j}R)\setminus \B^d(\vep_{i,r_{\ell},k}, 4^{-j-1} R)\right) \right) \leq C(d)\, \eps_0, \\[5mm]
				\displaystyle \vol_{g_{\vPhi_k}}\left( \vPhi_k^{-1}\left(\B^d(\vep_{i,r_{\ell},k}, R) \right) \right)\leq V\leq C(n,d)\, E.
			\end{cases} 
		\end{align*}
        
		\textit{Step 2: Slicing.}\\
		We now modify slightly the proof of Proposition \ref{pr:est_good_slice}. Fix $0\leq j< \log_4(1/\eta)$. By a straightforward adaptation of Claim \ref{cl:est_above_slice}, we have
		\begin{align}\label{eq:adapted_slice}
			\begin{aligned} 
			& \fint_{\B^d\left(\vep_{i,r_{\ell},k},\frac{\eta R}{100}\right)} \Bigg( \int_{\vPhi_k^{-1}\left( \B^d(\vq,4^{-j}R)\setminus \B^d(\vq,\frac{2}{3}\, 4^{-j} R)\right) }
			\left|d|\vPhi_k-\vq|\right|_{g_{\vPhi_k}}\left[ \frac{|\vPhi_k-\vq|^{2n}\ |\vII_{\vPhi_k}|^n_{g_{\vPhi_k}}}{\left|\big(\vn_{\vPhi_k}\wedge (\vPhi_k-\vq)\big)\llcorner (\vPhi_k-\vq) \right|^n \eps_0}\right. \\[3mm]
			& \qquad \left. + \frac{|\vPhi_k-\vq|^n}{\left|(\vn_{\vPhi_k}\wedge(\vPhi_k-\vq))\llcorner (\vPhi_k-\vq) \right|^n} + \frac{1}{r^n} \right]\, d\vol_{g_{\vPhi_k}}\Bigg)\, d\vq 
			\leq C(n,d,E).
			\end{aligned} 
		\end{align}
		We find a point $\vq^{\ j}_{\ell,k},\vq^{\ j}_{\ell,k}\in \B^d(\vep_{i,r_{\ell},k},\frac{\eta\, R}{100})$ and a radius $\rho_{\ell,k}^j \in (\frac{2}{3}\, 4^{-j}R,4^{-j}R)$ such that the set $S_{\ell,k}^j\coloneqq \vPhi_k^{-1}\big( \B^d(\vq^{\ j}_{\ell,k} , \rho^j_{\ell,k}) \big)$ satisfy the conclusion of Proposition \ref{pr:est_good_slice}. That is to say, if $\vA$ is the second fundamental form of $\vPhi_k\colon S^j_{\ell,k}\to \s^{d-1}(\vq^{\ j}_{\ell,k} , \rho^j_{\ell,k})$, then it holds
		\begin{align}\label{eq:vol_bound}
			(\rho^j_{\ell,k})^{\frac{1}{n}} \left\| \vA \right\|_{L^n \left( S^j_{\ell,k},g_{\vPhi_k} \right)} + \frac{(\rho^j_{\ell,k})^{\frac{1}{n}}}{\eps_0} \left\| \vII_{\vPhi_k} \right\|_{L^n \left( S^j_{\ell,k},g_{\vPhi_k} \right)} + (\rho^j_{\ell,k})^{1-n}\, \vol_{g_{\vPhi_k}}(S^j_{\ell,k}) \leq C(n,d,E).
		\end{align}
		Hence, the proof of \Cref{pr:global_graph_slice} applies without any change and both slices $\vPhi_k(S^j_{\ell,k})$ are given by graphs that satisfy the conclusion of Proposition \ref{pr:global_graph_slice}. Up to a subsequence, we can assume that the following limits exist:
		\begin{align}\label{eq:limit}
			\begin{cases} 
				\displaystyle \lim_{\ell\to+\infty} \lim_{k\to +\infty} \vq^{\ 0}_{\ell,k} = \vq_1 \in \B^d\left(\vep_i,\frac{R}{100} \right), \\[5mm]
				\displaystyle \lim_{\ell\to+\infty} \lim_{k\to +\infty} \vq^{\ \lfloor \log_4(1/\eta) \rfloor }_{\ell,k} = \vq_{2,\eta} \in \B^d\left(\vep_i,\frac{\eta\, R}{100} \right), \\[5mm]
				\displaystyle \lim_{\ell\to+\infty} \lim_{k\to +\infty} \rho^{\ \lfloor \log_4(1/\eta)\rfloor }_{\ell,k} = \rho_{\eta} \in \left( \frac{\eta R}{4}, \eta R \right).
			\end{cases} 
		\end{align}

		\textit{Step 3: Topology.}\\
		Consider $\Cr_{\infty}$ a connected component of $\vPhi_{\infty}^{-1}(\B^d(\vep_i,R)\setminus \{\vep_i\})$ such that $\vep_i\in \overline{ \vPhi_{\infty}(\Cr_{\infty})}$. Since $\vep_i\notin \vPhi_{\infty}(\Sigma_{\infty})$, there exists $\eta_0\in(0,1)$ such that for any $\eta\in(0,\eta_0)$, the set $\Cr_{\infty}\cap \vPhi_{\infty}^{-1}(\B^d(\vq_1,R)\setminus \B^d(\vq_2,\eta R))$ is connected. Let $\eta\in(0,\eta_0)$. Since $\vPhi_k $ converges uniformly to $\vPhi_{\infty}$ on $\Cr_{\infty}\setminus \vPhi_{\infty}^{-1}(\B^d(\vep_i,\eta R/5))\subset \Gr^{\eta R/5}_{\infty}$ by Lemma \ref{lm:Conv_Good}, we obtain that if $\ell$ and $k$ are large enough, then the set $\Cr_{\infty}\cap \vPhi_{\infty}^{-1}(\B^d(\vep_i,  R)\setminus \B^d(\vep_i,\eta R))$ is homeomorphic to a connected component $\Cr_{\eta,k}\subset \Sigma$ of $\vPhi_k^{-1}\left( \B^d(\vq^{\ 0}_{1,\ell,k},R) \setminus \B^d(\vq^{\ \lfloor \log_4(1/\eta)\rfloor }_{2,\ell,k},\eta R) \right)$. \\

		We now apply the same strategy as in \Cref{sec:Topology}. Fix $0\leq j<\lfloor \log_4(1/\eta)\rfloor$ and denote $a^j_{\eta,\ell,k}$ the number of connected component of $\overline{\Cr_{\eta,k}} \cap \vPhi_k^{-1}\left(\s^{d-1}(\vq^{\ j}_{\ell,k},\rho^{j}_{\ell,k})\right)$. We extend each connected component of the set $\overline{\Cr_{\eta,k}}\cap \vPhi_k^{-1}\left(\s^{d-1}(\vq^{\ j}_{\ell,k},\rho^j_{\ell,k})\right)$ as in \Cref{sec:extension}. We also extend each connected component of $\overline{\Cr_{\eta,k}} \cap \vPhi_k^{-1}\left(\s^{d-1}(\vq^{\ j+1}_{\ell,k},\rho^{j+1}_{\ell,k})\right)$ in $\B^d(\vq^{\ j+1}_{\ell,k}, \rho^{j+1}_{\ell,k})$ using the extension in \cite[Remark 6.9]{MarRiv2025}. We obtain an immersion $\vPsi^{\ j}_{\eta,\ell,k}\in \I_{0,(n,2)}(\tilde{\Sigma}^{\ j}_{\eta,\ell,k};\R^d)$ with $\tilde{\Sigma}^{\ j}_{\eta,\ell,k}$ given by the union of $\Cr_{\eta,k}$ with a $a^j_{\eta,\ell,k}$ of copies of $\R^n\setminus \B^n$ and a $a^{j+1}_{\eta,\ell,k}$ copies of $\B^n$ with the following properties 
		\begin{enumerate}
			\item We have the upper bound $a^j_{\eta,\ell,k}+a^{j+1}_{\eta,\ell,k} \leq C(n,d,E)$ by \eqref{eq:vol_bound}.
			\item The immersion $\vPsi^j_{\eta,\ell,k}$ verifies
			\begin{align}\label{eq:small_ext}
				\left\| \vII_{\vPsi_{\eta,\ell,k}^{\ j}} \right\|_{L^{(n,2)}\left(\tilde{\Sigma}^{\ j}_{\eta,\ell,k},g_{\vPsi^{\ j}_{\eta,\ell,k}} \right)} \leq C(n,d,E)\, \eps_0.
			\end{align}
		\end{enumerate}
		We recover a setting similar to \Cref{sec:Topology} and the same proof leads to the conclusion that $a^j_{\eta,\ell,k}=1$ for $0\leq j<\log_4(1/\eta)$, only the number $a^{\lfloor \log_4(1/\eta)\rfloor}_{\eta,\ell,k}$ remains to be determined. Moreover, the manifold $\hat{\Sigma}_{\eta,k} \coloneq \Cr_{\eta,k} \cup \bigcup_{1\leq \beta\leq b^{\lfloor \log_4(1/\eta) \rfloor}_{\eta,\ell,k} } \B^n$ is diffeomorphic to the Euclidean ball $\B^n$. Applying the same proof as in \Cref{sec:Densities}, we obtain that $\theta_{\vPhi_k(\Cr_{\eta,k})}(\vx)=1$ for any $\vx\in \vPhi_k(\Cr_{\eta,k})$, but also that $\theta_{\vPhi_{\infty}(\Cr_{\infty})}(\vx)=1$ for any $\vx\in \vPhi_{\infty}(\Cr_{\infty})\cap \B^d(\vq_1,R)\setminus \B^d(\vq_2,\eta R)$. This is valid for any $\eta>0$ and thus, we obtain $\theta_{\vPhi_{\infty}(\Cr_{\infty})}(\vx)=1$ for any $\vx\in \vPhi_{\infty}(\Cr_{\infty})\setminus \{\vep_i\}$. Moreover, up to taking an adapted slice of $\vPhi_k(\Cr_{\eta,k})$ and passing to the limit $k\to \infty$ as in \eqref{eq:adapted_slice}, we can extend $\vPhi_{\infty}(\Cr_{\infty})$ outside of $\B^d(\vq_1,R)$ by the graph of a function with compact support and verifying \eqref{eq:small_ext}. Applying the same argument as in \Cref{sec:density}, we obtain that $\theta_{\vPhi(\Cr_{\infty})}(\vep_i)=1$ as well.\\
		
		We now show that $a^{\lfloor \log_4(1/\eta)\rfloor}_{\eta,\ell,k}=1$.\\
		Since $\vPhi_{\infty}\in \I_{\frac{n}{2}-1,2}(\Cr_{\infty};\R^d)$, the associated varifold $\V_{\vPhi_{\infty}}$ has bounded $L^n$-norm of its second fundamental form, bounded volume and density constant equal to 1. Hence, $\V_{\vPhi_{\infty}}$ has tangent cones at $\vep_i$, see for instance \cite[Section 42]{simon1983}, having density $\theta_{\vPhi_{\infty}}(\vep_i)=1$ at the origin. Moreover, these cones have vanishing second fundamental form, hence the results \cite[Corollary p.92]{hutchinson2} and \cite[Theorem 4.3]{aiex2024} imply that every tangent cone is a finite union of flat planes passing through the origin. Since $\theta_{\vPhi(\Cr_{\infty})}=1$, we obtain that there is a unique plane. Consequently, we obtain $\eta_1>0$ such that for $\eta\in(0,\eta_1)$, we have $a^{\lfloor \log_4(1/\eta) \rfloor}_{\eta,\ell,k}=1$ for $k$ large enough. Otherwise, we would have a sequence $\eta_j\to 0$ as $j\to +\infty$ such that $a^{\lfloor \log_4(1/\eta_j)\rfloor}_{\eta_j,\ell,k}\geq 2$ for every $k\geq k_j$ for some $k_j\in\N$. Letting $k\to +\infty$, we obtain that the slice $\vPhi(\Cr_{\infty})\cap \s^{d-1}(\vq_{2,\eta_j},\rho_{\eta_j} )$ (the quantities $\vq_{2,\eta_j}$ and $\rho_{\eta_j}$ are defined in \eqref{eq:limit}) is the image of at least two distinct connected components, both of them having constant multiplicity equal to $1$. Letting $j\to +\infty$ (up to a subsequence), we obtain a tangent cone passing through the origin with multiplicity at least 2, which is impossible. Therefore there exists $\eta_1>0$ such that for any $\eta\in(0,\eta_1)$ and any $k\geq k(\eta)$, it holds $a^{\lfloor \log_4(1/\eta)\rfloor}_{\eta,\ell,k}=1$. Therefore, $\Cr_{\eta,k}$ is diffeomorphic to $\B^n_2\setminus \B^n_1$.\\ 
        
		Since $\Cr_{\eta,k}\subset \Gr^{\eta R}_k$, we can apply \Cref{th:Atlas2} and proceed as in \Cref{lm:Conv_Good} to obtain that $(\Cr_{\eta,k})_{k\in\N}$ converges to a subset $\Cr_{\eta,\infty}\subset \Sigma_{\infty}$ in the pointed Gromov--Hausdorff topology. Hence $\Cr_{\eta,\infty}$ is homeomorphic to $\Cr_{\eta,k}$, that is to say, homeomorphic to an annulus, and $\Cr_{\eta,\infty}$ is a connected component of $\Cr_{\infty} \cap \vPhi_{\infty}^{-1}\left( \B^d(\vq_1,R) \setminus \B^d(\vq_{2,\eta},\rho_{\eta}) \right)$\footnote{Homeomorphisms map connected components onto connected components.}. Therefore, we obtain the homeomorphisms
		\begin{align*}
			\Cr_{\infty} \cap \vPhi_{\infty}^{-1}\left( \B^d(\vq_1,R) \setminus \B^d(\vq_2,\rho_{\eta}) \right) = \Cr_{\eta,\infty} \simeq \s^{n-1}\times (1,2).
		\end{align*}
		Given $0<\eta<\eta'<\eta_0$, we have $\Cr_{\eta',\infty}\subset \Cr_{\eta,\infty}$. Thus, we obtain a number $T\in(0,+\infty]$ such that we have the following homeomorphism
		\begin{align*}
			\Cr_{\infty} = \bigcup_{0<\eta<\eta_0} \Cr_{\eta,\infty} \simeq \s^{n-1}\times (0,T).
		\end{align*}

		\textit{Step 4: Construction of the coordinates.}\\
		Let $\eps>0$. Up to reducing $R>0$, we can assume that 
		\begin{align*}
			\Er\left(\vPhi_{\infty};\Cr_{\infty}\right) \leq \eps^n.
		\end{align*}
		Let $\eta\in(0,1)$. Since $\Cr_{\infty}$ is homeomorphic to an annulus, we can follow the arguments of Steps 2 and 3 in order to obtain an extension $\vPsi_{\eta}\in \I_{\frac{n}{2}-1,2}(\R^n;\R^d)$ (parametrizing a flat $n$-dimensional plane outside of $\B^d(\vep_1,4R)$ and another one in $\B^d(\vep_1,\eta R/2)$) to the weak immersion $\vPhi_{\infty}\in \I_{\frac{n}{2}-1,2}\left( \Cr_{\infty}\cap \vPhi_{\infty}^{-1}\big( \B^d(\vep_1,R/2)\setminus \B^d(\vep_1,2\eta R) \big) ;\R^d\right)$ with the estimate 
		\begin{align*}
			\left\| \vII_{\vPsi_{\eta}} \right\|_{L^{(n,2)}\left( \R^n, g_{\vPsi_{\eta}}\right)} \leq C(n,d)\, \sqrt{\eps}.
		\end{align*}
        By \cite[Theorem 6.5]{MarRiv2025}, we obtain an open set $\Or_{\eta}\subset \R^n$, a bi-Lipschitz homeomorphism $f_{\eta}\colon \Or_{\eta}\to \B^n(0,R)\setminus \B^n(0,\eta R)$ and a harmonic chart $\vp_{\eta}\colon \Or_{\eta}\to U_{\eta}\subset \Cr_{\infty}$ such that 
		\begin{itemize}
			\item If we decompose the boundary $\dr \Or_{\eta} = S_{\eta}^1\cup S^2_{\eta}$, with $f_{\eta}\colon S_{\eta}^1\to \dr \B^n(0,R)$ and $f_{\eta}\colon S^2_{\eta}\to \dr \B^n(0,\eta R)$, then it holds
			\begin{align*}
				& \|\g f_{\eta} \|_{L^{\infty}(\dr \Or_{\eta})} + \|\g f^{-1}_{\eta} \|_{L^{\infty}(\dr \left[\B^n(0,R)\setminus \B^n(0,\eta R)\right])} + R^{\frac{1}{n}} \left[ \|\g^2 f_{\eta} \|_{L^n(S^1_{\eta})} + \|\g^2 f_{\eta}^{-1}\|_{L^n(\dr B^n(0,R))}\right] \\[2mm]
				&  + (\eta R)^{\frac{1}{n}} \left[ \|\g^2 f_{\eta}^{-1}\|_{L^n(\dr \B^n(0,\eta R))} + \|\g^2 f_{\eta}\|_{L^n(S^2_{\eta})} \right] \leq C(n,d).
			\end{align*}
			
			\item The induced metric $g_{\vPhi_{\infty}\circ\vp_{\eta}}$ verifies for any $1\leq a,b,c,d\leq n$
			\begin{align*}
				& \left\| \left(g_{\vPhi_{\infty}\circ\vp_{\eta}}\right)_{ab} - \delta_{ab} \right\|_{L^{\infty}(\Or_{\eta})} + \left\| \dr_c \left(g_{\vPhi_{\infty}\circ\vp_{\eta}}\right)_{ab} \right\|_{L^{(n,1)}(\Or_{\eta})}  +  \left\| \dr^2_{cd} \left(g_{\vPhi_{\infty}\circ\vp_{\eta}}\right)_{ab} \right\|_{L^{\left(\frac{n}{2},1\right)}(\Or_{\eta})}  \leq C(n,d)\, \eps.
			\end{align*}
		\end{itemize}
		
		By following the arguments of \Cref{sec:Oalpha}, we obtain that the open sets $\Or_{\eta}$ converge in the $C^{1,\frac{1}{2n}}$-topology to a bounded open set $\Or_0$ whose boundary is $C^{1,\frac{1}{n}}$-diffeomorphic to $\s^{n-1}(0,R)\cup \{0\}$.\\

		Moreover, the transition maps $\vp_{\eta'}^{-1}\circ\vp_{\eta}$ are uniformly bounded in $W^{3,\left(\frac{n}{2},1\right)}$ and the maps $\vp_{\eta}$ are uniformly continuous as explained in Step 2 of the proof of Lemma \ref{lm:Conv_Good}. Hence, the transition maps $\vp_{\eta'}^{-1}\circ\vp_{\eta}$ and $\vp_{\eta}$ converge as $\eta\to 0$ up to a subsequence to some maps $\vp_{0}\colon \Or_0\to U_0\subset \Sigma_{\infty}$ and $\psi_{\eta',0}\in W^{3,\left(\frac{n}{2},1\right)}(\vp_0^{-1}(U_{\eta'});\vp_{\eta'}^{-1}(U_0))$ verifying $\vp_{\eta'}\circ\psi_{\eta',0} = \vp_{0}$. By following the arguments of Step 2 of the proof of Lemma \ref{lm:Conv_Good}, we obtain that the weak immersion $\vPhi_{\infty}\circ\vp_{0}$ is injective. We obtain a harmonic chart $\vp_{0}\colon \Or_{0} \to U_{0}\subset \Sigma_{\infty}$ such that 
		\begin{align*}
			& \left\| \left(g_{\vPhi_{\infty}\circ\vp_{0}}\right)_{ab} - \delta_{ab} \right\|_{L^{\infty}(\Or_0)} + \left\| \dr_c \left(g_{\vPhi_{\infty}\circ\vp_0}\right)_{ab} \right\|_{L^{(n,1)}(\Or_0)}  +  \left\| \dr^2_{cd} \left(g_{\vPhi_{\infty}\circ\vp_0}\right)_{ab} \right\|_{L^{\left(\frac{n}{2},1\right)}(\Or_0)}  \leq C(n,d)\, \eps.
		\end{align*} 
		Since the coefficients $(g_{\vPhi_{\infty}\circ\vp_0})_{ij}$ are continuous up to $\overline{\Or_{0,\infty}}$, we deduce that they have a well-defined limit at the singularity. Moreover, we have that any transition map between $\vp_0$ and another harmonic chart obtained by \Cref{th:Atlas2} lies in $W^{3,\left(\frac{n}{2},1\right)}$. Hence, the harmonic chart $\vp_0$ provides an extension of $\Cr_{\infty}$ to $\overline{\Cr_{\infty}}$ into a $W^{3,\left(\frac{n}{2},1\right)}$-Riemannian manifold, by adding one point to $\Cr_{\infty}$.
	\end{proof}

	\bibliography{sobolovBib.bib}
	\bibliographystyle{plain}
\end{document}